\pdfoutput=1
\RequirePackage{ifpdf}
\ifpdf 
\documentclass[pdftex]{sigma}
\else
\documentclass{sigma}
\fi

\usepackage{tkz-graph}
\usetikzlibrary{calc}
\usetikzlibrary{decorations.markings}
\usepackage{pgf}
\usepackage{tikz}
\usetikzlibrary{shapes,arrows,automata,decorations.pathreplacing,angles,quotes}
\usepackage{enumerate}
\usepackage{enumitem}

\numberwithin{equation}{section}

\newcommand{\minus}{\scalebox{0.6}[1.0]{$-$}}

\newtheorem{thm}{Theorem}[section]
\newtheorem{lem}[thm]{Lemma}
\newtheorem{cor}[thm]{Corollary}
\newtheorem{pro}[thm]{Proposition}
\newtheorem*{prob*}{Direct Monodromy Problem}
\newtheorem*{probinv*}{Inverse Monodromy Problem}
\newtheorem*{probinvred*}{Reduced Inverse Monodromy Problem}
\newtheorem*{lem*}{Lemma}
\newtheorem*{cor*}{Corollary}
\newtheorem*{con*}{Conjecture}
\newtheorem*{thm*}{Theorem}
\newtheorem*{claim}{Claim}

{\theoremstyle{definition}
\newtheorem{Def}[thm]{Definition}
\newtheorem{rem}[thm]{Remark}
\newtheorem{rem*}[thm]{Remark}
\newtheorem*{def*}{Definition}

}

\newcommand{\e}{\varepsilon}
\newcommand{\bb}[1]{\mathbb{ #1 }}

\newcommand{\om}{\omega}
\newcommand{\la}{\lambda}
\newcommand{\al}{\alpha}

\newcommand{\beq}{\begin{equation}}
\newcommand{\eeq}{\end{equation}}

\begin{document}

\allowdisplaybreaks

\newcommand{\arXivNumber}{1707.05222}

\renewcommand{\thefootnote}{}

\renewcommand{\PaperNumber}{002}

\FirstPageHeading

\ShortArticleName{Poles of Painlev\'e IV Rationals and their Distribution}

\ArticleName{Poles of Painlev\'e IV Rationals and their Distribution\footnote{This paper is a~contribution to the Special Issue on Painlev\'e Equations and Applications in Memory of Andrei Kapaev. The full collection is available at \href{https://www.emis.de/journals/SIGMA/Kapaev.html}{https://www.emis.de/journals/SIGMA/Kapaev.html}}}

\Author{Davide MASOERO~$^\dag$ and Pieter ROFFELSEN~$^\ddag$}

\AuthorNameForHeading{D.~Masoero and P.~Rof\/felsen}

\Address{$^\dag$~Grupo de F\'isica Matem\'atica e Departamento de Matem\'atica da Universidade de Lisboa,\\
\hphantom{$^\dag$}~Campo Grande Edif\'icio C6, 1749-016 Lisboa, Portugal}
\EmailD{\href{mailto:dmasoero@gmail.com}{dmasoero@gmail.com}}
\URLaddressD{\url{http://gfm.cii.fc.ul.pt/people/dmasoero/}}

\Address{$^\ddag$~School of Mathematics and Statistics F07, The University of Sydney, NSW 2006, Australia}
\EmailD{\href{mailto:p.roffelsen@maths.usyd.edu.au}{p.roffelsen@maths.usyd.edu.au}}

\ArticleDates{Received July 20, 2017, in f\/inal form December 18, 2017; Published online January 06, 2018}

\Abstract{We study the distribution of singularities (poles and zeros) of rational solutions of the Painlev\'e IV equation by means of the isomonodromic deformation method. Singularities are expressed in terms of the roots of generalised Hermite~$H_{m,n}$ and generalised Okamoto~$Q_{m,n}$ polynomials. We show that roots of generalised Hermite and Okamoto polynomials are described by an inverse monodromy problem for an anharmonic oscillator of degree two. As a consequence they turn out to be classif\/ied by the monodromy representation of a class of meromorphic functions with a f\/inite number of singularities introduced by Nevanlinna. We compute the asymptotic distribution of roots of the generalised Hermite polynomials in the asymptotic regime when~$m$ is large and~$n$ f\/ixed.}


\Classification{34M55; 34M56; 34M60; 33C15; 30C15}

\begin{flushright}
 \begin{minipage}{95mm}\it
	To the memory of Andrei A.~Kapaev, a master of Pain\-lev\'e equations, untiring author, referee and reviewer.
\end{minipage}
\end{flushright}

\renewcommand{\thefootnote}{\arabic{footnote}}
\setcounter{footnote}{0}

\section{Introduction}

In this paper we address, by means of the isomonodromic deformation method \cite{fokas,its86,jimbomiwaII1981}, the distribution of movable singularities (which are zeros and poles) of rational solutions of the fourth Painlev\'e equation, also called Painlev\'e~IV and denoted by ${\rm P}_{\rm IV}$, which is the following second order dif\/ferential equation,
\begin{gather}\label{eq:piv}
\om_{zz}=\displaystyle\frac{1}{2\om}\om_z^2+ \frac{3}{2} \om^3 + 4 z\om^2+2\big(z^2+1-2\theta_\infty\big)\om-\frac{8\theta_0^2}{\om}
,\qquad \theta:=(\theta_0,\theta_\infty) \in \mathbb{C}^2.
\end{gather}
We sometimes write ${\rm P}_{\rm IV}(\theta)$ to stipulate the particular parameter values in consideration.

${\rm P}_{\rm IV}$ is one of the famous six Painlev\'e equations, all of which have had an enormous impact in several branches of science,
including mathematical physics, algebraic
geometry, applied mathematics, f\/luid dynamics and statistical mechanics, see, e.g.,
\cite{deifttalk,dubrovintopological,fokasquantum} and references therein.

Special solutions, such as rational solutions, and the distribution of movable singularities have proven to be
particularly important in applications and were thoroughly
studied by means of ad-hoc methods which are unavailable for regular values or generic solutions; see, e.g.,
\cite{bertolavorob,costin12,dubrovinmazzocco,piwkb, noumi} and references therein. In the case under
consideration, we will show that singularities of rational solutions of Painlev\'e IV are characterised by the
monodromy representation of three particular classes of meromorphic
functions (introduced by Nevanlinna~\cite{nevanlinna32}) with a f\/inite number of critical points and transcendental singularities.

One of the most striking features of solutions of Painlev\'e equations is that their value distributions are often
observed to describe some approximate lattice structure, as was f\/irst discovered by Boutroux~\cite{boutroux} for solutions of Painlev\'e~I
and~II.
This is also the case for singularities of Painlev\'e~IV rationals. Indeed one of the main
inspirations of our work is the highly regular pattern of their distribution, which was
observed by Clarkson \cite{clarkson2003piv} (see~\cite{reegerfornberg2014} for more general solutions),
and which has so far eluded any rigorous clarif\/ication. By computing
the distribution of singularities in a particular asymptotic regime we furnish here the f\/irst rigorous, albeit partial, explanation\footnote{We notice that the same question for Painlev\'e~II rationals has recently been settled by many authors with a~wealth of dif\/ferent methods \cite{bertolavorob,buckinghamnoncr,buckinghamcr,millerrational,roffelsenirrational,roffelsennumber}.}.

For the sake of clarity of our exposition, before introducing our main results together with an outline of the article,
we brief\/ly review some well-known facts from the theory of
movable singularities and rational solutions of Painlev\'e~IV.

\subsection{Zeros and poles of solutions}
The Painlev\'e property implies that any local solution of ${\rm P}_{\rm IV}$ has a unique meromorphic continuation to the entire complex plane~\cite{steinmetz}. As a consequence the solution space is the set of meromorphic functions on~$\mathbb{C}$ that satisfy \eqref{eq:piv}, which we denote by
\begin{gather}\label{eq:defiwtheta}
\mathcal{W}_\theta=\{{\rm P}_\text{IV}(\theta) \ \text{transcendents}\}.
\end{gather}
Upon f\/ixing an $a\in \mathbb{C}$, any Painlev\'e IV transcendent (i.e., solution) enjoys a Laurent expansion at this point. In particular the generic Laurent expansion takes the form
\begin{gather}\label{eq:laurentgeneric}
\om(z)=b+c(z-a)+\om_2(z-a)^2+\om_3(z-a)^3+\mathcal{O}\big((z-a)^4\big),
\end{gather}
for $b\in\mathbb{C}^*$ and $c\in\mathbb{C}$, where the higher order coef\/f\/icients are of the form $\om_n=\frac{1}{b}p_n$, with~$p_n$ polynomial in $a$, $b$, $c$, $\theta$ for $n\geq 2$. However this expansion breaks down when $\omega(a)\in\{0,\infty\}$, i.e., $\omega$ has a zero or pole at $z=a$, in which case the Laurent expansions take the respective forms
\begin{gather}\label{eq:laurentzero}
\om(z)=4\epsilon\theta_0(z-a)+b(z-a)^2+\om_3(z-a)^3+\mathcal{O}\big((z-a)^4\big),\qquad \epsilon=\pm 1,
\end{gather}
where we refer to the value of $\epsilon$ as the sign of the zero, or
\begin{gather}\label{eq:laurentpole}
\om(z)=\frac{\epsilon}{z-a}-a+\om_1(z-a)+b(z-a)^2+\om_3(z-a)^3+\mathcal{O}\big((z-a)^4\big),\qquad \epsilon=\pm 1,
\end{gather}
where $\om_1=\tfrac{1}{3}\epsilon(a^2-2+4(\theta_\infty-\epsilon))$, in both cases $b\in\mathbb{C}$, and all higher order coef\/f\/icients have polynomial dependence on $a$, $b$, $\theta$; conversely, for any choice of parameters the Laurent series \eqref{eq:laurentgeneric}--\eqref{eq:laurentpole}
converge locally to a solution of ${\rm P}_{\rm IV}$. Indeed, one can f\/ind a positive constant~$C_{a,b}$ such that $|\omega_n|\leq C_{a,b}^n$, for $n\geq 3$, and thus the formal power-series actually has a non-zero radius of convergence, see, e.g.,~\cite{filipuk09, gromak}.

From an algebro-geometric point of view \cite{joshimilena, okamoto1979}, it is clear that zeros and poles of ${\rm P}_{\rm IV}$ transcendents play a fundamentally distinguished role; they are the movable singularities of solutions.

\subsection{Rational solutions}
Firstly, let us remark that ${\rm P}_{\rm IV}$ enjoys various B\"acklund transformations, which relate solutions with dif\/ferent parameter values. In particular we have transformations $\mathcal{R}_1$--$\mathcal{R}_4$, see Appendix~\ref{section:backlund} for their def\/initions, which allow us to relate the solution spaces $\mathcal{W}_\theta$ and $\mathcal{W}_{\theta'}$ whenever $\theta-\theta'\in\mathbb{Z}^2\cup (\mathbb{Z}+\tfrac{1}{2})^2$.

Painlev\'e IV admits a rational solution if and only if the parameters satisfy either
\begin{gather}
\theta_\infty\in\tfrac{1}{2}\mathbb{Z}, \qquad \theta_0-\theta_\infty \in\mathbb{Z},\label{eq:parameterhermite}
\end{gather}
or
\begin{gather}
\theta_\infty \in\tfrac{1}{2}\mathbb{Z}, \qquad
\theta_0-\theta_\infty \in\mathbb{Z}\pm \tfrac{2}{3}. \label{eq:parameterokamoto}
\end{gather}
Furthermore for any such parameter values the associated rational solution is unique \cite{gromakrussian,lukasevic,murata}.
Using the $\tau$-function formalism, Noumi and Yamada~\cite{noumiyamada} expressed these rational solutions conveniently in terms of
generalised Hermite $H_{m,n}(z)$ and generalised Okamoto polyno\-mials~$Q_{m,n}(z)$, see Appendix~\ref{appendix:generalisedhermite} for their precise def\/inition.

Firstly, the parameter cases \eqref{eq:parameterhermite}, up to equivalence $\theta_0\leftrightarrow -\theta_0$, are given by
\begin{subequations}
	\label{eq:hermiterationals}
	\begin{alignat}{4}
& \omega_{m,n}^{({\rm I})}=\frac{{\rm d}}{{\rm d}z}\log\frac{H_{m+1,n}}{H_{m,n}}, \quad && \theta_0=\tfrac{1}{2}n, \quad && \theta_\infty =m+\tfrac{1}{2}n+1, & \label{eq:hermitepar1}\\
& \omega_{m,n}^{({\rm II})} =\frac{{\rm d}}{{\rm d}z}\log\frac{H_{m,n}}{H_{m,n+1}},\quad && \theta_0=\tfrac{1}{2}m, \quad && \theta_\infty =-\tfrac{1}{2}m-n,& \label{eq:hermitepar2}\\
& \omega_{m,n}^{({\rm III})}=-2z+\frac{{\rm d}}{{\rm d}z}\log\frac{H_{m,n+1}}{H_{m+1,n}}, \quad \ \ && 	\theta_0=\tfrac{1}{2}(m+n+1),\quad \ & & \theta_\infty =\tfrac{1}{2}(n-m+1),& \label{eq:hermitepar3}
	\end{alignat}
\end{subequations}
where $m,n\in\mathbb{N}$, which we refer to as the Hermite I, II and III families respectively. Some particularly simple members of these respective families are
\begin{subequations}\label{eq:hermitesimple}
	\begin{alignat}{4}
& \omega_{0,1}^{({\rm I})} = \frac{1}{z}, \qquad && \theta_0 =\tfrac{1}{2}, \qquad && \theta_\infty=\tfrac{3}{2},&\label{eq:hermitesimple1}\\
& \omega_{1,0}^{({\rm II})}=- \frac{1}{z}, \qquad && \theta_0=\tfrac{1}{2}, \qquad && \theta_\infty =-\tfrac{1}{2},& \label{eq:hermitesimple2}\\
& \omega_{0,0}^{({\rm III})} =-2z, \qquad && \theta_0=\tfrac{1}{2},\qquad & & \theta_\infty=\tfrac{1}{2},&\label{eq:hermitesimple3}
	\end{alignat}
\end{subequations}
and the other ones can be obtained via application of the B\"acklund transformations $\mathcal{R}_1$--$\mathcal{R}_4$, as depicted in Table~\ref{table:bthermite}.

{\renewcommand{\arraystretch}{1.4}
	\begin{table}[t]
		\centering
		\begin{tabular}{l || l l l | l }
			& $\omega_{m,n}^{({\rm I})}$ & $\omega_{m,n}^{({\rm II})}$ & $\omega_{m,n}^{({\rm III})}$ & $\widetilde{\omega}_{m,n}$\\
			\hline\hline
			$\mathcal{R}_1$ & $\omega_{m+1,n-1}^{({\rm I})}$ & $\omega_{m-1,n}^{({\rm II})}$ & $\omega_{m-1,n}^{({\rm III})}$ & $\widetilde{\omega}_{m+1,n-1}$\\
			$\mathcal{R}_2$ & $\omega_{m-1,n+1}^{({\rm I})}$ & $\omega_{m+1,n}^{({\rm II})}$ & $\omega_{m+1,n}^{({\rm III})}$ & $\widetilde{\omega}_{m-1,n+1}$\\
			$\mathcal{R}_3$ & $\omega_{m,n+1}^{({\rm I})}$ & $\omega_{m+1,n-1}^{({\rm II})}$ & $\omega_{m,n+1}^{({\rm III})}$ & $\widetilde{\omega}_{m,n+1}$\\
			$\mathcal{R}_4$ & $\omega_{m,n-1}^{({\rm I})}$ & $\omega_{m-1,n+1}^{({\rm II})}$ & $\omega_{m,n-1}^{({\rm III})}$ & $\widetilde{\omega}_{m,n-1}$
		\end{tabular}
		\caption{Action of B\"acklund transformations on rational solutions of ${\rm P}_{\rm IV}$.}		\label{table:bthermite}
	\end{table}}

Secondly, the parameter cases \eqref{eq:parameterokamoto}, up to equivalence $\theta_0\leftrightarrow -\theta_0$, are given by
\begin{gather}
\widetilde{\omega}_{m,n}=-\frac{2}{3}z+\frac{{\rm d}}{{\rm d}z}\log\frac{Q_{m+1,n}}{Q_{m,n}}, \qquad \theta_0=-\tfrac{1}{6}+\tfrac{1}{2}n,
\qquad \theta_\infty=\tfrac{1}{2}(2m+n+1),	\label{eq:okamotorationals}
\end{gather}
where $m,n\in\mathbb{Z}$, which we refer to as the Okamoto family. A particularly simple member of the Okamoto family, is given by
\begin{gather*}
\widetilde{\omega}_{0,0}=-\tfrac{2}{3}z, \qquad \theta_0=-\tfrac{1}{6}, \qquad \theta_\infty=\tfrac{1}{2}.
\end{gather*}
and again the other ones can be obtained via application of the B\"acklund transformations \smash{$\mathcal{R}_1$--$\mathcal{R}_4$}, as Table~\ref{table:bthermite} shows.

Remarkably, all zeros and poles of rational solutions can be expressed as roots of the gene\-ralised Hermite and Okamoto polynomials,
see Table~\ref{table:zerosandpoles}. Therefore the study of the distribution of movable singularities of~${\rm P}_\text{IV}$ rationals
is reduced to that of zeros of the generalised Hermite and Okamoto polynomials.

{\renewcommand{\arraystretch}{1.4}
	\begin{table}[t]
		\centering
		\begin{tabular}{l || l | l | l | l }
			& zeros ($\epsilon=+1$) & zeros ($\epsilon=-1$) & poles ($\epsilon=+1$) & poles ($\epsilon=-1$)\\
			\hline\hline
			$\omega_{m,n}^{({\rm I})}$ & $H_{m+1,n-1}$ & $H_{m,n+1}$ & $H_{m+1,n}$ & $H_{m,n}$\\
			$\omega_{m,n}^{({\rm II})}$ & $H_{m-1,n+1}$ & $H_{m+1,n}$ & $H_{m,n}$ & $H_{m,n+1}$\\
			$\omega_{m,n}^{({\rm III})}$ & $H_{m,n}$ & $H_{m+1,n+1}$ & $H_{m,n+1}$ & $H_{m+1,n}$\\
			\hline
			$\widetilde{\omega}_{m,n}$ & $Q_{m+1,n-1}$ & $Q_{m,n+1}$ & $Q_{m+1,n}$ & $Q_{m,n}$\\
		\end{tabular}
\caption{The zeros and poles of rational solutions in terms of generalised Hermite and Okamoto polynomials. As an example, poles with residue $\e=+1$ of $\omega_{m,n}^{({\rm III})}$ coincide with the roots of $H_{m,n+1}$. The table is computed using equations \eqref{eq:hermiterationals} and \eqref{eq:okamotorationals} and	the action of the B\"acklund transformations~\eqref{eq:backluend1}.}\label{table:zerosandpoles}
	\end{table}
}

\subsection{Outline and main results}
In Section \ref{section:polesofrationals} we recall the isomonodromic deformation interpretation of Painlev\'e~IV, via the Garnier--Jimbo--Miwa Lax pair.
The Riemann--Hilbert correspondence associates bijectively any solution of Painlev\'e~IV to a unique monodromy datum of the Garnier--Jimbo--Miwa linear system; given a~point~$z$ in the complex plane and a solution $\omega$ of ${\rm P}_{\rm IV}$, the inverse monodromy problem furnishes the value $\omega(z)$ unless $z$ is a zero or a pole, in which cases the inverse monodromy problem for the linear system does not have any solution.

However, following \cite{piwkb}, we show that the inverse monodromy problem can be def\/ined in case of zeros and poles, and in fact it simplif\/ies to that of the anharmonic oscillator
\begin{gather*}
\psi''(\lambda) = V(\lambda;a,b,\theta)\psi(\lambda),\\ \nonumber
V(\lambda;a,b,\theta)=\lambda^2+2a\lambda+a^2+2(1-\theta_\infty)
-\big[b+\big(2\theta_\infty-\tfrac{1}{2}\big)a\big]\lambda^{-1}+\big(\theta_0^2-\tfrac{1}{4}\big)\lambda^{-2} .
\end{gather*}
The main result of this section is that the aforementioned simplif\/ication allows us the characterise zeros and poles of rational solutions exactly as
the solutions of an inverse monodromy problem concerning the anharmonic oscillator in question; see Theorem~\ref{thm:hermiteanharmonic}.

According to the beautiful theory developed by Nevanlinna and his school \cite{elfving1934, nevanlinna32}, anharmonic oscillators (in case all singularities in the complex plane are Fuchsian and apparent) naturally def\/ine Riemann surfaces which are inf\/initely-sheeted coverings of the Riemann sphere uniformised by meromorphic functions. In Section \ref{section:nevanlinna} we def\/ine three families of such Nevanlinna functions and show how they classify the zeros and poles of rational solutions. This characterisation is rather powerful, as an easy corollary gives us the solution to a previously open problem, namely
to determine exactly the number of real roots of the generalised Hermite polynomials; see Corollary \ref{cor:numberrealroots}.

Finally, in Section \ref{section:asymptotic} we study the asymptotic distribution of zeros of the generalised Hermite polynomials~$H_{m,n}$, and hence zeros and poles of corresponding rational solutions in the asymptotic regime $m \to \infty$ and~$n$ bounded. In order to state precisely our main result we need to introduce some new notation and functions. First of all, as we will f\/ind that the roots grow like $(2m+n)^{\frac12}$, it is convenient in our analysis to work with the new unknown $\al$ and `big parameter'~$E$, def\/ined by
\begin{gather}\label{eq:alE}
\al=E^{-\frac12}a,\qquad E=2m+n.
\end{gather}
The $m \cdot n $ roots of $H_{m,n}(z)$ turn out to be organised in $n$ approximately horizontal lines. We parametrise these lines by the set
\begin{gather*}
J_n:=\lbrace -n+1,-n+3, \dots , n-3, n-1 \rbrace ,
\end{gather*}
and introduce the real functions
\begin{alignat*}{4}
& f\colon \ && [-1,1] \to [-\tfrac{\pi}{4},\tfrac{\pi}{4}],\qquad && f(x)=\tfrac{1}{2}\big(x(1-x^2)^{\frac12}-\cos^{-1}(x)\big)+\frac{\pi}{4},& \\
& g_{j,E} \colon \ && (-1,1) \to \bb{R},\qquad && g_{j,E}(x) =
\frac{2 j \log \big(2E^{\frac12}\big(1-x^2\big)^{\frac34}\big)-\log F_{n,j}}{2 E (1-x^2)^{\frac12}},
\end{alignat*}
where $j\in J_n$ and the constant $F_{n,j}$ is def\/ined by
\begin{gather}\label{eq:Fnj}
F_{n,j}=\frac{\Gamma\big[\tfrac{1}{2}(1+n+j)\big]}{\Gamma\big[\tfrac{1}{2}(1+n-j)\big]}.
\end{gather}
Notice that $f$ is strictly monotone and therefore globally invertible. Finally, we def\/ine our approximate (rescaled) roots~$\al$.
\begin{Def}\label{def:aljk} For every $k$, with $|k|\leq \frac E4$ and $k$ integer if $m$ is odd or half-integer if $m$ is even, and every $j \in J_n$, we def\/ine the approximate root $\al_{j,k}=\al_R +i\al_I$, with $\al_R,\al_I \in \bb{R}$, as the unique complex number determined by the following relations
\begin{subequations}\label{eqrealandimaginaryI}
	\begin{gather}\label{eq:realI}
	f(\al_R)=\frac{\pi k}{E},\qquad \al_R \in(-1,1),\\ \label{eq:imaginaryI}
	\al_I= g_{j,E}(\al_R).
	\end{gather}
\end{subequations}
\end{Def}

The points $\al_{j,k}$ are the dominant term in the asymptotic expansion of the (rescaled) roots belonging to the 'bulk', as it is pictorially illustrated in Fig.~\ref{fig:approximate} below.

\begin{figure}[ht]\centering
\includegraphics[width=7cm]{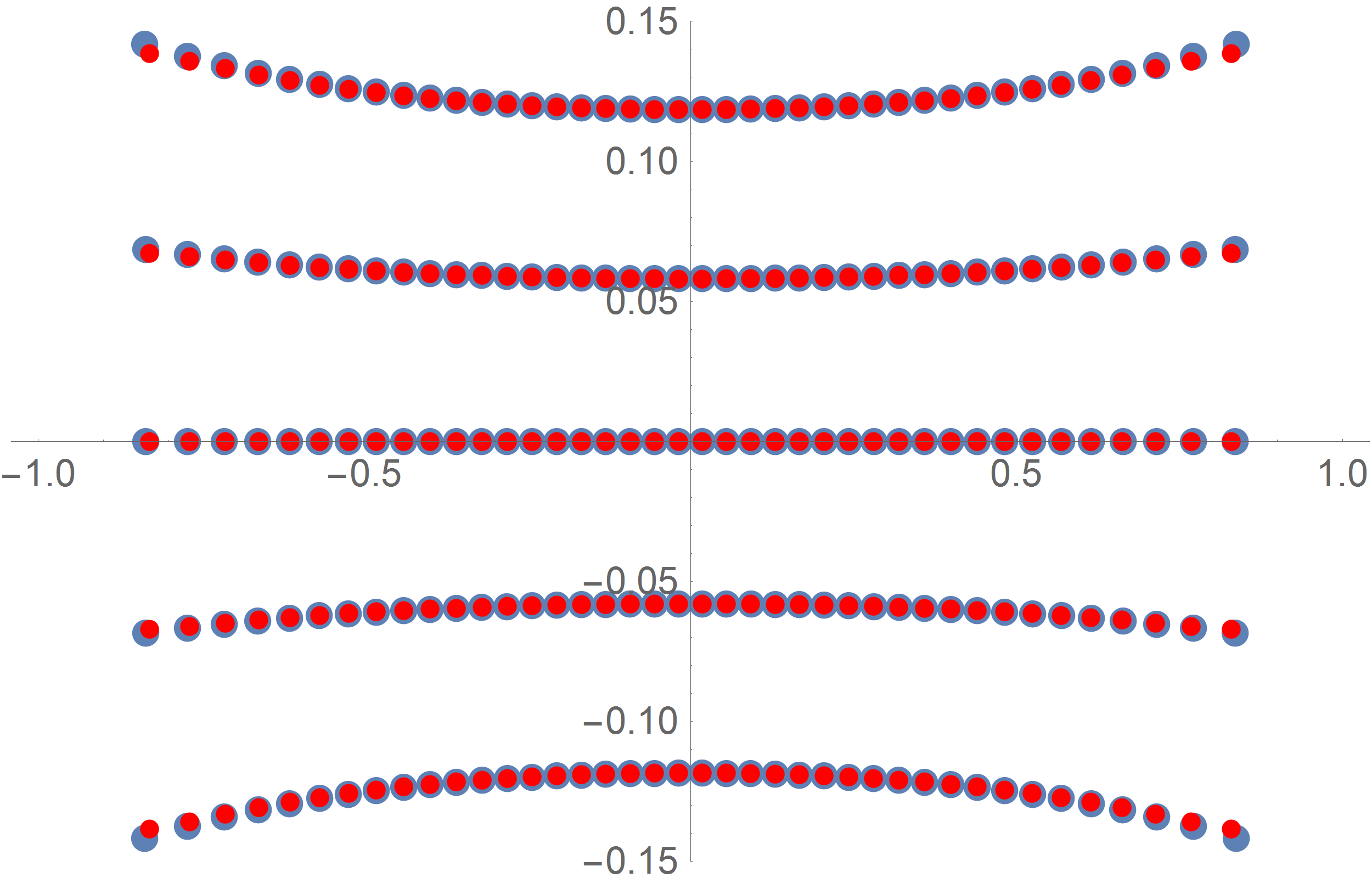} \qquad \includegraphics[width=7cm]{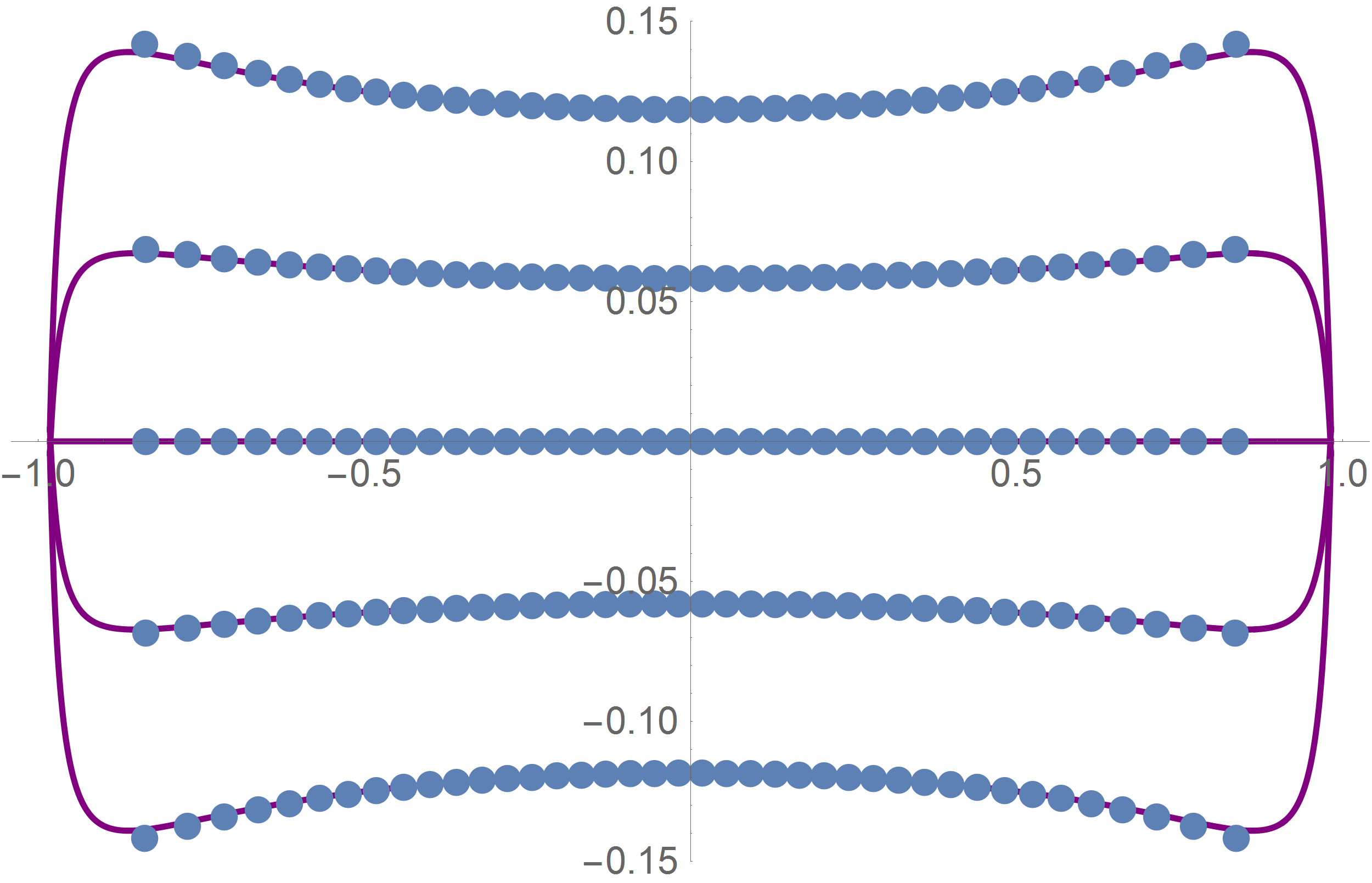}
\caption{Rescaled roots of $H_{40,5}$. In blue numerically exact location, in red the asymptotic approximation $\al_{j,k}$. In purple the curves $\Im \al= g_{j,E}(\Re \al)$, $E=85$, $j \in J_5$.} \label{fig:approximate}
\end{figure}

We consider two cases of 'bulk' behaviour. In the f\/irst case we compute the error in the approximation when we limit the real part of the roots to a~f\/ixed closed subinterval of $(-1,1)$.
\begin{thm}\label{thm:asymptoticzeros}
Fix $0<\sigma<\frac14$. There exist a constant $E_0$ and a constant $C_{\sigma}$ such that, if $E\geq E_0$, then for all $|k|\leq \sigma E$, the polynomial
	$H_{m,n}(z)$ has one and only one zero in each disc of the form $\big|z-E^{\frac12}\al_{j,k}\big|\leq C_{\sigma} E^{-\frac56}$.
\end{thm}

In the second case, we suppose that the real part of the (rescaled) roots belong to a closed subinterval of $(-1,1)$, which is growing to the full interval at some restricted rate, as $E$ becomes large. An interval of the kind $\big[{-}1+c E^{-1+\delta},1-cE^{-1+\delta}\big]$ for some $c>0$ and $\delta>\frac13$.
\begin{thm}\label{thm:asymptoticzerosedge}
Fix $\frac13<\delta\leq1$ and $s >0$. There exist a constant $E_0$ and a constant $C_{s}$ such that, if $E\geq E_0$, then for all $|k|\leq \big( \frac{1}{4}-s E^{-\frac{3}{2}(1-\delta)} \big) E$, the polynomial $H_{m,n}(z)$ has one and only one zero in each disc of the form $|z-E^{\frac{1}{2}}\al_{j,k}|\leq C_s E^{-\delta+\frac16}$.
\end{thm}

In Fig.~\ref{fig:thmsnumeric} the error estimates in Theorems~\ref{thm:asymptoticzeros} and~\ref{thm:asymptoticzerosedge} are visualised. In particular they showcase that the bounds are optimal.

\begin{figure}[ht]	\centering
\begin{tabular}{ c c c c}
 \includegraphics[width=3cm]{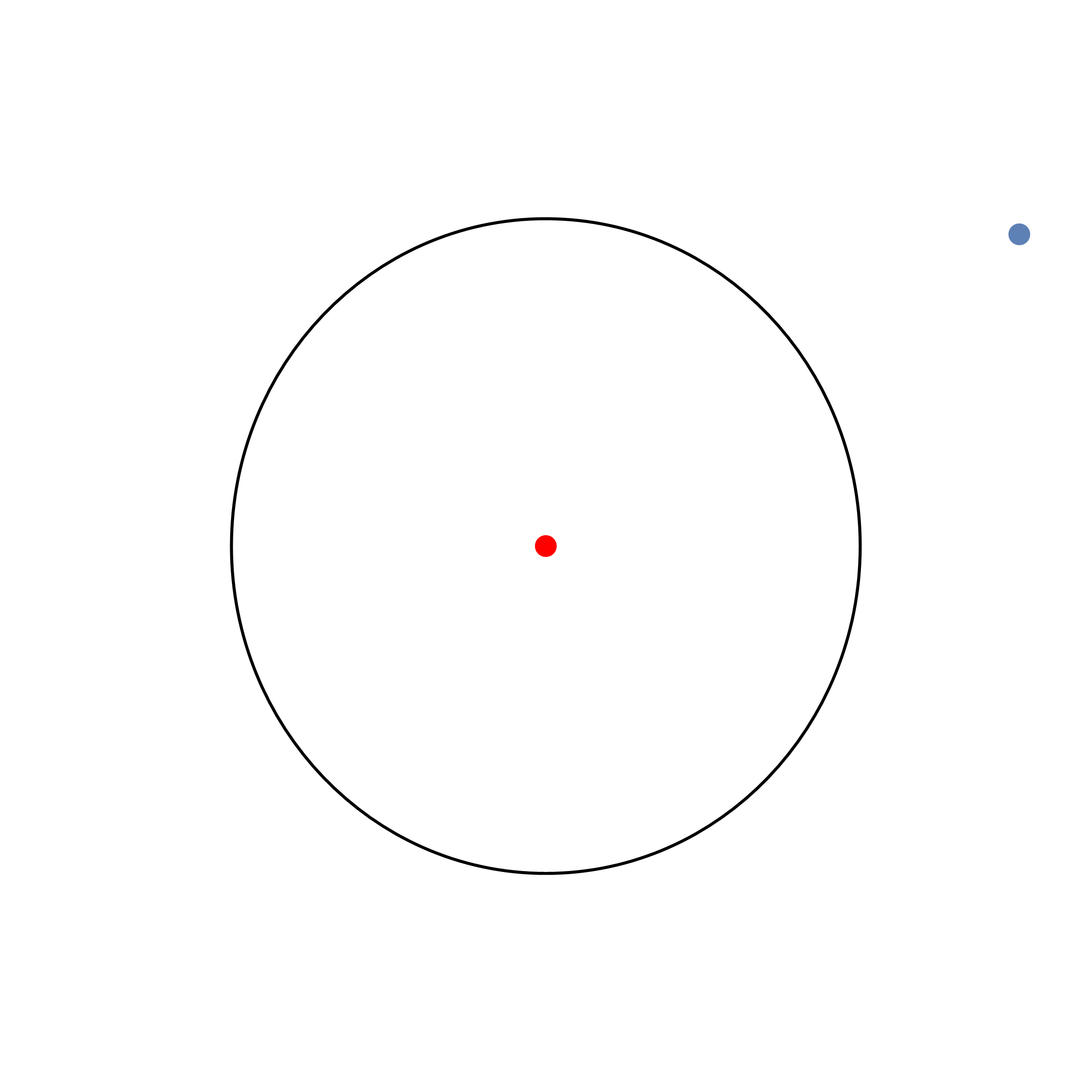} & \includegraphics[width=3cm]{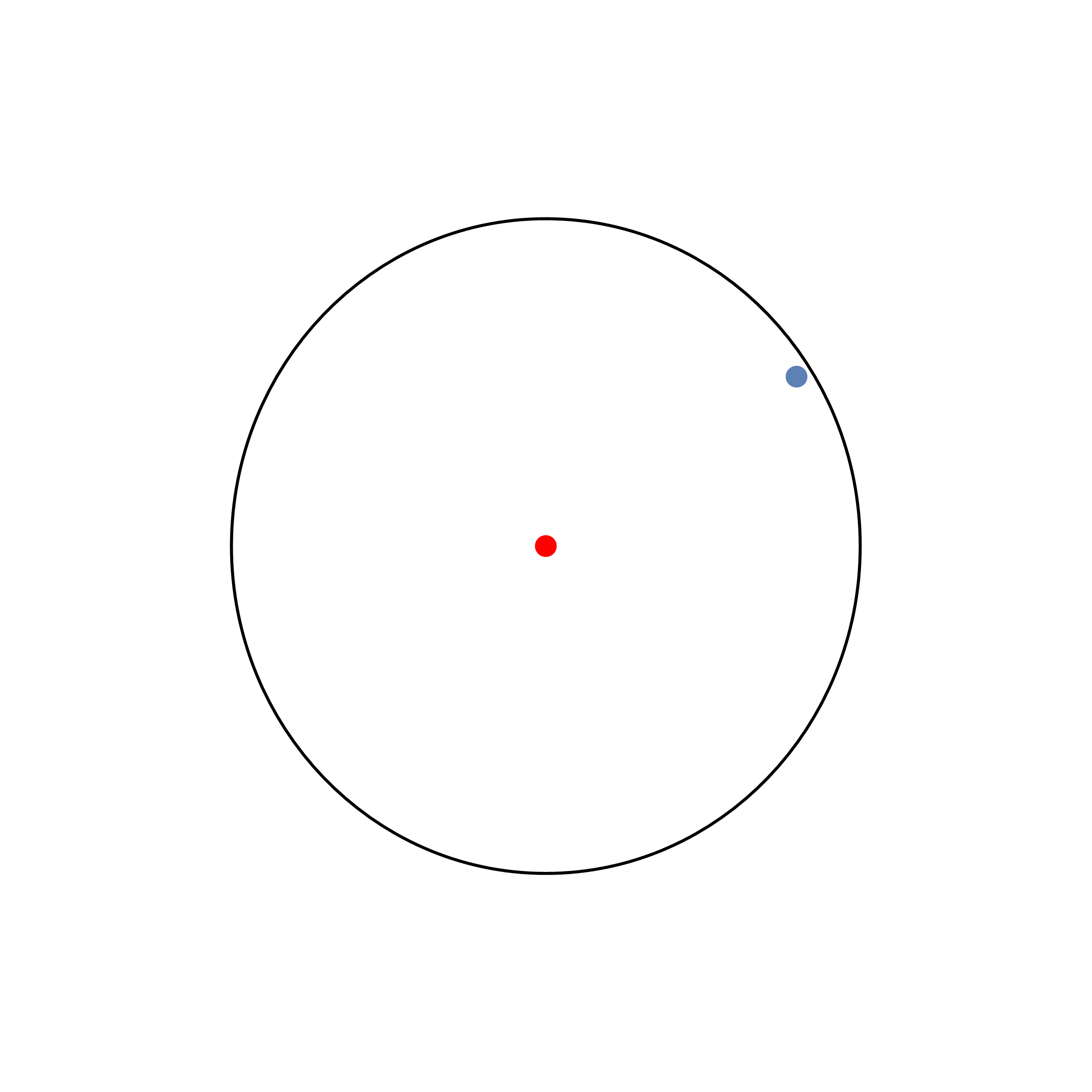} & \includegraphics[width=3cm]{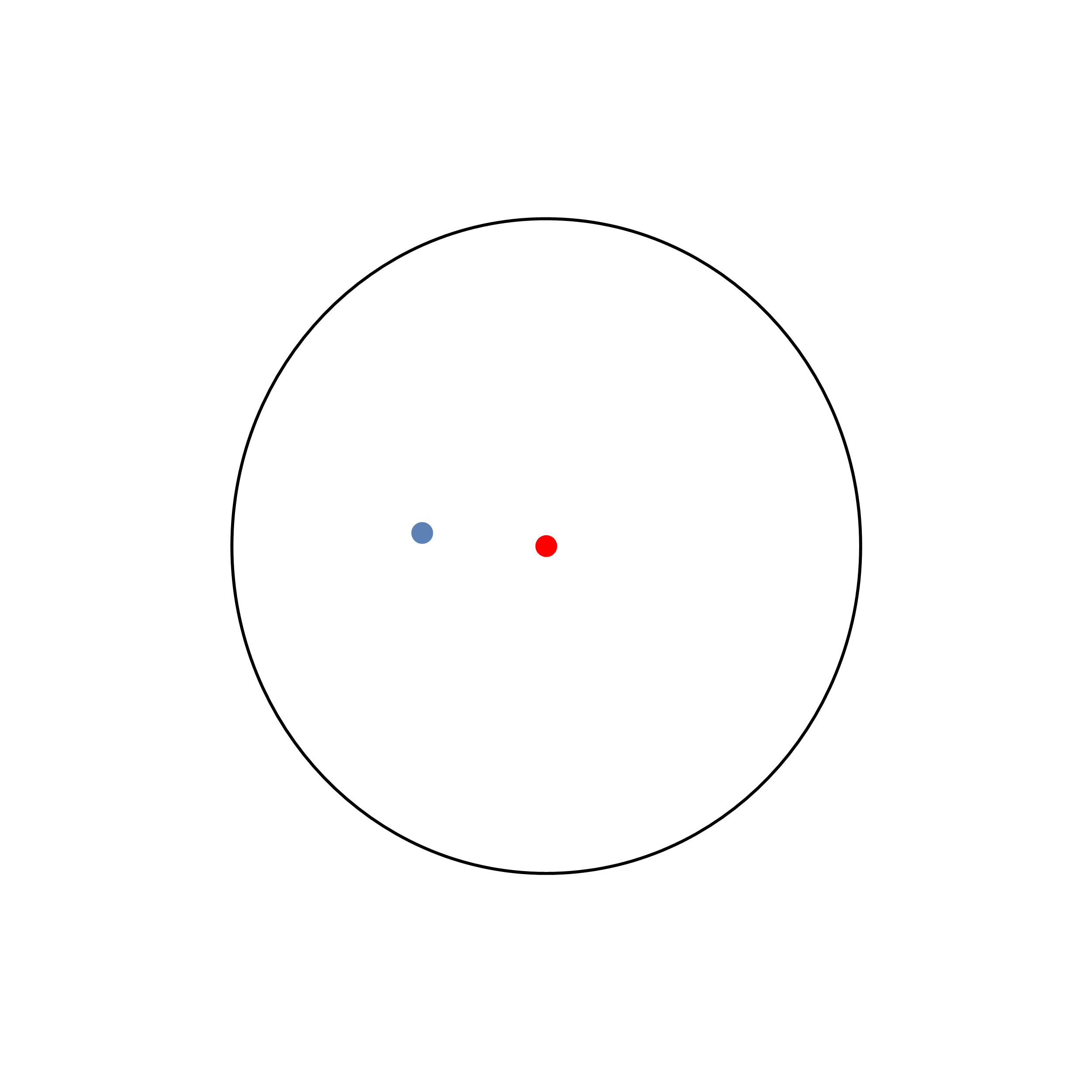}
 & \includegraphics[width=3cm]{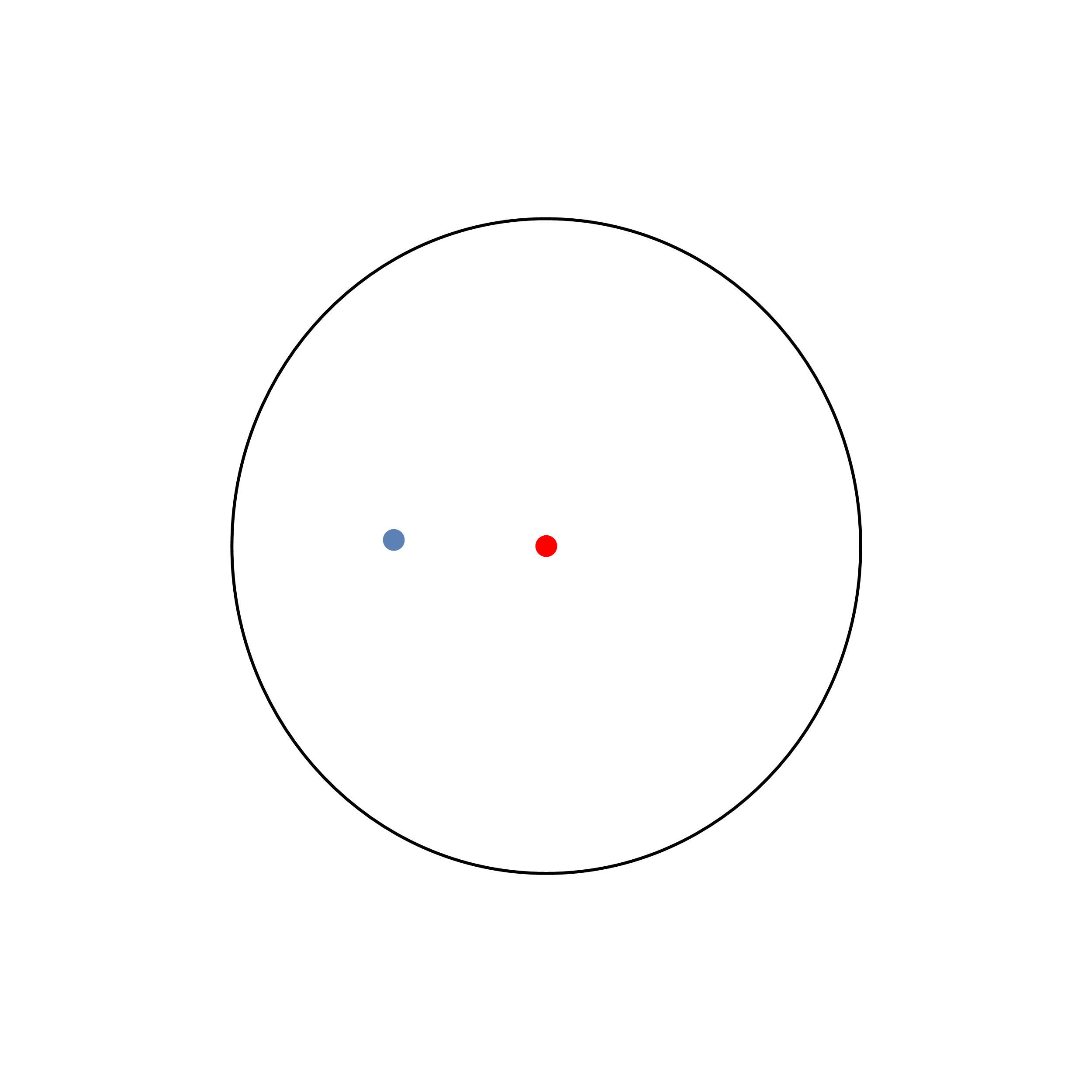}\\[-1ex]
 $m=8$ & $m=16$ & $m=100$ & $m=144$
 \end{tabular}\\
\begin{minipage}{0.8\textwidth}{\small \textbf{Bulk.} Rescaled asymptotic approximation
$\al_{4,k}$ of root of $H_{m,5}\big(E^\frac{1}{2}z\big)$ in red, encircled with a circle of radius $C_\sigma E^{-\frac{4}{3}}$,
where $C_\sigma=\tfrac{1}{3}$ and $k:=\frac{m+2}{4}\sim \tfrac{1}{8}E$ as $E\rightarrow \infty$, for ranging values of $m$.
 In blue the corresponding numerically exact location, conf\/irming the error estimate in Theorem~\ref{thm:asymptoticzeros}.}
\end{minipage}
 \begin{tabular}{ c c c c}
\includegraphics[width=3cm]{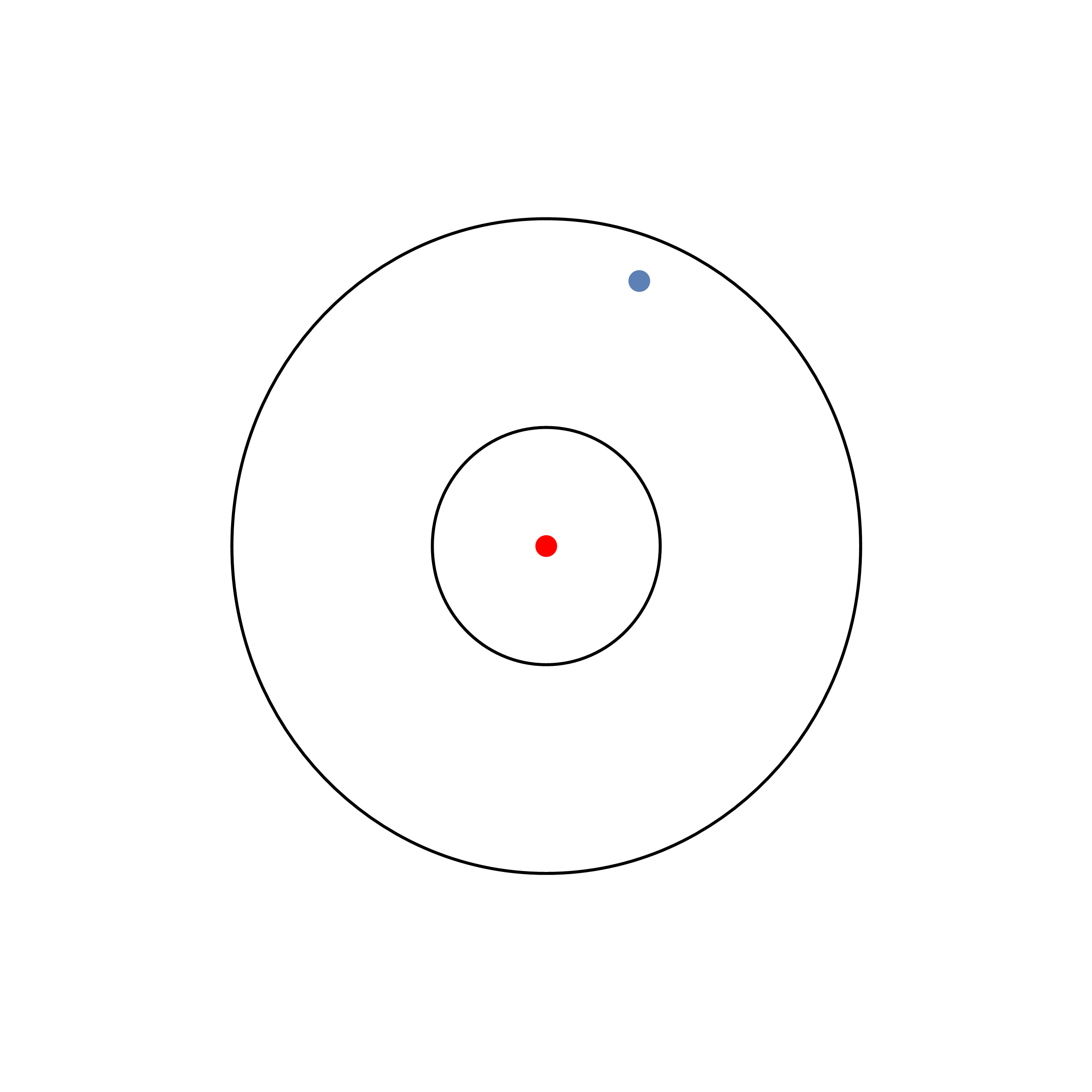} & \includegraphics[width=3cm]{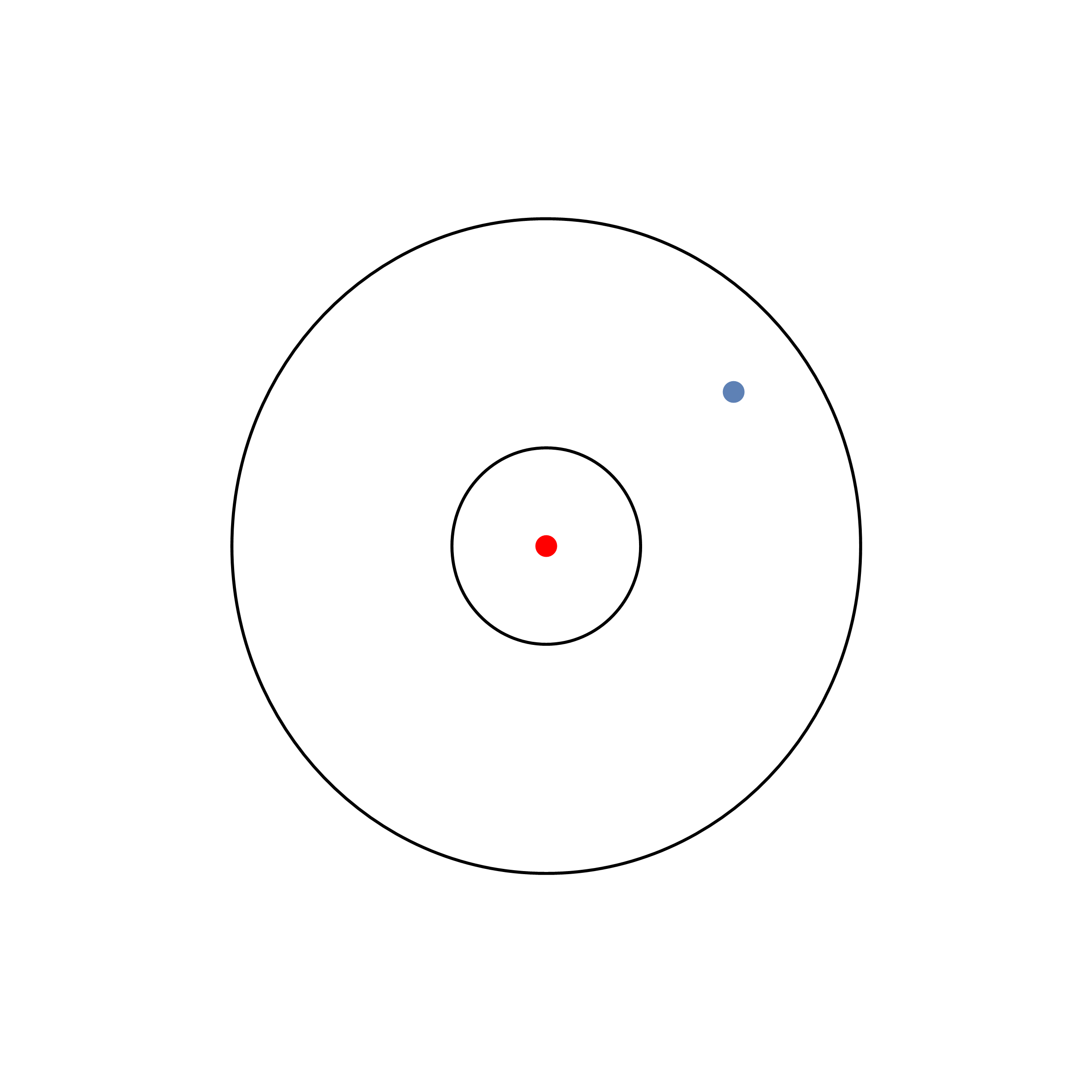} & \includegraphics[width=3cm]{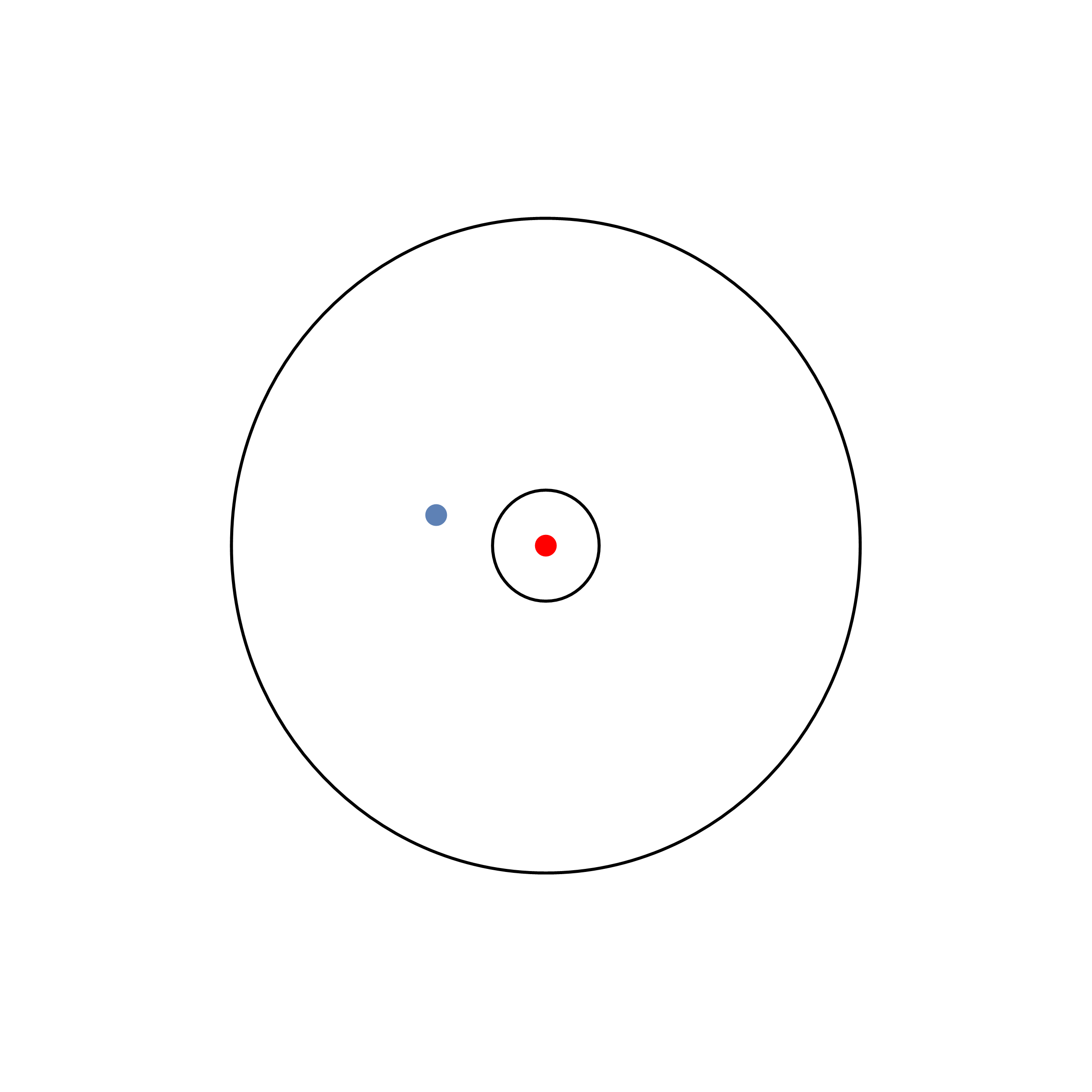} & \includegraphics[width=3cm]{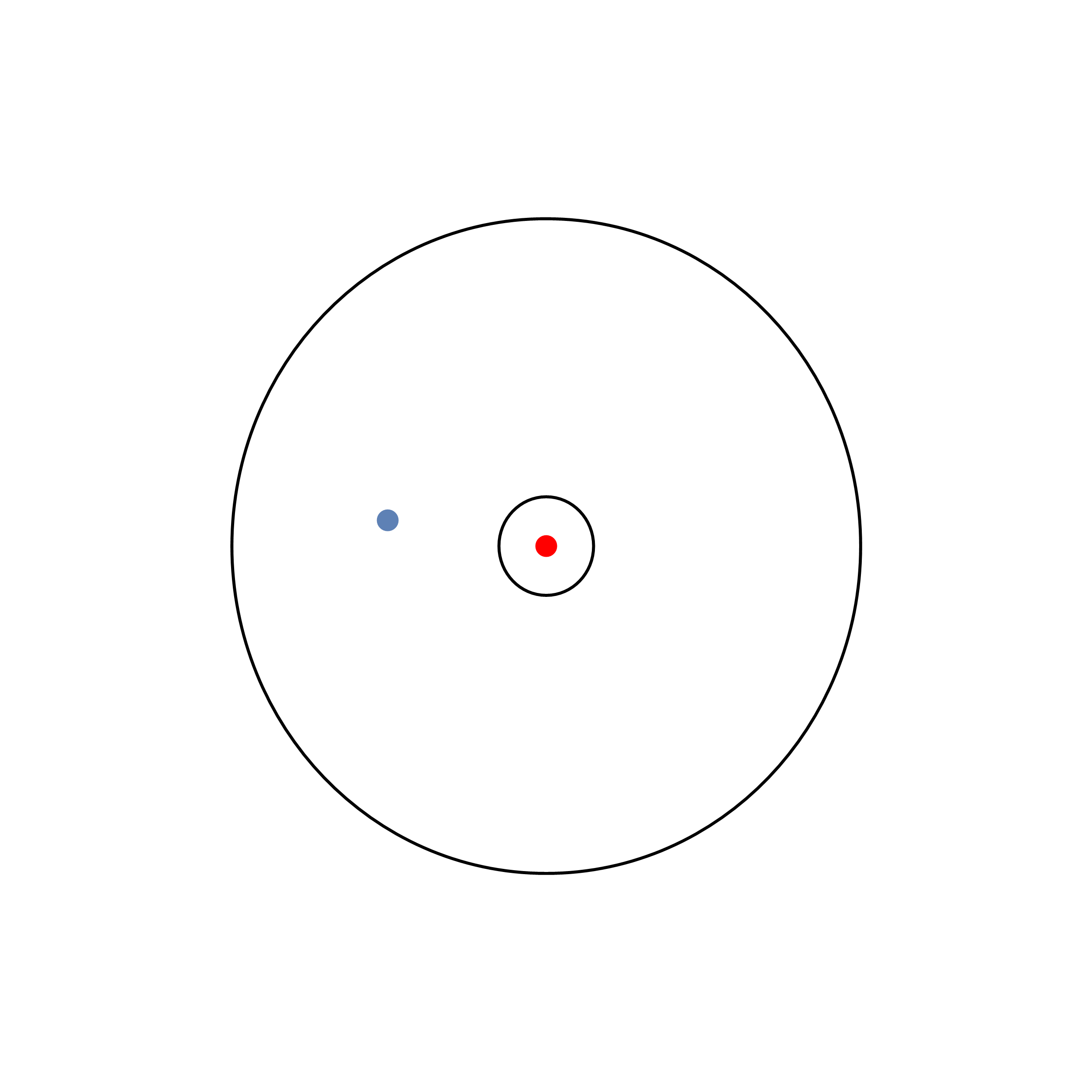}\\[-1ex]
 $m=8$ & $m=16$ & $m=100$ & $m=144$
	\end{tabular}\\
\begin{minipage}{0.8\textwidth}{\small \textbf{Approaching the edge.}
	Asymptotic approximation $\al_{4,k}$ of root of $H_{m,5}\big(E^\frac{1}{2}z\big)$ in red, encircled with two circles of radii
	$C_\sigma E^{-\frac{4}{3}}$ and $C_s E^{-1}$, where $C_\sigma=C_s=\tfrac{1}{12}$ and
	$k:=\big\lfloor\tfrac{1}{4}E-E^{\frac{1}{2}}\big\rfloor+\tfrac{1}{2}\sim \tfrac{1}{4}E$ as
	$E\rightarrow \infty$, for ranging values of $m$.		 In blue the corresponding numerically
	exact location, conf\/irming the error estimate in Theorem~\ref{thm:asymptoticzerosedge} with $s=1$ and $\delta=\tfrac{2}{3}$.}
		 \end{minipage}
\caption{Graphical illustration of errors in Theorems \ref{thm:asymptoticzeros} and~\ref{thm:asymptoticzerosedge}.} \label{fig:thmsnumeric}
\end{figure}

We end the introduction with a few remarks on our asymptotic results.
\begin{rem*}\sloppy
Notice that equation \eqref{eq:realI} for the location of the real part of the zeros of~$H_{m,n}(z)$ coincides with the equation describing the asymptotic location of the zeros of the standard Hermite polynomials \cite[Section~10.17]{bateman2}.
\end{rem*}

\begin{rem*}
	Suppose that the index $k$ in $\al_{j,k}$ is bounded by a number independent on $E$, namely $|k|\leq k_0$ for some $k_0 \in \bb{N}$,
	then the points $\al_{j,k}$ belong to the regular lattice
	\begin{gather*}
	\al_R=\frac{k \pi}{E}+\mathcal{O}\big(E^{-2}\big), \qquad \al_I=\frac{j \log \big(2 E^{\frac12} \big)
		-\frac{1}{2} \log F_{n,j}}{E}+\mathcal{O}\big(E^{-2}\big) ,
	\end{gather*}
	in the large $E$ limit. This explains the lattice-like pattern numerically observed in~\cite{clarkson2003piv}.
\end{rem*}

\begin{rem*} Our approximation scheme includes all zeros, except for a set which asymptotically has probability measure zero. Indeed let $N(A,B)$ be the counting function for $-\infty< A\leq B <\infty$ def\/ined as follows: set $N_E(A,B)$ equal to the number of rescaled roots whose real part belong to the interval $[A,B]$, divided by $m=\frac{E-n}2$, and def\/ine $N(A,B)=\lim\limits_{E\to \infty}N_E(A,B)$. An easy computation, based on the identity $ f'(x)= \sqrt{1-x^2}$, shows that
\begin{gather*}
	N(A,B)=\int_A^B \frac{2 n}{\pi} \sqrt{1-x^2} \chi_{[-1,1]}(x) {\rm d}x.
	\end{gather*}
Here $\chi_{[-1,1]}$ is the characteristic function of the interval $[-1,1]$.	This is a manifestation of Wigner's celebrated semi-circle law.
This behaviour was conjectured in~\cite{FEL}.
\end{rem*}

\begin{rem*} Even though after Theorems \ref{thm:asymptoticzeros} and \ref{thm:asymptoticzerosedge} the bulk of the interval $(-1,1)$ contains
the real part of almost all roots of the generalised Hermite polynomials in the asymptotic regime under consideration, there are roots which converge to the edge of the fundamental domain too fast for the hypothesis of Theorem~\ref{thm:asymptoticzerosedge}. In fact, their real part coalesce with $\pm1$ with speed $E^{-\frac23}$. These roots will be brief\/ly discussed in Section~\ref{sub:edge} at the end of the paper where we show that their distribution is
governed by an Airy-like behaviour.
\end{rem*}

\begin{rem*} While we were f\/inalising the present paper, the paper \cite{buckinghampiv} by R.~Buckingham appeared on the arXiv. This contains the asymptotic analysis, in the regime $m\to \infty$, $\frac{m}{n}=r$, of rational solutions of Hermite type in the `pole-free' region. Even though~\cite{buckinghampiv} does not address the distribution of poles and zeros, it contests the results of \cite{novok15} about the distribution of singularities in the regime $m=n \gg 0$. The paper \cite{novok15} was however already peer-reviewed. Since we are not able to judge on this matter, we await for the answer of its authors.
 \end{rem*}

\section{Poles of rational solutions and anharmonic oscillators}\label{section:polesofrationals}
It is well-known that Painlev\'e equations can be realised as isomonodromic deformations of linear systems. This section is dedicated to, using aforementioned formalism, giving an exact characterisation of poles of rational solutions in terms of anharmonic oscillators satisfying some prescribed constraints. Let us consider the following quadratic oscillators with centrifugal terms,
\begin{gather}
\psi''(\lambda) = V(\lambda;a,b,\theta)\psi(\lambda),\nonumber\\
V(\lambda;a,b,\theta)=\lambda^2+2a\lambda+a^2+2(1-\theta_\infty)
-\big[b+\big(2\theta_\infty-\tfrac{1}{2}\big)a\big]\lambda^{-1}+\big(\theta_0^2-\tfrac{1}{4}\big)\lambda^{-2},\label{eq:anharmoniccent}
\end{gather}
where $a,b\in\mathbb{C}$ and $\theta=(\theta_0,\theta_\infty)\in\mathbb{C}^*\times\mathbb{C}$. Without loss of generality,
we may always assume $\Re[\theta_0]\geq 0$, because of the invariance of~\eqref{eq:anharmoniccent} under $\theta_0\mapsto-\theta_0$.

Now note that $\lambda=0$ is a regular singular point with indices $\tfrac{1}{2}\pm \theta_0$, and there generically exist two linearly independent solutions $\psi_{\pm}(\lambda)$, with corresponding Frobenius expansions, uniquely def\/ined by the asymptotic behaviour at zero. However, in the resonant case
$\theta_0=\tfrac{1}{2}n\in\tfrac{1}{2}\mathbb{N}^*$, the Frobenius expansion of the dominant solution is not uniquely def\/ined by its behaviour in zero $\psi_{\minus}(\lambda)\sim\lambda^{\frac{1}{2}-\theta_0}$ as $\lambda\rightarrow 0$, and logarithmic terms may appear. Indeed, we have
\begin{gather*}
\psi_+(\lambda)=\lambda^{\tfrac{1}{2}+\theta_0}(1+\mathcal{O}(\lambda)),\qquad \lambda\rightarrow 0,\\
 \psi_-(\lambda)=c\log(\lambda) \psi_+(\lambda)+\lambda^{\tfrac{1}{2}-\theta_0}\bigg(1+\sum_{k\geq1}c_k \la^k\bigg),
\end{gather*}
where the constant $c$ and $c_k$'s are uniquely def\/ined by imposing $c_{n}=0$.

For convenience of the reader we recall the def\/inition of an \textit{apparent} singularity, a concept of remarkable importance for the rest of the paper.
\begin{Def}A resonant regular singularity of a second order linear ODE is called apparent, if the dominant solution $\psi_+$ does not have any logarithmic
contribution.

Similarly, in the case of a f\/irst order system of linear ODEs, we say that a regular singularity is apparent, if the monodromy about the singularity is a~scalar multiple of the identity.
\end{Def}

In turn, equation \eqref{eq:anharmoniccent} has an irregular singularity at $\lambda=\infty$, hence solutions exhibit the Stokes phenomenon, which we discuss brief\/ly. For each $k \in \bb{Z}_4$, we def\/ine the Stokes sector $\Omega_k=e^{\frac{k\pi i}{4}}\Omega_0$, where $\Omega_0=\{ - \frac{\pi}{4} < \arg{\lambda} <\frac{\pi}{4}\}$. In each Stokes sector there exists an (up to normalisation) unique subdominant solution $\psi_k(\lambda)$, which decays exponentially as $\lambda\rightarrow \infty$ in this sector.

Fixing the principal branch for powers of $\lambda$ in the $\lambda$-plane, with respect to the branch cut $\arg{\lambda}=-\frac{\pi}{4}$, and denoting
\begin{gather}\label{eq:gfunction}
g(\lambda,z)=\tfrac{1}{2}\lambda^2+z \lambda ,
\end{gather}
we f\/ix the following normalisation of the subdominant solution $\psi_k(\la)$,
\begin{gather} \label{eq:1psikscalar}
\psi_{k}(\lambda)=\begin{cases}
\big(1+\mathcal{O}\big(\lambda^{-1}\big)\big)e^{g(\lambda,a)}\big(\frac{1}{\lambda}\big)^{-\tfrac{1}{2}+\theta_\infty} & \text{if $k$ is odd,}\\
\big({-}\frac{1}{2}+\mathcal{O}\big(\lambda^{-1}\big)\big)e^{-g(\lambda,a)}\big(\frac{1}{\lambda}\big)^{\tfrac{3}{2}-\theta_\infty} & \text{if $k$ is even,}
\end{cases}
\end{gather}
as $\lambda\rightarrow \infty$ in $\Omega_k$, for $k\in\mathbb{Z}_4$.

We write $\psi\equiv \phi$, for solutions of \eqref{eq:anharmoniccent}, if\/f they are linearly dependent.

We have the three following characterisations of roots of Hermite polynomials~$H_{m,n}$, i.e., zeros and poles of rational solutions of Hermite type, via
an inverse monodromy problem for the anharmonic oscillator~\eqref{eq:anharmoniccent}.
\begin{thm}\label{thm:hermiteanharmonic}
Fix $m,n\in\mathbb{N}^*$, then $z=a$ is a zero of the generalised Hermite polyno\-mial~$H_{m,n}(z)$, is equivalent to any of the following three statements:
	\begin{enumerate}[label={{\rm H.\arabic*:}},ref={\rm H.\arabic*}] \itemsep=0pt
\item There exists $b\in\mathbb{C}$ $($a fortiori unique$)$, such that the anharmonic oscillator \eqref{eq:anharmoniccent} with
		$\theta_0=\tfrac{1}{2}n$ and $\theta_\infty=m+1+\tfrac{1}{2}n$, has an apparent singularity at $\lambda=0$,
		and the subdominant solution at $\lambda=\infty$ in $\Omega_0$, is also subdominant at $\lambda=0$, i.e.,
		$\psi_0(\lambda)\equiv \psi_{+}(\lambda)$.
		
		Furthermore, in such case, it turns out that $\psi_0(\lambda)\equiv \psi_2(\lambda)$ and
		\begin{gather*}
		\psi_0(\lambda)=\lambda^{\frac{1}{2}(n+1)}e^{-g(\lambda,a)}p(\lambda),
		\end{gather*}
		where $p(\lambda)$ is a polynomial of degree $m-1$ with no repeated roots and $p(0)\neq 0$,
		and $g(\la,z)$ is as def\/ined in~\eqref{eq:gfunction}. This in particular implies that
		$\psi_0(\lambda)$ has exactly $m-1$ nonzero simple roots in the complex plane.	\label{H1}

\item There exists $b\in\mathbb{C}$ $($a fortiori unique$)$, such that the anharmonic oscillator~\eqref{eq:anharmoniccent}
		with $\theta_0=\tfrac{1}{2}m$ and $\theta_\infty=1-\tfrac{1}{2}m-n$, has an apparent singularity at $\lambda=0$, and
		the subdominant solution at $\lambda=\infty$ in $\Omega_1$, is also subdominant at $\lambda=0$, i.e.,
		$\psi_1(\lambda)\equiv \psi_{+}(\lambda)$.
		
		Furthermore, in such case, it turns out $\psi_1(\lambda)\equiv \psi_3(\lambda)$ and
		\begin{gather*}
		\psi_1(\lambda)=\lambda^{\frac{1}{2}(m+1)}e^{g(\lambda,a)}p(\lambda),
		\end{gather*}
		where $p(\lambda)$ is a polynomial of degree $n-1$ with no repeated roots and $p(0)\neq 0$,
		and $g(\la,z)$ is as def\/ined in~\eqref{eq:gfunction}. This in particular implies that~$\psi_1(\lambda)$ has exactly $n-1$ nonzero simple roots in the complex plane.\label{H2}

\item There exists $b\in\mathbb{C}$ $($a fortiori unique$)$, such that, for the anharmonic oscillator~\eqref{eq:anharmoniccent} with
		$\theta_0=\tfrac{1}{2}(m+n)$ and $\theta_\infty=\tfrac{1}{2}(n-m+2)$, subdominant solutions near $\lambda=\infty$ in
		opposite Stokes sectors are linearly dependent, i.e., $\psi_0(\lambda)\equiv \psi_2(\lambda)$ and $\psi_1(\lambda)\equiv \psi_3(\lambda)$.
		
		Furthermore, in such case, it turns out that $\lambda=0$ is an apparent singularity and
			\begin{gather*}
			\psi_0(\lambda)=\lambda^{-\frac{m+n-1}{2}}e^{-g(\lambda,a)}p(\lambda),\qquad \psi_1(\lambda)=\lambda^{-\frac{m+n-1}{2}}e^{g(\lambda,a)}q(\lambda),
			\end{gather*}
		where $p(\lambda)$ and $q(\lambda)$ are polynomials of degree $n-1$ and $m-1$ respectively without repeated roots and $p(0),q(0)\neq0$,
		and $g(\la,z)$ is as def\/ined in \eqref{eq:gfunction}. This in particular implies that~$\psi_0(\lambda)$ and~$\psi_1(\lambda)$ have exactly $n-1$ and $m-1$ nonzero simple roots in the complex plane respectively.	\label{H3}
	\end{enumerate}
	\end{thm}
	
To state the corresponding theorem for Okamoto rationals, we brief\/ly discuss the Stokes multipliers for the anharmonic oscillator~\eqref{eq:anharmoniccent}.
If $k=1,2$, then the asymptotic characterisa\-tion~\eqref{eq:1psikscalar} of the solution~$\psi_k(\la)$ is in fact valid on the larger sector
$\Omega_{k+1}\cup\overline{\Omega}_{k}\cup \Omega_{k-1}$. It follows that $\{\psi_k,\psi_{k-1}\}$ is necessarily linearly independent and hence forms a basis of solutions of~\eqref{eq:anharmoniccent}. Comparison of the asymptotic behaviour in $\Omega_k$ as $\lambda \rightarrow \infty$, of $\psi_{k+1}$ and the basis elements, leads to the relation
\begin{gather}\label{eq:stmult12}
\psi_{k+1}(\lambda)=\psi_{k-1}(\lambda)+s_k\psi_{k}(\lambda), \qquad k=1,2
\end{gather}
for a unique $s_k\in\mathbb{C}$, called the $k$th Stokes multiplier. Appropriately modifying the above argument for the $k=0,3$ cases, due to the choice of branch in the $\lambda$-plane, leads to
\begin{subequations}\label{eq:stmult03}
	\begin{gather}
	\psi_{0}(\lambda)=-e^{-\pi i \theta_\infty}\big(\psi_{2}\big(e^{2\pi i}\lambda\big)+s_3\psi_{3}\big(e^{2\pi i}\lambda\big)\big),\\
	\psi_{1}(\lambda)=-e^{\pi i \theta_\infty}\psi_{3}\big(e^{2\pi i}\lambda\big)+s_0\psi_{0}(\lambda).
	\end{gather}
\end{subequations}
for unique $s_0,s_3\in\mathbb{C}$.

We have the following characterisation of roots of the generalised Okamoto polynomials, or equivalently of zeros and poles of rational solutions of Okamoto type.
\begin{thm}\label{thm:okamotoanharmonic}
		Fix $m,n\in\mathbb{Z}$, then $z=a$ is a zero of the generalised Okamoto polyno\-mial~$Q_{m,n}(z)$, if and only if
there exists $b\in\mathbb{C}$, such that the Stokes data of the anharmonic oscillator~\eqref{eq:anharmoniccent} with $\theta_0=\tfrac{1}{2}n-\tfrac{1}{6}$ and $\theta_\infty=\tfrac{1}{2}(2m+n+1)$, satisfy
			\begin{gather*}
			s_0=s_2,\qquad s_1=s_3,\qquad s_0s_1+1=0.
			\end{gather*}
\end{thm}

The remainder of this section is dedicated to the proofs of Theorems~\ref{thm:hermiteanharmonic} and~\ref{thm:okamotoanharmonic}. In Section~\ref{section:iso} we discuss the realisation of ${\rm P}_{\rm IV}$ as an isomonodromic deformation. Then, in Section~\ref{section:localisation} we show that, upon localising the isomonodromic system near poles, we arrive at the anharmonic oscillator~\eqref{eq:anharmoniccent}, while preserving the monodromy of the system. This allows us to characterise poles of ${\rm P}_{\rm IV}$ transcendents as solutions of certain inverse monodromy problems concerning the anharmonic oscillator,
in particular leading to proofs of the aforementioned theorems in Section~\ref{section:proofofcharacterisations}.

\begin{rem*} In recent literature some special interest is given to anharmonic oscillators for which one solution is expressible in closed form as a polynomial (or rational function) times an exponential function. These are called quasi exactly solvable potentials \cite{bender98quasi,turbiner2016}. We remark here that, by Theorem \ref{thm:hermiteanharmonic} (see~\cite{marquettequesne2016} for some similar considerations), all three types of oscillators related to the roots of generalised Hermite polynomials are quasi exactly solvable in this sense. Moreover, the type~III case is particularly special as all solutions are expressible in terms of a linear combination of rational functions times an exponential and hence deserve to be called super quasi exactly solvable. It is interesting to note that, while quasi exactly solvable potentials are mostly studied with the help of the ${\mathfrak{sl}}_2$ Lie algebra, in the present case they are related to B\"acklund transformations of ${\rm P}_{\rm IV}$, which form the af\/f\/ine Weyl group of type $A_2^{(1)}={\mathfrak{sl}}_3^{(1)}(\bb{C})$ \cite{noumi}.
\end{rem*}
		
\subsection{The Garnier--Jimbo--Miwa Lax pair}\label{section:iso}
In this section, we introduce the isomonodromic deformation method for ${\rm P}_{\rm IV}$, as developed by Kitaev~\cite{kitaev1988}, Ablowitz et al.~\cite{fokasmugan1988} and Kapaev \cite{kapaev1996,kapaev1998}. The main aim of this section, is to characterise rational solutions by means of monodromy data of an associated linear problem.

Jimbo and Miwa \cite{jimbomiwaII1981} realised ${\rm P}_{\rm IV}$ as the compatibility condition of the following two linear dif\/ferential systems
\begin{subequations}\label{eq:laxpair}
	\begin{gather}
	Y_{\lambda}(\lambda,z)=
	\begin{pmatrix}\displaystyle
	\lambda+z+\lambda^{-1}(\theta_0-v) &
	u\left(1-\displaystyle\frac{\om}{2\lambda}\right) \\
	\displaystyle\frac{2}{u}\left(v-\theta_0-\theta_\infty+\displaystyle\frac{v}{\lambda \om}(v-2\theta_0)\right)&
	-\lambda-z-\lambda^{-1}(\theta_0-v)
	\end{pmatrix}
	Y(\lambda,z),\label{eq:1isopiv}\\
	Y_{z}(\lambda,z)=
	\begin{pmatrix}\displaystyle
	\lambda &
	u \\
	\displaystyle\frac{2}{u}\left(v-\theta_0-\theta_\infty\right)&
	-\lambda
	\end{pmatrix}
	Y(\lambda,z), \label{eq:1isopivdeformation}
	\end{gather}
\end{subequations}
where $u$ is an auxiliary function satisfying
\begin{gather}\label{eq:udefi}
\frac{u_z}{u}=-\om-2z,
\end{gather}
and $v$ is def\/ined by
\begin{gather*}
4v=-\om_z+\om^2+2z\om+4\theta_0.
\end{gather*}
That is, $Y_{\lambda z}=Y_{z\lambda}$ implies that $\om$ and $u$ satisfy ${\rm P}_{\rm IV}(\theta)$ and~\eqref{eq:udefi} respectively.

Note that the linear system \eqref{eq:1isopiv}, associated to ${\rm P}_{\rm IV}$, is not def\/ined at zeros and poles of~$\om$.
To study what happens near such points, we follow the strategy pioneered by A.~Its and V.~Novoskhenov for Painlev\'e~II~\cite{its86},
and then employed by one of the authors \cite{myphd,piwkb} for Painlev\'e~I, see also~\cite{chudnovsky90}.
It turns out that, even though the matrix linear system associated to ${\rm P}_{\rm IV}$ is not def\/ined at zeros and poles, on the contrary its scalar version has a well-def\/ined limit. We consider the equivalent 2nd order scalar dif\/ferential equation for the gauged f\/irst component
$\Psi(\lambda,z)=\lambda^{\frac{1}{2}}Y_1(\lambda,z)$ in~\eqref{eq:laxpair},
\begin{subequations}
	\label{eq:scalarlaxpair}
\begin{gather}
\Psi_{\lambda\lambda}(\lambda,z) = \frac{2}{2\lambda-\om}\Psi_{\lambda}(\lambda,z)+V(\lambda,z)\Psi(\lambda,z), \label{eq:anharmonicgen}\\
\Psi_{z}(\lambda,z) = \frac{2\lambda}{2\lambda-\om}\Psi_{\lambda}(\lambda,z)-\frac{\lambda(\om+2z)+1+2(\theta_0-v)}{2\lambda-\om}\Psi(\lambda,z),\label{eq:deform}\\
V(\lambda,z)=\lambda^2+2z\lambda+z^2-2\theta_\infty+1+R(\lambda,z)\lambda^{-1}+\big(\theta_0^2-\tfrac{1}{4}\big)\lambda^{-2} ,\nonumber
\end{gather}
\end{subequations}
which we refer to as the scalar Garnier--Jimbo--Miwa (GJM) Lax pair, where
\begin{gather*}
R(\lambda,z)=\big(\theta_0+\theta_\infty-v-\tfrac{1}{2}\big)\om+2z(\theta_0-v)-\frac{2(2\theta_0-v)v}{\om}-\frac{\om^2+2z\om-4v+4\theta_0+2}{2(2\lambda-\om)}.
\end{gather*}
Besides $\lambda=0$ regular singular and $\lambda=\infty$ an irregular singular point, equation~\eqref{eq:anharmonicgen} has a~further apparent singularity at $\lambda=\tfrac{1}{2}\om$, with indices~$\{0,2\}$. So solutions of \eqref{eq:anharmonicgen} can branch only at $\lambda=0$ and $\lambda=\infty$, and hence live on the universal covering space of $\mathbb{C}^*$, which we denote by $\mathbb{C}^\infty$.
\begin{lem}\label{lem:behavioursolutions}
Any local solution $\Psi(\lambda,z)$ of \eqref{eq:scalarlaxpair}, extends uniquely to a global single-valued meromorphic function on~$\mathbb{C}^\infty\times \mathbb{C}$, singular in $z$ only where $\om(z)$ has a pole with $+1$ residue, in which case $\Psi(\lambda,z)$ has a simple pole in~$z$. The behaviour of $\Psi(\lambda,z)$ near zeros and poles of~$\om(z)$ is characterised as follows:
	\begin{itemize}\itemsep=0pt
		\item 	If $\om(z)$ has a zero with negative sign at $z=a$, say corresponding expansion takes the form~\eqref{eq:laurentpole} with $\epsilon=-1$, then $\Psi(\lambda,z)$ is holomorphic at $z=a$ and	$\psi(\lambda):=\lambda^{-\frac{1}{2}}\Psi(\lambda,a)$
defines a solution of the anharmonic oscillator
		\begin{align*}
		\psi''(\lambda)=\big( \lambda^2+2a\lambda+a^2+1-2\theta_\infty+\tfrac{1}{2}b\lambda^{-1}+
		\theta_0(\theta_0+1)\lambda^{-2}\big)\psi(\lambda) .
		\end{align*}
		\item 	If $\om(z)$ has a zero with positive sign at $z=a$, say corresponding expansion takes the form~\eqref{eq:laurentpole}
		with $\epsilon=1$, then $\Psi(\lambda,z)$ is holomorphic at $z=a$ and $\psi(\lambda):=\lambda^{-\frac{1}{2}}\Psi(\lambda,a)$
	 defines a solution of the anharmonic oscillator
		\begin{align*}
		\psi''(\lambda)=\big( \lambda^2+2a\lambda+a^2+1-2\theta_\infty+\tfrac{1}{2}b\lambda^{-1}+
		\theta_0(\theta_0-1)\lambda^{-2} \big)\psi(\lambda) .
	\end{align*}
		\item 	If $\om(z)$ has a pole with residue $-1$ at $z=a$, say corresponding expansion takes the form~\eqref{eq:laurentpole} with $\epsilon=-1$, then $\Psi(\lambda,z)$ is holomorphic at $z=a$ and $\psi(\lambda):=\Psi(\lambda,a)$ defines a solution of the anharmonic oscillator~\eqref{eq:anharmoniccent}, namely
		\begin{gather*}
\psi''(\lambda) = V(\lambda;a,b,\theta)\psi(\lambda),\\
V(\lambda;a,b,\theta)=\lambda^2+2a\lambda+a^2+2(1-\theta_\infty)
-\big[b+\big(2\theta_\infty-\tfrac{1}{2}\big)a\big]\lambda^{-1}+\big(\theta_0^2-\tfrac{1}{4}\big)\lambda^{-2},
\end{gather*}
	\item 	If $\om(z)$ has a pole with residue $+1$ at $z=a$, say corresponding expansion takes the form~\eqref{eq:laurentpole} with $\epsilon=+1$, then $\Psi(\lambda,z)$ has a simple pole at $z=a$, and corresponding residue $\psi(\lambda):=\lim\limits_{z\rightarrow a}{(z-a)\Psi(\lambda,z)}$ defines a solution of the anharmonic oscillator
	\begin{gather*}
	\psi''(\lambda)=\big( \lambda^2+2a\lambda+a^2-2\theta_\infty-\big[b+\big(2\theta_\infty-\tfrac{3}{2}\big)a\big]\lambda^{-1}+
	\big(\theta_0^2-\tfrac{1}{4}\big)\lambda^{-2} \big)\psi(\lambda) .
	\end{gather*}
	\end{itemize}	
\end{lem}
\begin{proof}
Note that, without assuming the Painlev\'e property of ${\rm P}_{\rm IV}$, the statement of the lemma is rather nontrivial, as it in particular implies the Painlev\'e property for ${\rm P}_{\rm IV}$, see for instance Fokas et al.~\cite{fokas}.

However, we know that, around any point $z=a$ in the f\/inite complex plane, $\om(z)$ has a~convergent Laurent expansion, taking any of the forms \eqref{eq:laurentgeneric}--\eqref{eq:laurentpole}, which greatly simplif\/ies the proof. The unique global single-valued meromorphic continuation of local solutions is proved by showing:
\begin{enumerate}\itemsep=0pt
	\item[1)] near any point $(\lambda,z)=(\lambda_0,a)\in \mathbb{C}^\infty\times \mathbb{C}$, there exists a meromorphic local fundamental solution\footnote{In this context, a local fundamental solution is def\/ined as a row vector of two linearly independent solutions on a simply connected domain.} of \eqref{eq:scalarlaxpair};
	\item[2)] any two local fundamental solutions $Y(\lambda,z)$ and $\widetilde{Y}(\lambda,z)$ of \eqref{eq:scalarlaxpair}, with non-empty intersection of domains, are related by a constant invertible matrix $C$ as in \eqref{eq:connectionmatrix}, on this intersection. 
\end{enumerate}
To prove the f\/irst statement, we have to distinguishing between the dif\/ferent possibilities \smash{\eqref{eq:laurentgeneric}--\eqref{eq:laurentpole}}, and whether $\lambda_0=\tfrac{1}{2}\om(a)$ or not. We give a detailed account of the case where $z=a$ is a~pole with residue~$-1$ of~$\om(z)$. The other cases are proven analogously.

So suppose $z=a$ is a pole with residue $-1$ of $\om(z)$, and let $b\in\mathbb{C}$ be def\/ined by~\eqref{eq:laurentpole} with $\epsilon=-1$.
Firstly, we choose small open discs $\lambda_0\in \Lambda\subseteq \mathbb{C}^\infty$ and $a\in Z\subseteq\mathbb{C}$ such that $\om(z)$ has no zeros or poles on $Z\setminus\{a\}$ and $2\lambda-w(z)$ does not vanish on $\Lambda\times (Z\setminus\{a\})$.

Writing $\mathcal{U}_1=\Psi$ and $\mathcal{U}_2=\Psi_\lambda$, equations~\eqref{eq:scalarlaxpair} can be rewritten as
\begin{subequations}\label{eq:U}
\begin{gather}
\mathcal{U}_\lambda =A(\lambda,z)\mathcal{U},\label{eq:U1}\\
\mathcal{U}_z =B(\lambda,z)\mathcal{U},\label{eq:U2}
\end{gather}
\end{subequations}
where
\begin{gather*}
A=\begin{pmatrix}
0 & 1\\
V& \frac{2}{2\lambda-w}
\end{pmatrix},\qquad B=\begin{pmatrix}
b & \frac{2\lambda}{2\lambda-w}\\
b_\lambda+\frac{2\lambda}{2\lambda-w} V & b+\frac{2}{2\lambda-w}
\end{pmatrix},\\
b:=-\frac{\lambda(\om+2z)+1+2(\theta_0-v)}{2\lambda-\om}.
\end{gather*}
Furthermore, note that, as $\om$ solves ${\rm P}_{\rm IV}$, the compatibility condition
\begin{gather}\label{eq:compatibility}
A_z-B_\lambda+AB-BA=0,
\end{gather}
is satisf\/ied. Direct calculation gives
\begin{gather}\label{eq:Vconvergence}
V(\lambda,z)=V(\lambda;a,b,\theta)+\mathcal{O}(z-a),\qquad z\rightarrow a,
\end{gather}
locally uniformly in $\lambda\in \Lambda$, where $V(\lambda;a,b,\theta)$ is the potential def\/ined in~\eqref{eq:anharmoniccent}. Using this fact, it easily follows that $A(\lambda,z)$ is analytic on $\Lambda\times Z$.

We will show the existence of a unique fundamental solution $\mathcal{U}(\lambda,z)$ of equations~\eqref{eq:U}, analytic on $\Lambda\times Z$, with $\mathcal{U}(\lambda_0,a)=I$. To this end, f\/irstly note that the Cauchy problem
\begin{gather*}
U_\lambda=A(\lambda,z)U, \lambda\in\Lambda,\\
U(\lambda_0)=I,
\end{gather*}
has a unique analytic solution $U=U(\lambda,z)$, for every $z\in Z$, and in fact $U(\lambda,z)$ is analytic on $\Lambda\times Z$, as $A(\lambda,z)$ is analytic on $\Lambda\times Z$. Furthermore it is easy to see that
\begin{gather}\label{eq:Udet}
|U(\lambda,z)|=\frac{2\lambda-w}{2\lambda_0-w}.
\end{gather}
 We now search for a solution of \eqref{eq:U}, of the form
\begin{gather}\label{eq:mathcalU}
\mathcal{U}(\lambda,z)=U(\lambda,z)F(z),
\end{gather}
where $F(z)$ analytic on $Z$ with $F(a)=I$, independent of $\lambda$. Note that $\mathcal{U}(\lambda,z)$ trivially satis\-f\/ies~\eqref{eq:U1}, and~\eqref{eq:U2} reduces to
\begin{gather}
F_z=C(\lambda,z)F, \qquad z\in Z,\nonumber\\
F(a)=I,\label{eq:Fcauchy}
\end{gather}
where $C:=U^{-1}BU-U^{-1}U_z$. From equation \eqref{eq:Udet}, it easily follows that $C(\lambda,z)$ is analytic on~$\Lambda\times Z$. Hence, to show that the Cauchy problem \eqref{eq:Fcauchy} has a unique solution~$F$, independent of $\lambda$, it remains to show that $C_\lambda=0$. By direct calculation, using
\begin{gather*}
\big(U^{-1}\big)_\lambda=-U^{-1}U_\lambda U^{-1},\qquad U_{z\lambda}=A_zU+AU_z,
\end{gather*}
we obtain
\begin{gather*}
C_\lambda=U^{-1} (-AB+B_\lambda+BA-A_z )U=0,
\end{gather*}
where the latter equality is precisely equation \eqref{eq:compatibility}.

We conclude that the Cauchy problem \eqref{eq:U2} has a unique fundamental solution $F(z)$, ana\-ly\-tic on $Z$, which is independent of $\lambda$. Equation \eqref{eq:mathcalU} now def\/ines a unique fundamental solu\-tion~$\mathcal{U}(\lambda,z)$ of~\eqref{eq:U}, analytic on $\Lambda\times Z$, with $\mathcal{U}(\lambda_0,a)=I$. Set
\begin{gather}\label{eq:Ystar}
Y^*(\lambda,z)=(\Psi_1^*(\lambda,z),\Psi_2^*(\lambda,z)):=(\mathcal{U}_{11}(\lambda,z),\mathcal{U}_{12}(\lambda,z)),
\end{gather}
then $Y^*(\lambda,z)$ def\/ines a fundamental solution of \eqref{eq:scalarlaxpair}, on the open environment $\Lambda\times Z$ of $(\lambda_0,a)$. Furthermore, note that, by equation \eqref{eq:Vconvergence},
\begin{gather*}
A(\lambda,a)=\begin{pmatrix}
0 & 1\\
V(\lambda;a,b,\theta) & 0
\end{pmatrix},
\end{gather*}
with the potential $V(\lambda;a,b,\theta)$ as def\/ined in \eqref{eq:anharmoniccent}. In particular $Y^*(\lambda,a)$ denotes a local fundamental solution on $\Lambda$ of the anharmonic oscillator \eqref{eq:anharmoniccent}.

Considering part (2), if $Y(\lambda,z)$ and $\widetilde{Y}(\lambda,z)$ are local fundamental solutions, with non-empty intersection of domains, then there exists a unique meromorphic matrix $C=C(\lambda,z)$ specif\/ied by equation~\eqref{eq:connectionmatrix} on this intersection. By equations~\eqref{eq:anharmonicgen} and~\eqref{eq:deform}, we have respectively $C_\lambda=0$ and $C_z=0$, hence $C\in {\rm GL}_2(\mathbb{C})$ is constant on this intersection.

We conclude that any local solution $\Psi(\lambda,z)$ of \eqref{eq:scalarlaxpair}, extends uniquely to a global single-valued meromorphic function on $\mathbb{C}^\infty\times \mathbb{C}$. Suppose now that $\om(z)$ has a pole with residue~$-1$ at $z=a$, say corresponding expansion takes the form \eqref{eq:laurentpole} with $\epsilon=-1$, and let $Y^*=(\Psi_1^*,\Psi_2^*)$ denote a local analytic fundamental solution of~\eqref{eq:scalarlaxpair}, as constructed in \eqref{eq:Ystar}, around $(\lambda_0,a)$, for any $\lambda_0\in C^\infty$. Then there must exist constants $c_1,c_2\in\mathbb{C}$ such that $\Psi=c_1\Psi_1^*+c_2\Psi_2^*$, in particular $\Psi(\lambda,z)$ is indeed analytic at $(\lambda,z)=(\lambda_0,a)$ and $\psi(\lambda)=\Psi(\lambda,a)$ def\/ines a solution of the anharmonic oscillator~\eqref{eq:anharmoniccent}. The behaviour of solutions of~\eqref{eq:scalarlaxpair} near zeros and poles with~$+1$ residue of $\omega(z)$, is proven similarly.
\end{proof}

\begin{lem}\label{lem:lineardeplocalglobal}
Given solutions $\Psi(\lambda,z)$ and $\widetilde{\Psi}(\lambda,z)$ of~\eqref{eq:scalarlaxpair}, suppose there exists $z^*\in\mathbb{C}$, not a~pole with residue~$+1$ of~$\om(z)$, and $c\in\mathbb{C}$ such that
	\begin{gather}\label{eq:lineardependence}
	\Psi(\lambda,z)=c\widetilde{\Psi}(\lambda,z),\qquad \lambda\in\mathbb{C}^\infty,
	\end{gather}
	holds at $z=z^*$, then~\eqref{eq:lineardependence} holds globally in $z$.

 Furthermore, suppose $Y(\lambda,z)$ and $\widetilde{Y}(\lambda,z)$ are fundamental solutions of~\eqref{eq:scalarlaxpair}, then there exists a constant matrix $C\in {\rm GL}_2(\mathbb{C})$ such that
	\begin{gather}\label{eq:connectionmatrix}
	\widetilde{Y}(\lambda,z)=Y(\lambda,z)C.
	\end{gather}
\end{lem}
\begin{proof}
	Using the deformation equation \eqref{eq:deform}, one inductively shows that
\begin{gather*}
\frac{\partial^n\Psi}{\partial z^n}(\lambda,z^*)=c \frac{\partial^n\widetilde{\Psi}}{\partial z^n}(\lambda,z^*), \qquad \text{for all $n\in\mathbb{N}$}.
\end{gather*} Hence \eqref{eq:lineardependence} holds locally in $z$ around $z=z^*$, and hence globally by Lemma \ref{lem:behavioursolutions}.
	The second part follows from item~(2) in the proof of Lemma~\ref{lem:behavioursolutions}.
\end{proof}

We write $\Psi\equiv \Phi$, for solutions of \eqref{eq:scalarlaxpair}, if\/f they are linearly dependent.
Ablowitz et al.~\cite{fokasmugan1988} show that,
for $k\in\mathbb{Z}_4$, there exists a unique subdominant solution of \eqref{eq:scalarlaxpair} in $\Omega_k$, normalised by means of the asymptotic characterisation
\begin{gather}\label{eq:asympinfscalar}
\Psi_k(\lambda,z)=\begin{cases} \big(1+\mathcal{O}\big(\lambda^{-1}\big)\big)e^{ g(\lambda,z)}\big(\frac{1}{\lambda}\big)^{\theta_\infty-\tfrac{1}{2}} & \text{if $k$ odd,}\\
\big({-}\tfrac{1}{2}u+\mathcal{O}\big(\lambda^{-1}\big)\big)e^{-g(\lambda,z)}\big(\frac{1}{\lambda}\big)^{\tfrac{1}{2}-\theta_\infty} & \text{if $k$ even,}
\end{cases}
\end{gather}
as $\lambda\rightarrow \infty$ in $\Omega_k$, locally uniformly in $z$ away from zeros and poles of $\om(z)$. The Stokes phenomenon near $\lambda=\infty$ in \eqref{eq:scalarlaxpair} now translates to the exponentially small jumps
\begin{subequations}\label{eq:stmultscalarlaxpair}
	\begin{gather}
	\Psi_{k+1}(\lambda,z)=\Psi_{k-1}(\lambda,z)+s_k\Psi_{k}(\lambda,z), \qquad k=1,2,\\
	\Psi_{0}(\lambda,z)=-e^{-2\pi i \theta_\infty}\big(\Psi_{2}(e^{2\pi i}\lambda,z)+s_3\Psi_{3}\big(e^{2\pi i}\lambda\big),z\big),\\
	\Psi_{1}(\lambda,z)=-e^{2\pi i \theta_\infty}\Psi_{3}\big(e^{2\pi i}\lambda,z\big)+s_0\Psi_{0}(\lambda,z).
	\end{gather}
\end{subequations}
for a unique $s_k=s_k(z)\in\mathbb{C}$, called the $k$th Stokes multiplier. From Lemma \ref{lem:lineardeplocalglobal}, we immediately obtain that the Stokes multipliers $s_k=s_k(z)$ are constant with respect to $z$.

\begin{subequations}\label{eq:scalarpsi10} Considering \eqref{eq:scalarlaxpair} near $\lambda=0$, we restrict our discussion to $\theta_0\notin -\tfrac{1}{2}\mathbb{N}$. Kapaev \cite{kapaev1996,kapaev1998} shows the existence of solutions $\Psi_1^0(\lambda,z)$ and $\Psi_2^0(\lambda,z)$ of \eqref{eq:scalarlaxpair} with Frobenius expansions in $\lambda$, at $\lambda=0$,
	\begin{gather}
	\Psi_+(\lambda,z)=\kappa\lambda^{\tfrac{1}{2}+\theta_0} (1+\lambda F_1(\lambda;z) ), \\
	\Psi_-(\lambda,z)=\kappa^{-1}\lambda^{\tfrac{1}{2}-\theta_0}\left(\frac{u\om}{4\theta_0}+\lambda F_2(\lambda;z)\right)+j\log(\lambda)\Psi_+(\lambda,z), \label{eq:scalarpsi10minus}
	\end{gather}
\end{subequations}
where $\kappa=\kappa(z)$ satisf\/ies
\begin{gather}\label{eq:kappadefi}
\frac{\kappa_z}{\kappa}=-2\frac{v}{\om},
\end{gather}
 and $F_1(\lambda;z)$ and $F_2(\lambda;z)$ entire in $\lambda$, with $j=0$ in the non-resonant case $\theta_0\notin \tfrac{1}{2}\mathbb{Z}$.

In the resonant case $\theta_0\in \tfrac{1}{2}\mathbb{N}^*$, we have
\begin{gather*}
-e^{-2\pi i \theta_0}\Psi_-\big(e^{2\pi i}\lambda,z\big)=\Psi_-(\lambda,z)+j2\pi i \Psi_+(\lambda,z),
\end{gather*}
hence $j=j(z)$ is independent of $z$, by Lemma~\ref{lem:lineardeplocalglobal}. We call $\lambda=0$ an apparent singularity if $j=0$. The main objective of this section, is to prove the following two propositions.
\begin{pro}\label{pro:hermitecharacterisation}
	Fix $m,n\in\mathbb{N}^*$, then the three families of Hermite rationals are characterised via the scalar GJM Lax pair \eqref{eq:scalarlaxpair}, as follows.
	\begin{enumerate}[label={\rm H.\Roman*:}, ref={\rm H.\Roman*}]\itemsep=0pt
		\item Considering parameter values $\theta_0=\tfrac{1}{2}n$ and $\theta_\infty=m+1+\tfrac{1}{2}n$, a solution $\om\in \mathcal{W}_\theta$, recall definition~\eqref{eq:defiwtheta}, equals the rational solution $\om_{m,n}^{({\rm I})}$, if and only if the scalar GJM Lax pair~\eqref{eq:scalarlaxpair} has an apparent singularity at $\lambda=0$,
		and the subdominant solution at $\lambda=\infty$ in $\Omega_0$, is also subdominant at $\lambda=0$, i.e.,
		$\Psi_0(\lambda,z)\equiv \Psi_{+}(\lambda,z)$.
		
		Furthermore, in such case, it turns out that $\Psi_0(\lambda,z)\equiv \Psi_2(\lambda,z)$ and
		\begin{gather}\label{eq:h1psi0}
		\Psi_0(\lambda,z)=\lambda^{\frac{1}{2}(n+1)}e^{-g(\lambda,z)}P_1(\lambda,z),
		\end{gather}
		where $P_1(\lambda,z)$ is a polynomial in $\lambda$ of degree $m$, with constant term $P_1(0,z)=c\kappa(z)$, for some $c\in\mathbb{C}^*$,
		and $g(\la,z)$ is as defined in~\eqref{eq:gfunction}.		\label{H1general}

		\item Considering parameter values $\theta_0=\tfrac{1}{2}m$ and $\theta_\infty=1-\tfrac{1}{2}m-n$, a solution $\om\in \mathcal{W}_\theta$ equals the rational solution $\om_{m,n-1}^{({\rm II})}$, if and only if the scalar GJM Lax pair~\eqref{eq:scalarlaxpair} has an apparent singularity at $\lambda=0$,
		and the subdominant solution at $\lambda=\infty$ in $\Omega_1$, is also subdominant at $\lambda=0$, i.e.,
		$\Psi_1(\lambda,z)\equiv \Psi_{+}(\lambda,z)$.
	
		Furthermore, in such case, it turns out that $\Psi_1(\lambda,z)\equiv \Psi_3(\lambda,z)$ and
		\begin{gather*}
		\Psi_1(\lambda,z)=\lambda^{\frac{1}{2}(m+1)}e^{g(\lambda,z)}P_2(\lambda,z),
		\end{gather*}
		where $P_2(\lambda,z)$ is a polynomial in $\lambda$ of degree $n-1$ with constant term $P_2(0,z)=c\kappa(z)$, for some $c\in\mathbb{C}^*$,
		and $g(\la,z)$ is as defined in \eqref{eq:gfunction}.				\label{H2general}

\item Considering parameter values $\theta_0=\tfrac{1}{2}(m+n)$ and $\theta_\infty=\tfrac{1}{2}(n-m+2)$, a solution $\om\in \mathcal{W}_\theta$ equals the rational solution $\om_{m-1,n}^{({\rm III})}$, if and only if subdominant solutions near $\lambda=\infty$, of the scalar GJM Lax pair~\eqref{eq:scalarlaxpair}, in
		opposite Stokes sectors are linearly dependent, i.e., $\Psi_0(\lambda,z)\equiv \Psi_2(\lambda,z)$ and $\Psi_1(\lambda,z)\equiv \psi_3(\lambda,z)$.			
		Furthermore, in such case, it turns out that $\lambda=0$ is an apparent singularity and
		\begin{gather*}
		\Psi_0(\lambda)=\lambda^{-\frac{m+n-1}{2}}e^{-g(\lambda,a)}P_3(\lambda,z),\qquad \Psi_1(\lambda)=\lambda^{-\frac{m+n-1}{2}}e^{g(\lambda,a)}Q_3(\lambda,z),
		\end{gather*}
		where $P_3(\lambda,z)$ and $Q_3(\lambda,z)$ are polynomials in $\lambda$ of degree $n$ and $m-1$ respectively, with constant terms $P_3(0,z)=c_1 u\om\kappa^{-1}$ and $Q_3(0,z)=c_2 u\om\kappa^{-1}$, for some $c_1,c_2\in\mathbb{C}^*$,
		and~$g(\la,z)$ is as defined in~\eqref{eq:gfunction}. 	\label{H3general}
	\end{enumerate}
\end{pro}
\begin{rem}\label{rem:simple}
Each of the polynomials $P_1(\lambda,z)$, $P_2(\lambda,z)$, $P_3(\lambda,z)$ and $Q_3(\lambda,z)$ in Proposi\-tion~\ref{pro:hermitecharacterisation}, has only nonzero simple roots in the complex $\lambda$-plane.
\end{rem}

\begin{pro}\label{pro:okamotoanharmonicgeneral}
Fix $m,n\in\mathbb{Z}$, let $\theta_0=\tfrac{1}{2}n-\tfrac{1}{6}$ and $\theta_\infty=\tfrac{1}{2}(2m+n+1)$, then a solution $\om\in \mathcal{W}_\theta$ equals the rational solution $\widetilde{w}_{m,n}(z)$, if and only if the Stokes data of the scalar GJM Lax pair~\eqref{eq:scalarlaxpair} satisfy
	\begin{gather*}
	s_0=s_2,\qquad s_1=s_3,\qquad s_0s_1+1=0.
	\end{gather*}
\end{pro}

In order to prove Propositions \ref{pro:hermitecharacterisation} and \ref{pro:okamotoanharmonicgeneral},
we f\/irst discuss the monodromy space and monodromy mapping induced by the scalar GJM Lax pair.
There are dif\/ferent cases to be taken care of, the non-resonant and resonant, both of which are necessary to
study rational solutions. We then discuss the monodromy corresponding to the rational solutions. The Okamoto and
Hermite~I and~II families of rational solutions, can be found in the literature~\cite{kapaev1998}, see also~\cite{milne}.
However, the Hermite~III case seems to be missing in the literature; hence we study it in Proposition~\ref{prop:rationalmonodromy} below.
Finally we use these data to prove aforementioned propositions. For sake of later convenience, we introduce the following matrices
\begin{gather*}
\sigma_3=\begin{pmatrix}
1 & 0\\
0 & -1
\end{pmatrix},\qquad \sigma_+=\begin{pmatrix}
0 & 1 \\
0 & 0
\end{pmatrix}.
\end{gather*}

\subsubsection{Monodromy of linear problem}
We discuss the monodromy data of the scalar GJM Lax pair, following Ablowitz et al.~\cite{fokasmugan1988} and Kapaev~\cite{kapaev1996,kapaev1998} closely. Let us def\/ine fundamental solutions of the scalar GJM Lax pair \eqref{eq:scalarlaxpair} near $\lambda=\infty$,
\begin{alignat*}{3}
& Y_0(\lambda,z) =\begin{pmatrix}
-e^{2\pi i \theta_\infty}\Psi_3(e^{2\pi i}\lambda,z) & \Psi_0(\lambda,z)
\end{pmatrix},\qquad &&
Y_2(\lambda,z) =\begin{pmatrix}
\Psi_1(\lambda,z) & \Psi_2(\lambda,z)
\end{pmatrix},& \\
& Y_1(\lambda,z) =\begin{pmatrix}
\Psi_1(\lambda,z) & \Psi_0(\lambda,z)
\end{pmatrix},\qquad &&
Y_3(\lambda,z)=\begin{pmatrix}
\Psi_3(\lambda,z) & \Psi_2(\lambda,z)
\end{pmatrix},&
\end{alignat*}
and a fundamental solution near $\lambda=0$,
\begin{gather*}
Y^0(\lambda,z)=\begin{pmatrix}
\Psi_+(\lambda,z) & \Psi_-(\lambda,z)
\end{pmatrix}.
\end{gather*}
By Lemma \ref{lem:lineardeplocalglobal}, there exist constant matrices $S_0,S_1,S_2,S_3,E\in {\rm GL}_2(\mathbb{C})$ such that
\begin{gather}
Y_{k+1}(\lambda,z) =Y_{k}(\lambda,z)S_k,\qquad k=0,1,2,\nonumber\\
Y_{0}(\lambda,z) =-Y_{3}\big(e^{2\pi i}\lambda,z\big)S_3e^{2\pi i \theta_\infty \sigma_3},\nonumber\\
Y_{0}(\lambda,z) =Y^{0}(\lambda,z)E.\label{eq:connection}
\end{gather}
Given any solutions $\Psi$ and $\Phi$ of the scalar GJM Lax pair \eqref{eq:scalarlaxpair}, it follows by direct calculation that there exists a $c\in\mathbb{C}$ such that their Wronskian equals
\begin{gather*}
\Psi(\lambda,z)\Phi_\lambda(\lambda,z)-\Psi_\lambda(\lambda,z)\Phi(\lambda,z)=c u(z)\big(\lambda-\tfrac{1}{2}\om(z)\big).
\end{gather*}
It is easily seen that $Y_{0}(\lambda,z)$ and $Y^{0}(\lambda,z)$ have identical Wronskian, given by $c=1$ in the above formula.
Therefore, by~\eqref{eq:connection}, the connection matrix $E$ has unit determinant. For $k\in\mathbb{Z}_4$, we call $S_k$ the $k$-th Stokes matrix, which, by~\eqref{eq:stmultscalarlaxpair}, equals
\begin{gather}\label{eq:stconst}
S_k=\begin{pmatrix}
1 & 0\\
s_k & 1
\end{pmatrix}\quad \text{if $k$ is even,}\qquad S_k=\begin{pmatrix}
1 & s_k\\
0 & 1
\end{pmatrix}\quad \text{if $k$ is odd.}
\end{gather}
Writing $Y^0(e^{2\pi i}\lambda,z)=Y^0(\lambda,z)M_0$, we have the semi-cyclic relation
\begin{gather}\label{eq:semi}
-S_0S_1S_2S_3e^{2\pi i \theta_\infty\sigma_3}=E^{-1}M_0^{-1}E,
\end{gather}
which, by taking traces, implies
\begin{gather}\label{eq:stcubic}
(1+s_1s_2)e^{2\pi i \theta_\infty}+(s_0s_3+(1+s_2s_3)(1+s_0s_1))e^{-2\pi i \theta_\infty}=2 \cos{2\pi \theta_0},
\end{gather}
as $M_0$ is given explicitly by
\begin{gather}\label{eq:M0defi}
M_0=-e^{2\pi i \theta_0\sigma_3}\quad \text{if $\theta_0\notin\tfrac{1}{2}\mathbb{Z}$,} \qquad
M_0=(-1)^{n+1}\left(I+2 j\pi i\sigma_+\right)\quad \text{if $\theta_0= \tfrac{1}{2}n\in \tfrac{1}{2}\mathbb{N}^*$.}
\end{gather}
Now the Stokes data $\mathbf{s}=(s_0,s_1,s_2,s_3)$ and connection matrix $E$ depend on the choice of solution $\om \in \mathcal{W}_\theta$, and auxiliary functions $u$ and $\kappa$, characterised by \eqref{eq:udefi} and~\eqref{eq:kappadefi} respectively. Dif\/ferent choices $\widetilde{u}=\alpha u$ and $\widetilde{\kappa}=\beta \kappa$, with $\alpha,\beta\in\mathbb{C}^*$, lead to a change of monodromy data given by
\begin{gather}
\widetilde{\mathbf{s}} =\big(\alpha^{-1} s_0,\alpha s_1,\alpha^{-1} s_2,\alpha s_3\big), \label{eq:stokesaction}\\
\widetilde{E} =\begin{pmatrix}
\beta^{-1} E_{11} & \alpha\beta^{-1} E_{12}\\
\alpha^{-1} \beta E_{21} & \beta E_{22}
\end{pmatrix}.\label{eq:connectionaction}
\end{gather}
Let $M_\theta^s$ denote space obtained by cutting $\big\{\mathbf{s}\in\mathbb{C}^4\big\}$ with respect to \eqref{eq:stcubic}.
\begin{def*}
 In the non-resonant case $\theta_0\notin \tfrac{1}{2}\mathbb{Z}$, the monodromy space $\mathcal{M}_\theta$ is def\/ined as the quotient of $M_\theta^s$ with respect to the action given in~\eqref{eq:stokesaction}. We denote the corresponding monodromy mapping by
 \begin{gather*}
 \mathcal{T}_\theta\colon \ \mathcal{W}_\theta\rightarrow \mathcal{M}_\theta,
 \end{gather*}
 which sends a solution of ${\rm P}_{\rm IV}(\theta)$ to corresponding orbit of Stokes data of \eqref{eq:scalarlaxpair}.
\end{def*}
\begin{pro}\label{pro:injectivenonresonant}
Let $\theta_0\notin \tfrac{1}{2}\mathbb{Z}$, then the monodromy mapping $\mathcal{T}_\theta$ is injective.
\end{pro}
\begin{proof}
	See Ablowitz et al.~\cite{fokasmugan1988}.
\end{proof}

In the resonant case $\theta_0\in \tfrac{1}{2}\mathbb{N}^*$, the dominant solution
$\Psi_-(\lambda,z)$ at $\lambda=0$, is no longer uniquely specif\/ied by the asymptotic expansion in
\eqref{eq:scalarpsi10minus} as one can add arbitrary multiples of the subdominant solution $\Psi_+(\lambda,z)$ to it.
This amounts to arbitrary right multiplication of the fundamental solution at $\lambda=0$,
\begin{gather*}
\widetilde{Y}^0=Y^0(I-r\sigma_+),\qquad r\in\mathbb{C},
\end{gather*}
which correspondingly transforms $E$ as $\widetilde{E}=(I+r\sigma_+)E$. Invariant under this transformation, and the action induced by changing $\kappa$ in \eqref{eq:connectionaction}, is the quantity
\begin{gather}\label{eq:iEdefi}
i_E:=E_{22}/E_{21}\in\mathbb{P}^1.
\end{gather}
\begin{lem} \label{lem:trivialmon}
	Let $\theta_0\in\tfrac{1}{2}\mathbb{N}^*$, then $\lambda=0$ is an apparent singularity of~\eqref{eq:1isopiv}, i.e., $j=0$, if and only if the stokes multipliers $\mathbf{s}$ are elements of the one-dimensional submanifold $M_\theta^0$ of $M_\theta^s$, defined by
	\begin{gather}\label{eq:submanifoldtrivial}
	s_0=-s_2 e^{2\pi i (\theta_\infty+\theta_0)}, \qquad s_1=-s_3 e^{-2\pi i (\theta_\infty+\theta_0)}, \qquad 1+s_1s_2= e^{-2\pi i (\theta_\infty+\theta_0)}.
	\end{gather}
	Furthermore, in case $j\neq 0$, then the quantity $i_E$ is given explicitly by
	\begin{gather*}
	i_E=e^{-2 \pi i \theta_\infty}\frac{s_1+s_3+s_1s_2s_3}{e^{2 \pi i \theta_\infty}(1+s_1s_2)-e^{2 \pi i \theta_0}},
	\end{gather*}	
	where we note that numerator and denominator of the right-hand side vanish simultaneously if and only if~\eqref{eq:submanifoldtrivial} holds, for $\mathbf{s}\in M_\theta^s$.
\end{lem}
\begin{proof}
	This follows directly from equations \eqref{eq:semi} and \eqref{eq:M0defi}, see also Kapaev \cite{kapaev1998}.
\end{proof}

Note that the action in \eqref{eq:connectionaction} gives
\begin{gather}
\widetilde{i}_E=\alpha i_E,\label{eq:ieaction}
\end{gather}
and we def\/ine
\begin{gather*}
M_\theta=\big(M_\theta^s\setminus M_\theta^0\big)\sqcup \big(M_\theta^0\times \big\{i_E\in \mathbb{P}^1\big\}\big).
\end{gather*}
\begin{def*}
In the resonant case $\theta_0\in \tfrac{1}{2}\mathbb{N}^*$, the monodromy space $\mathcal{M}_\theta$ is def\/ined as the quotient of $M_\theta$ with respect to the action given in~\eqref{eq:stokesaction} and~\eqref{eq:ieaction}. We denote corresponding monodromy mapping by
	\begin{gather*}
	\mathcal{T}_\theta\colon \ \mathcal{W}_\theta\rightarrow \mathcal{M}_\theta,
	\end{gather*}
	which sends a solution of ${\rm P}_{\rm IV}(\theta)$ to corresponding orbit of monodromy data of \eqref{eq:scalarlaxpair}.
\end{def*}
\begin{pro}\label{pro:injectiveresonant}
	Let $\theta_0\in \tfrac{1}{2}\mathbb{N}^*$, then the monodromy mapping $\mathcal{T}_\theta$ is injective.
\end{pro}
\begin{proof}
	See Kapaev \cite{kapaev1998}.
\end{proof}

Finally, the monodromy data, are not only invariant under the ${\rm P}_{\rm IV}$ f\/low, but also under the action of the B\"acklund transformations $\mathcal{R}_1$--$\mathcal{R}_4$, def\/ined in Appendix~\ref{section:backlund}.
\begin{pro}\label{pro:schlesinger}
	For $1\leq i\leq 4$ and $\theta\in\mathbb{C}^2$ with $\theta_0,\theta_0^{(i)}\notin -\tfrac{1}{2}\mathbb{N}$, monodromy data correspon\-ding to solutions are invariant under the B\"acklund transformation $\mathcal{R}_i$, i.e.,
	\begin{gather*}
	\mathcal{T}_{\theta^{(i)}}\circ \mathcal{R}_i=\mathcal{T}_\theta.
	\end{gather*}
	\end{pro}
	\begin{proof}
	See Fokas et al.~\cite{fokasmugan1988}.
	\end{proof}

\subsubsection{Monodromy corresponding to rational solutions}\label{section:monodromyrational}
\begin{pro}\label{prop:rationalmonodromy}
	The monodromy data corresponding to the Hermite I family \eqref{eq:hermitepar1}, are given by
	\begin{gather}\label{eq:monodromydataHI}
	s_0=-s_2\neq 0,\qquad s_1=s_3=0,\qquad i_E=0.
	\end{gather}
	The monodromy data corresponding to the Hermite II family \eqref{eq:hermitepar2}, are given by
	\begin{gather*}
	s_1=-s_3\neq 0,\qquad s_0=s_2=0,\qquad i_E=\infty.
	\end{gather*}
	The monodromy data corresponding to the Hermite III family \eqref{eq:hermitepar3}, are given by
	\begin{gather*}
	s_0=s_1=s_2=s_3=0,\quad i_E\in\mathbb{P}^1\setminus\{0,\infty\}.
	\end{gather*}
	The Stokes data corresponding to the Okamoto family \eqref{eq:okamotorationals}, are given by
	\begin{gather*}
	s_0=s_2\qquad s_1=s_3,\qquad s_0s_1+1=0.
	\end{gather*}
\end{pro}
\begin{proof}
	As to the Okamoto family, see Kapaev \cite{kapaev1998} and Milne et al.~\cite{milne}. In the former paper, Kapaev also handles the Hermite I and
	II cases. Let us consider the Hermite~III case. Because the monodromy data are invariant under B\"acklund transformations $\mathcal{R}_1$--$\mathcal{R}_4$, by Proposition \ref{pro:schlesinger}, we only consider the simple case~\eqref{eq:hermitesimple3}. Then, for any $\alpha\in\mathbb{C}^*$, a solution of \eqref{eq:udefi} is given by $u(z)=\alpha$. Now, 	considering the scalar GJM Lax pair~\eqref{eq:scalarlaxpair}, it follows by direct calculation that a~fundamental solution is given by
	\begin{gather*}
	Y_*(\lambda,z)=\begin{pmatrix}
	e^{g(\lambda,z)} & -\tfrac{1}{2}\alpha e^{-g(\lambda,z)}
	\end{pmatrix}.
	\end{gather*}
	Then, comparison with the asymptotic characterisations \eqref{eq:asympinfscalar}, gives
	\begin{gather*}
	Y_0(\lambda,z)=Y_1(\lambda,z)=Y_2(\lambda,z)=Y_3(\lambda,z)=Y_*(\lambda,z),
	\end{gather*}
	which implies that all Stokes multipliers vanish and $\lambda=0$ is an apparent singularity.
Now, for any $\beta\in\mathbb{C}^*$, a solution of \eqref{eq:kappadefi} is given by $\kappa(z)=\beta z$. It is straightforward to check that
 	\begin{gather*}
 	\Psi_+(\lambda,z)=Y_*(\lambda,z)\cdot \begin{pmatrix}
 	\tfrac{1}{2}\beta\\
 	\alpha^{-1}\beta
 	\end{pmatrix},
 	\end{gather*}
 hence $i_E=-\tfrac{1}{2}\alpha\in\mathbb{P}^1\setminus\{0,\infty\}$.	
\end{proof}

\begin{proof}[Proof of Proposition \ref{pro:hermitecharacterisation}]
Let us consider the Hermite I case \ref{H1general}. Because of the injectivity of the monodromy mapping \ref{pro:injectiveresonant}, to establish the f\/irst part, all we have to show is that the monodromy data corresponding to Hermite I \eqref{eq:monodromydataHI},
are equivalent to \eqref{eq:scalarlaxpair} having an apparent singularity at $\lambda=0$, and $\Psi_0(\lambda,z)\equiv \Psi_{+}(\lambda,z)$. Now suppose that the monodromy data of \eqref{eq:scalarlaxpair} are given by \eqref{eq:monodromydataHI}, then Lemma \ref{lem:trivialmon} shows that $\lambda=0$ is indeed an apparent singularity. The fact that $i_E=0$ readily translates to $\Psi_0(\lambda,z)\equiv \Psi_{+}(\lambda,z)$.

Conversely, suppose $\lambda=0$ is an apparent singularity and $\Psi_0(\lambda,z)\equiv \Psi_{+}(\lambda,z)$, then the latter immediately gives $i_E=0$. Furthermore Lemma \ref{lem:trivialmon} shows
 \begin{gather*}
 s_0=-s_2,\qquad s_1=-s_3,\qquad s_1s_2=0,
 \end{gather*}
 Hence $s_1=0$ or $s_2=0$. Suppose, for the sake of contradiction, that $s_2=0$. Then $\Psi_1(\lambda,z)=\Psi_3(\lambda,z)$ and hence
 \begin{gather*}
 \Psi_1(\lambda,z)\sim e^{g(\lambda,z)}\lambda^{-m-\frac{1}{2}(n+1)},
 \end{gather*}
 as $\lambda\rightarrow \infty$ in $\mathbb{C}$.
 It is now easily seen that
 \begin{gather*}
 f(\lambda,z)=e^{-g(\lambda,z)} \lambda^{\frac{1}{2}(n-1)}\psi_1(\lambda)
 \end{gather*}
 is an entire function satisfying $f(\lambda,z)\sim \lambda^{-m-1}$ as $\lambda\rightarrow \infty$ in $\mathbb{C}$. Such a function does not exist, hence $s_2\neq 0$ and we are left with $s_1=s_3=0$.

 Considering the second part of the Hermite I case \ref{H1general}, as $s_1=0$, we indeed have $\Psi_0(\lambda,z)\equiv \Psi_2(\lambda,z)$. Let us def\/ine $P_1(\lambda,z)$ by equation \eqref{eq:h1psi0}. Then it is easily seen that $P_1(\lambda,z)$ is entire satisfying $P_1(\lambda,z)\sim -\tfrac{1}{2}u(z)\lambda^m$, as $\lambda\rightarrow \infty$ in $\mathbb{C}$, from which it follows that $P_1(\lambda,z)$ is a polynomial in $\lambda$ of degree $m$. The expression for the constant term of $P_1(\lambda,z)$ stems from the asymptotic characterisation of $\Psi_+(\lambda,z)$.
The cases \ref{H2general} and \ref{H3general} follow by a similar line of argument.
\end{proof}

\subsection{Localisation of Lax pair at poles}\label{section:localisation}
Recall that the scalar GJM Lax pair~\eqref{eq:scalarlaxpair} has a regular singular point at $\lambda=0$, an irregular singular point at $\lambda=\infty$ and a further apparent singularity at $\lambda=\tfrac{1}{2}\om$. Now, considering Lemma~\ref{lem:behavioursolutions}, a pole of $\om$ with $-1$ residue is a point where the further apparent singularity merges with the irregular singular point, resulting in an integer jump of one of the exponents of the irregular singular point, as can be seen by comparison of the asymptotic expansions~\eqref{eq:asympinfscalar} and~\eqref{eq:1psikscalar}, in the $k$ is odd case.
In this section, we wish to show that the monodromy of the scalar GJM Lax pair is preserved in such a limit.

\subsubsection{Monodromy of anharmonic oscillator}
Let us reconsider the anharmonic oscillator \eqref{eq:anharmoniccent}, for some f\/ixed $a,b\in\mathbb{C}$. For $k\in \mathbb{Z}_4$,
we def\/ined unique solutions $\psi_k(\lambda)$, subdominant in $\Omega_k$, by \eqref{eq:1psikscalar}.
Let us also recall the Stokes phenomenon of the anharmonic oscillator near $\lambda=\infty$, made explicit by equations \eqref{eq:stmult12} and \eqref{eq:stmult03}, with corresponding Stokes data $\mathbf{s}\in\mathbb{C}^4$.

Similar to equations \eqref{eq:scalarpsi10}, there exist, for $\theta_0\notin -\tfrac{1}{2}\mathbb{N}$, solutions of \eqref{eq:anharmoniccent} enjoying Frobenius expansions near $\lambda=0$ of the form
\begin{gather*}
\psi_+(\lambda)=\lambda^{\tfrac{1}{2}+\theta_0} (1+\lambda f_1(\lambda)),\qquad
\psi_-(\lambda)=\lambda^{\tfrac{1}{2}-\theta_0}\left(-\frac{1}{2\theta_0}+\lambda f_2(\lambda)\right)+j\log(\lambda)\psi_1^0(\lambda),
\end{gather*}
where $f_1(\lambda)$ and $f_2(\lambda)$ entire and $j=0$ in the non-resonant case $\theta_0\notin \tfrac{1}{2}\mathbb{N}^*$. There exists a unique matrix $E\in {\rm GL}_2(\mathbb{C})$, which we call the connection matrix, such that
 \begin{gather*}
 \begin{pmatrix}
 -e^{2\pi i \theta_\infty}\psi_3(e^{2\pi i}\lambda) & \psi_0(\lambda)
 \end{pmatrix}=\begin{pmatrix}
\psi_+(\lambda) & \psi_-(\lambda)
 \end{pmatrix}\cdot E.
 \end{gather*}
 As both fundamental solutions appearing in the above equation have unit Wronskian, the connection matrix has unit determinant.

Let us def\/ine $S_k$ by \eqref{eq:stconst}, def\/ine $M_0$ by \eqref{eq:M0defi}, then equations \eqref{eq:semi} and hence \eqref{eq:stcubic} hold, so $\mathbf{s}\in M_\theta^s$. Furthermore Lemma~\ref{lem:trivialmon} also holds true in this case. We def\/ine $i_E\in\mathbb{P}^1$ by \eqref{eq:iEdefi} if $\theta_0\in\tfrac{1}{2}\mathbb{N}^*$. Finally we def\/ine the monodromy mapping for the anharmonic oscillator \eqref{eq:anharmoniccent},
\begin{gather*}
\mathcal{T}_\theta^o\colon \ \mathbb{C}^2\rightarrow M_\theta,(a,b)\mapsto \mathcal{T}_\theta^o(a,b),
\end{gather*}
where $T_\theta^o(a,b)$ is the orbit corresponding to the monodromy data $\{\mathbf{s},E\}$ of~\eqref{eq:anharmoniccent} within $M_\theta$.

\subsubsection{Localisation and Monodromy}
We now wish to compare the monodromy data of the GJM scalar equation \eqref{eq:scalarlaxpair} and corresponding anharmonic oscillator
\eqref{eq:anharmoniccent}, upon localisation.
\begin{lem}\label{lem:ztoalem}
	Let $\om(z)$ be a solution of ${\rm P}_{\rm IV}$ and fix some $u(z)$ and $\kappa(z)$ satisfying~\eqref{eq:udefi} and~\eqref{eq:kappadefi} respectively. Say $z=a$ is a pole of $\om(z)$ with residue~$-1$ and corresponding coefficient
	$b$ as in \eqref{eq:laurentpole} with $\epsilon=-1$, and let $u_0,\kappa_0\neq 0$ be defined by
	\begin{gather}\label{eq:ukappaatpole}
	u(z)=u_0(z-a)+\mathcal{O}\big((z-a)^2\big),\qquad \kappa(z)=\kappa_0+\mathcal{O} (z-a ), \qquad z\rightarrow a.
	\end{gather}
	Then, for $k\in\mathbb{Z}$, we have
	\begin{gather} \label{eq:psiasymp}
	\Psi_k(\lambda,a)=\begin{cases}
	\psi_{k}(\lambda) & \text{if $k$ is odd,}\\
	\tfrac{1}{2}u_0\psi_{k}(\lambda) & \text{if $k$ is even.}
	\end{cases}
	\end{gather}
	Furthermore
	\begin{gather}\label{eq:psi0asymp}
	\Psi_+(\lambda,a)=\kappa_0\psi_+(\lambda),\qquad
	\Psi_-(\lambda,a)=\tfrac{1}{2}u_0\kappa_0^{-1} \psi_-(\lambda)+c\psi_+(\lambda),
	\end{gather}
	for some $c\in\mathbb{C}$, with $c=0$ in the non-resonant case $\theta_0\notin \tfrac{1}{2}\mathbb{Z}$.
\end{lem}
\begin{proof}
As was proven in Lemma \ref{lem:behavioursolutions}, in the limit $z \to a$, any solution $\Psi(\la,z)$ of \eqref{eq:scalarlaxpair}
converges to a solution $\psi(\la)$ of \eqref{eq:anharmoniccent}, since the potential $V(\la,z)$ of the former converges to the potential $V(\la;a,b,\theta)$
of the latter, see equation \eqref{eq:Vconvergence}.

We f\/irst consider the convergence of Frobenius solutions near $\lambda=0$, as $z\rightarrow a$, see \eqref{eq:psi0asymp}. It follows from Lemma \ref{lem:behavioursolutions}, that the terms $F_1(\lambda;z)$ and $F_2(\lambda;z)$ in equations \eqref{eq:scalarpsi10}, are analytic in $z$ away from poles with $+1$ residue of $\om(z)$. In particular, at $z=a$ equations \eqref{eq:scalarpsi10} reduce to
	\begin{gather*}
	\Psi_+(\lambda,a)=\kappa_0\lambda^{\tfrac{1}{2}+\theta_0}(1+\lambda F_1(\lambda;a)),\\
	\Psi_-(\lambda,a)=\kappa_0^{-1}\lambda^{\tfrac{1}{2}-\theta_0}\left(-\frac{u_0}{4\theta_0}+\lambda F_2(\lambda;a)\right)+j\log(\lambda)\Psi_+(\lambda,a),
	\end{gather*}
from which equations \eqref{eq:psi0asymp} trivially follow.

Establishing the convergence in \eqref{eq:psiasymp} is more involved because the singularity at $\lambda=\infty$ is irregular.
Following \cite[Theorem~4.5]{myphd}, where the same limit is established in the Painlev\'e I case, one def\/ines $\Psi_k(\la,z)$
by means of a linear integral equation of Volterra type, where the kernel $K(\la,\mu;z)$ is expressed in term of the
action integral $e^{\int^\la_\mu \sqrt{V(\nu,z)}{\rm d}\nu}$, see \cite[equation~(4.17)]{myphd}.
From the convergence of the kernel in the limit $\lim\limits_{z \to a}K(\la,\mu;z)=K(\la,\mu;a)$, trivial but tedious estimates
lead to the proof of the convergence of the solutions $\Psi_k(\la,z)$.
\end{proof}

We def\/ine the Laurent mapping
\begin{gather}\label{eq:laurentmapping}
\mathcal{L}_\theta^\pm\colon \ \mathbb{C}^2\mapsto \mathcal{W}_\theta, \qquad (a,b)\mapsto \mathcal{L}_\theta^\pm(a,b),
\end{gather}
where $\mathcal{L}_\theta^\pm(a,b)$ denotes the meromorphic continuation of \eqref{eq:laurentpole} with $\epsilon=\pm 1$.
\begin{pro}\label{prop:commute}
	For any parameter values $\theta\in\mathbb{C}^*\times\mathbb{C}$,
	\begin{gather*}
	\mathcal{T}_\theta\circ \mathcal{L}_\theta^{\minus}=\mathcal{T}_\theta^o,
	\end{gather*}
	where $T_\theta$ is the monodromy mapping of the scalar GJM Lax pair, $T_\theta^o$ is the monodromy mapping of the anharmonic oscillator and $\mathcal{L}_\theta^-$ is the Laurent mapping~\eqref{eq:laurentmapping}.
\end{pro}
\begin{proof}
	Take any $(a,b)\in\mathbb{C}^2$, then $\omega:=\mathcal{L}_\theta^-(a,b)$ is a solution of ${\rm P}_{\rm IV}(\theta)$ with Laurent expansion~\eqref{eq:laurentpole} about $z=a$ with $\epsilon=-1$. Let us take some $u(z)$ and $\kappa(z)$ satisfying \eqref{eq:udefi} and \eqref{eq:kappadefi} respectively, which we may assume are normalised such that $u_0=\kappa_0=1$ in \eqref{eq:ukappaatpole}. Then, using Lemma \ref{lem:ztoalem}, it is easy to see that the the Stokes multipliers, and quantity $i_E$ in the resonant case, are conserved as $z\rightarrow a$.
\end{proof}

\subsection{Exact characterisation of poles}\label{section:proofofcharacterisations}
\begin{thm}\label{thm:exactchar}
	Fix $\theta\in\mathbb{C}^*\times\mathbb{C}$ and a solution $\om$ of ${\rm P}_{\rm IV}(\theta)$, then $z=a$ is a pole with residue $-1$ of $\om$, if and only if there exists $b\in\mathbb{C}$ such that the monodromy of the anharmonic oscillator~\eqref{eq:anharmoniccent}, coincides with the monodromy of the scalar GJM Lax pair corresponding to the solution~$\om$, i.e., $\mathcal{T}_\theta^o(a,b)=\mathcal{T}_\theta(\om)$. In such case~$b$ turns out to be the coefficient in~\eqref{eq:laurentpole}.
\end{thm}
\begin{proof}
We have already established the ``only if" part. As to its converse, suppose $b\in\mathbb{C}$ is such that $\mathcal{T}_\theta^o(a,b)=\mathcal{T}_\theta(\om)$. Let $\widetilde{\om}$ be the solution of ${\rm P}_{\rm IV}(\theta)$ def\/ined by $\widetilde{\om}=\mathcal{L}_\theta^{\minus}(a,b)$. Then we know, by Proposition \ref{prop:commute}, that $\mathcal{T}_\theta(\widetilde{\om})=\mathcal{T}_\theta^o(a,b)=T_\theta(\om)$. Hence we have $\om=\widetilde{\om}$, by the injectivity of the monodromy mapping, see Propositions \ref{pro:injectivenonresonant} and \ref{pro:injectiveresonant}. In particular $z=a$ is indeed a pole with residue $-1$ of $\om$, and $b$ is the coef\/f\/icient in \eqref{eq:laurentpole}.
\end{proof}

Note that Theorem \ref{thm:okamotoanharmonic} is a direct consequence of Proposition \ref{prop:rationalmonodromy} and Theorem \ref{thm:exactchar}. However, Theorem \ref{thm:hermiteanharmonic} still requires some work.
\begin{proof}[Proof of Theorem \ref{thm:hermiteanharmonic} and Remark \ref{rem:simple}]
	We f\/ix $m,n\in\mathbb{N}^*$ and let us consider the equivalence with \ref{H1}. We set $\theta_0=\tfrac{1}{2}n$ and $\theta_\infty=m+1+\tfrac{1}{2}n$. From Proposition~\ref{prop:rationalmonodromy} and Theorem~\ref{thm:exactchar}, we conclude that, $z=a$ is a zero of $H_{m,n}(z)$, if and only if, there exists $b\in\mathbb{C}$ such that the monodromy of the anharmonic oscillator~\eqref{eq:anharmoniccent} is given by~\eqref{eq:monodromydataHI}. The latter statement is easily seen to be equivalent to \ref{H1}, by an argument identical to the proof of the f\/irst part of~\ref{H1general} in Proposition~\ref{pro:hermitecharacterisation}. Furthermore, in case $z=a$ is indeed a zero of $H_{m,n}(z)$, then comparison of~\ref{H1} and~\ref{H1general}, gives $P_1(\lambda,a)=\tfrac{1}{2}u_0p_1(\lambda)$, by Lemma~\ref{lem:ztoalem}.
		The roots of $p_1(\lambda)$ are necessarily simple. Now $P_1(\lambda,z)$ might a priori have a double root at $\lambda=\tfrac{1}{2}\om(z)$, for special values of~$z$. However, by Lemma~\ref{lem:lineardeplocalglobal}, this would imply that $P_1(\lambda,z)$ has a double root at $\lambda=\tfrac{1}{2}\om(z)$, for all values of~$z$, not equal to a zero or pole of $\om(z)$. Indeed, the latter follows from the fact that there exists an up to scalar multiplication unique solution of~\eqref{eq:scalarlaxpair}, which has a double root at $\lambda=\tfrac{1}{2}\om(z)$, for all values of $z$, not equal to a zero or pole of~$\om(z)$.
		
		Because of the identity $P_1(\lambda,a)=\tfrac{1}{2}u_0p_1(\lambda)$, this would in turn imply that $p_1(\lambda)$ is of degree at most $m-2$, in contradiction with the fact that $p_1(\lambda)$ is necessarily of degree $m-1$.
		We conclude that $P_1(\lambda,z)$, does not has a double root at $\lambda=\tfrac{1}{2}\om(z)$, for all values of $z$, not a zero or pole of $\om(z)$. In particular, for any such $z$, all the roots of $P_1(\lambda,z)$ are simple and nonzero.
		
		The equivalence with \ref{H2} and \ref{H3} is shown by the same line of argument.
	\end{proof}

\section{Nevanlinna functions and poles of rational solutions}\label{section:nevanlinna}
We have showed that poles of rational solutions are in bijection with anharmonic oscillators having particular prescribed monodromy.
Here we show that these oscillators naturally def\/ine Riemann surfaces which are inf\/initely-sheeted coverings of the
Riemann sphere uniformised by meromorphic functions.

More precisely we introduce two discrete classes of Riemann surfaces and we show that they classify roots of the generalised Hermite polynomials. We then derive a slightly weaker characterisation for roots of Okamoto polynomials.

Our approach is based on the seminal work by Nevanlinna~\cite{nevanlinna32} and Elfving~\cite{elfving1934}, which has lately been revived and found useful in modern applications, see, e.g.,~\cite{eremenko,pibelyi}.

The following Def\/inition is instrumental to the analysis below.
\begin{Def}
Let $\bb{K}$ be a one dimensional complex manifold.

A holomorphic function $f\colon \bb{K} \to \bb{P}^1$ (that is a meromorphic function) is called a branched covering of the sphere if there
exists a f\/inite subset $S \subset \bb{P}^1$ such that the restriction
\begin{gather*}
f\colon \ \bb{K}\setminus f^{-1}(S) \to \bb{P}^1 \setminus S
\end{gather*}
is a topological covering.

The minimal set $B$ among all the sets satisfying the above property is called the branching locus of $f$.
\end{Def}
We remark that in what follows we will always restrict to the case when $\bb{K}$ is either the complex plane or the Riemann sphere.

\subsection{Anharmonic oscillators and Nevanlinna theory}
Before tackling our characterisation, we brief\/ly sketch the theory of the Riemann surfaces associated with anharmonic oscillators
and suggest \cite{pibelyi} for a complete introduction. Consider an anharmonic oscillator
\begin{gather}\label{eq:anharmonic}
\psi''(\lambda)=r(\lambda)\psi(\lambda),
\end{gather}
where $r$ is some rational function, with at most double poles in the complex plane. Then, for any two linearly independent solutions $\{\psi,\phi\}$ of~\eqref{eq:anharmonic}, the function $f=\psi/\phi \colon \bb{C} \to \bb{P}^1$ is locally invertible at any point $\la$, unless~$\la$ is a double pole of~$r$.

Suppose $\la^*$ is a double pole and locally $r(\la)=\frac{n^2-1}{4(\la-\la^*)^2}+\mathcal{O}(\frac{1}{\la-\la^*})$ for some $n \in \bb{C}$.
Then the corresponding indices of \eqref{eq:anharmonic} are $ \frac{1\pm n}{2}$.

If $n\in \bb{N}$ and the Fuchsian singularity is apparent, then $f$ is locally single valued, but in this case $\la^*$ is a critical point of~$f$ of order $n-1$. If, on the contrary, $n \in \bb{N}$ but the Fuchsian singularity $\la^*$ is not apparent, or $n \notin \bb{N}$, then $f$ is a multivalued function in the neighbourhood of~$\la^*$.

The point at inf\/inity is in general an irregular singularity. Suppose $r(\la)=\la^M+ \mathcal{O}\big(\la^{M-1}\big)$ as $\la \to \infty$, then equation~\eqref{eq:anharmonic} admits $M+2$ Stokes sectors
\begin{gather*}
 \Omega_k^M=\left\lbrace \left| \arg\lambda - \frac{2 \pi k}{M+2} \right| < \frac{\pi}{2M+2} \right\rbrace , \qquad k=0,\dots,M+1 .
\end{gather*}
Each Stokes sector can be thought as a critical point of inf\/inite order~-- technically a logarithmic direct transcendental singularity~\cite{eremenko2004}. Indeed, in each Stokes sector, the function~$f$ has a well-def\/ined asymptotic value
\begin{gather*}
w_k:=\lim_{\lambda\rightarrow \infty,\, \lambda\in \Omega_k^M}{f(\lambda)}\in \mathbb{P}^1,
\end{gather*}
while all its derivatives vanish exponentially fast; here the limit must be taken along curves not tangential to the boundary of the Stokes sector. These asymptotic values can be directly computed from the Stokes multipliers but we do not need the general relation here \cite{masoerobethe,nevanlinna32}.

Now, $f$ is only unique up to composition by M\"obius transforms. Indeed if $\{\psi^*,\phi^*\}$ is another choice of linearly independent solutions of \eqref{eq:anharmonic}, then
\begin{gather}\label{eq:mobius}
\widetilde{f}=m\circ f,\qquad m(z)=\frac{az+b}{cz+d},\qquad \begin{pmatrix}\psi^*\\\phi^*
\end{pmatrix}=\begin{pmatrix}
a & b\\
c & d
\end{pmatrix}\begin{pmatrix}
\psi\\
\phi
\end{pmatrix}.
\end{gather}

Summing up, to any anharmonic oscillator \eqref{eq:anharmonic}, such that all poles of the potential $r$ in the plane are apparent Fuchsian singularities,
we can associate a branched covering of the sphere, up to automorphism of the target sphere, namely up to M\"obius equivalence. The branching locus $B$ is the union of the critical values and asymptotic values.

In turn, one can recover the potential $r$ from $f$, by means of the Schwarzian derivative,
\begin{gather*}
r=-\frac12 \mathcal{S}(f),\qquad \mathcal{S}(f):=\frac{f'''}{f'}-\frac{3}{2}\left(\frac{f''}{f'}\right)^2,
\end{gather*}
which indeed is invariant under composition by M\"obius transformations~\eqref{eq:mobius}.

Finally, we notice that two meromorphic functions $f$ and $f^*$, of the kind described above, are topologically equivalent coverings of the sphere
 if and only if there exist $c\neq0$ and $d$ such that $f^*(\la)=f(c \la+d)$.

The question that remains is whether all branched coverings of the sphere as above can be obtained by means of anharmonic oscillators.
This was positively settled by Elfving, as the following theorem shows.
\begin{thm}[Nevanlinna, Elfving]\label{thm:prenevanlinna}
Let $f$ be a function with $p <\infty$ transcendental singularities and $m <\infty$ critical points, lying over $q \geq 2$ points. Then all transcendental singularities are direct, logarithmic and $r=-\frac12 \mathcal{S}(f)$ is a rational	function of degree less or equal to $p-2+2m$. 	In particular, suppose the function $f$ has no critical points, i.e., $m=0$. Then~$r(z)$ is a polynomial of degree~$p-2$.
\end{thm}
	\begin{proof}
		For the case $m=0$ see \cite{nevanlinna32}. For the general case, see~\cite{elfving1934}.
	\end{proof}

\subsubsection{Combinatorics of branched coverings}\label{sub:combinatorics}
In order to present our results, we brief\/ly introduce the concept of a line complex corresponding to a branched covering of the sphere,
and refer the reader to \cite{elfving1934,pibelyi,nevanlinna32} for the precise def\/initions. Let $f$ be a branched covering of the sphere,
with ordered branching locus $\{b_1,\ldots,b_n\}$, and let us f\/ix an oriented Jordan curve $\gamma$, passing through all
of the branching points, respecting the particular ordering. This curve divides the sphere into an inner and outer polygon,
with common sides given by the arcs $(b_1,b_2),\ldots,(b_n,b_1)$, see Fig.~\ref{fig:linecomplex}.

We choose a point $P_i$ in the inner polygon and a point $P_o$ in the outer polygon, as well as for each
$1\leq k\leq n$, an analytic line (i.e., curve) $l_k$ connecting $P_i$ and $P_o$, going only through the side $(b_k,b_{k+1})$
and only once, with convention $b_{n+1}=b_1$. The line complex is the graph, given by the inverse image under $f$
of the union of lines $l_k$, where the set of vertices $V$ is given by the (disjoint) union of $V_i:=f^{-1}(P_i)$ and $V_o:=f^{-1}(P_o)$.
Note that the set of vertices $V$ is at most countable and the vertices do not accumulate in~$\mathbb{C}$. Furthermore the line complex is bipartite with respect to the partition $V=V_i\sqcup V_o$. We colour the edges of the graph by means of uniquely assigning the line corresponding to each edge via $f$, in particular adopting the cyclic ordering of the lines. The edges belonging to $V_i$ have positive circular order. The edges around a vertex belonging to $V_o$ have negative circular order. Notice that each line $l_k$ def\/ines a map $V_i \to V_o$ and thus the composition $l_k^{-1}\circ l_{k-1}$ def\/ines the monodromy representation $\sigma_k$ of a loop around $b_k$ on the set of internal vertices.
\begin{figure}[htpb]
	\centering\vspace*{-7mm}
	\begin{tikzpicture}
	\tikzstyle{circ} = [circle, minimum width=8pt, draw, inner sep=0pt];
	\tikzstyle{star} = [circle, minimum width=8pt, fill, inner sep=0pt];
	\tikzstyle{starsmall} = [circle, minimum width=4pt, fill, inner sep=0pt];
	
	\node[label=above right:{$P_i$},circ] (Pi) {};
	\draw (0: 3cm) node[label=above right:{$P_o$},star] (Po) {};
	
	\draw (300: 1.5cm) node[label=below right:$b_1$,starsmall] {};
	\draw (60: 1.5cm) node[label=above right:$b_2$,starsmall] {};
	\draw (180: 1.5cm) node[label=above left:$b_3$,starsmall] {};

	\draw[decoration={markings, mark=at position 0.25 with {\arrow{>}}},
	postaction={decorate}
	] (0,0) circle (1.5cm);
	
	\draw (90: 1.5cm) node[label=above:$\gamma$] {};

	\path (Pi) edge (Po);
	\draw (0: 1.5375cm) node[label=above right:$l_1$] {};
	
	\draw (0,0) .. controls (-2.25,3) and (0.75,3.75) .. node[midway,above]{$l_2$} (3,0);
	
	\draw (0,0) .. controls (-2.25,-3) and (0.75,-3.75) ..node[midway,below]{$l_3$} (3,0);
	
	\draw[fill=white] (0,0) circle (0.14cm);
	\end{tikzpicture}\vspace{-7mm}
	\caption{The case when the branching locus has three points. A choice of the triangle and of the three intersecting lines.
	The line complex is the inverse image of the above geometric conf\/iguration.} \label{fig:linecomplex}
\end{figure}
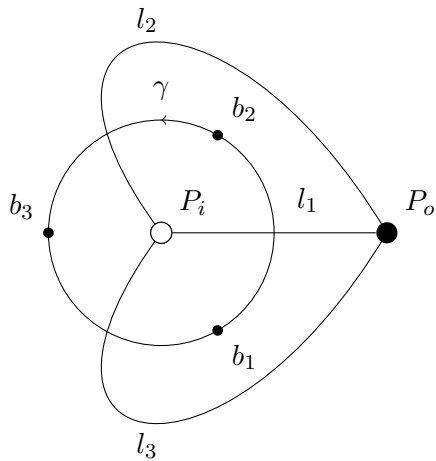

Critical points of $f$ are encircled by a closed loop of the graph with more than one internal (and external) vertex; the multiplicity of the critical point is the number of internal or external vertices in the loop minus one. If $b_k$ is the corresponding critical value, then the monodromy representation $\sigma_k$ acts on the internal vertices of the loop by a simple shift.

Stokes sectors are bounded by an inf\/inite sequence of internal and external vertices. If $b_k$ is the corresponding critical value, then the monodromy representation $\sigma_k$ acts transitively on the internal vertices of the sequences of internal vertices by a simple shift.

Note that the line complex of $f$, up to orientation-preserving homeomorphism of the domain, is far from unique, as it depends on the particular ordering of branch points and of the curve~$\gamma$. However, for f\/ixed choice of the ordering and curve, to any line complex corresponds a~unique equivalence class of meromorphic functions. Moreover, as to be expected, the braid group acts transitively on the possible line complexes of~$f$~\cite{elfving1934,lando2004}.

\subsection{Hermite oscillators and families of coverings}
In this subsection we construct two distinguished families of branched coverings of the sphere and show that they classify roots of generalised Hermite polynomials.

\begin{def*}
For $m,n\in\mathbb{N}^*$, we call the anharmonic oscillator \eqref{eq:anharmoniccent} (or its potential) with
$\theta_0=\tfrac{1}{2}n$ and $\theta_\infty=m+1+\tfrac{1}{2}n$, a level $(m,n)$ Hermite~I oscillator,
if it satisf\/ies all the properties in case \ref{H1} of Theorem~\ref{thm:hermiteanharmonic}.
\end{def*}

\begin{def*}
	For $m,n\in\mathbb{N}^*$, we call the anharmonic oscillator \eqref{eq:anharmoniccent} (or its potential) with
	$\theta_0=\tfrac{1}{2}m$ and $\theta_\infty=1-\tfrac{1}{2}m-n$,
	a level $(m,n)$ Hermite II oscillator, if it satisf\/ies all the properties in case \ref{H2} of Theorem \ref{thm:hermiteanharmonic}.
\end{def*}

\begin{def*}
	For $m,n\in\mathbb{N}^*$, we call the anharmonic oscillator \eqref{eq:anharmoniccent} (or its potential) with
	$\theta_0=\tfrac{1}{2}(m+n)$ and $\theta_\infty=\tfrac{1}{2}(n-m+2)$, a level $(m,n)$ Hermite III oscillator,
	if it satisf\/ies all the properties in case \ref{H3} of Theorem \ref{thm:hermiteanharmonic}.
\end{def*}

We can compute the branching locus of the Hermite I, II and III oscillators by Theorem~\ref{thm:hermiteanharmonic}. Let us consider the case I f\/irst. If $f=\frac{\psi_0}{\psi_1}$ then $w_0=w_2=f(0)=0$ while~$w_1$ and~$w_{-1}$ are nonzero and distinct. Similarly in case II, take $f=\frac{\psi_1}{\psi_0}$ then $w_1=w_3=f(0)=0$ while~$w_0$ and~$w_{2}$ are nonzero and distinct. Case~III is dif\/ferent. If $f=\frac{\psi_0}{\psi_1}$ then $w_0=w_2=0$, $w_1 =w_2=\infty$ and~$f(0)$ is dif\/ferent from zero and inf\/inity. We thus def\/ine two families of functions~$F_1$ and~$F_2$:
\begin{enumerate}\itemsep=0pt
 \item A function $f$ belongs to the family $F_1$, if it has a unique critical point and four direct singularities, with two asymptotic values coinciding with the critical value. Moreover $f$ is normalised such that
 \begin{itemize}\itemsep=0pt
 \item $\mathcal{S}(f)(\la)=-2\la^2+O(\la)$ as $\la \to \infty $;
 \item The critical point is $\la=0$;
 \item Let $w_k$ denote the asymptotic value of $f$ in the sector $\Omega_k^2$, $k=0,1,2,-1$. We have $w_0=w_2=f(0)=0$, $w_1=i$ and $w_{-1}=-i$.
 \end{itemize}

 \item A function $f$ belongs to the family $F_2$, if it has a unique critical point and four direct singularities, with the asymptotic values coinciding pairwise.
 Moreover $f$ is normalised such that
 \begin{itemize}\itemsep=0pt
 \item $\mathcal{S}(f)(\la)=-2\la^2+O(\la)$ as $\la \to \infty $;
 \item The critical point is $\la=0$;
 \item Let $w_k$ denote the asymptotic value of $f$ in the sector $\Omega_k^2$, $k=0,1,2,-1$. We have $w_0=w_2=0$, $w_1=w_{-1}=\infty$ and $f(0)=1$.
 \end{itemize}
\end{enumerate}

Notice that all functions in these families are Belyi functions \cite{pibelyi} because the f\/ive singularities lie over three distinct points.

In Fig.~\ref{fig:h12complex} (resp.~\ref{fig:h3complex}), where we use the notation depicted in Fig.~\ref{fig:defelementsh3complex}, we classify all line complexes describing functions in the families~$F_1$ (resp.~$F_2$), according to the combinatorics of coverings described above. Each line complex is completely determined by a set of arbitrary positive integers $n_1,\dots,n_4$. In Fig.~\ref{fig:h12complex}, the values of $b_1$, $b_2$, $b_3$ are
$b_1=-i$, $b_2=0$ and $b_3=i$ and the curve $\gamma$ is chosen to be the imaginary line. In Fig.~\ref{fig:h3complex}, the values of $b_1$, $b_2$, $b_3$ are
$b_1=1$, $b_2=0$ and $b_3=\infty$ and the curve $\gamma$ is chosen to be the real axis.

By the general theory, Fig.~\ref{fig:h12complex} (resp.~\ref{fig:h3complex}) classif\/ies all functions in $F_1$ (resp.~$F_2$). Indeed each of the line complexes described determines a unique function in $F_1$ (resp.~$F_2$), and vice versa for every $f \in F_1$ (resp.~$f \in F_2$) the line complex of $f$ is one of those depicted in Fig.~\ref{fig:h12complex} (resp.~\ref{fig:h3complex}).

One can read from the line complex in Fig.~\ref{fig:h12complex} that the critical point and two transcendental singularities lie over $b_2$, while the points $b_1$ and $b_3$ are simple asymptotic values. Moreover, the critical point has multiplicity $n_2+n_3-1$, and the equation $f(\la)=b_2$ has further $n_1+n_4$ simple solutions. For f\/ixed $n=n_2+n_3+1$ and $m=n_1+n_4+1$, there are $m \times n$ distinct line complexes and thus functions in $F_1$.

\begin{figure}[htpb]
	\centering
	\begin{tikzpicture}
	\tikzstyle{circ} = [circle, minimum width=8pt, draw, inner sep=0pt];
	\tikzstyle{star} = [circle, minimum width=8pt, fill, inner sep=0pt];
	
	\node[circ] (L1) {};
	\node[star] (L2) [left of=L1] {};
	\node[circ] (L3) [left of=L2] {};
	\node[star] (L4) [left of=L3] {};
	
	\node[circ] (UL1) [above left of=L4] {};
	\node[star] (UL2) [above left of=UL1] {};
	\node[circ] (DL1) [below left of=L4] {};
	\node[star] (DL2) [below left of=DL1] {};
	
	\node[star] (M1) [above right of=L1] {};
	\node[star] (M4) [below right of=L1] {};
	\node[circ] (M2) [right of=M1] {};
	\node[circ] (M3) [right of=M4] {};
	
	\node[star] (R1) [below right of=M2] {};
	\node[circ] (R2) [right of=R1] {};
	\node[star] (R3) [right of=R2] {};
	\node[circ] (R4) [right of=R3] {};
	
	\node[star] (UR1) [above right of=R4] {};
	\node[circ] (UR2) [above right of=UR1] {};
	\node[star] (DR1) [below right of=R4] {};
	\node[circ] (DR2) [below right of=DR1] {};
	
	\path (L2) edge node [above,inner sep=1pt]{$l_3$} (L1);
	\path (L4) edge node [above,inner sep=1pt]{$l_3$} (L3);
	
	\path (R1) edge node [above,inner sep=1pt]{$l_3$} (R2);
	\path (R3) edge node [above,inner sep=1pt]{$l_3$} (R4);
	
	\path (L4) edge node [above right,inner sep=1pt]{$l_2$} (UL1);
	\path (L4) edge node [below right,inner sep=1pt]{$l_1$} (DL1);
	
	\path (L1) edge node [above left,inner sep=1pt]{$l_2$} (M1);
	\path (L1) edge node [below left,inner sep=1pt]{$l_1$} (M4);
	
	\path (M2) edge node [above right,inner sep=1pt]{$l_2$} (R1);
	\path (M3) edge node [below right,inner sep=1pt]{$l_1$} (R1);
	
	\path (R4) edge node [above left,inner sep=1pt]{$l_2$} (UR1);
	\path (R4) edge node [below left,inner sep=1pt]{$l_1$} (DR1);
	
	\path (L3) edge[bend left=60] node [above,inner sep=1pt]{$l_2$} (L2);
	\path (L3) edge[bend right=60] node [below,inner sep=1pt]{$l_1$} (L2);
	
	\path (M1) edge[bend left=60] node [above,inner sep=1pt]{$l_3$} (M2);
	\path (M1) edge[bend right=60] node [below left,inner sep=1pt]{$l_1$} (M2);
	\path (M4) edge[bend left=60] node [above right,inner sep=1pt]{$l_2$} (M3);
	\path (M4) edge[bend right=60] node [below,inner sep=1pt]{$l_3$} (M3);
	
	\path (UL1) edge[bend left=60] node [below left,inner sep=1pt]{$l_1$} (UL2);
	\path (UL1) edge[bend right=60] node [above right,inner sep=1pt]{$l_3$} (UL2);
	
	\path (DL1) edge[bend left=60] node [below right,inner sep=1pt]{$l_3$} (DL2);
	\path (DL1) edge[bend right=60] node [above left,inner sep=1pt]{$l_2$} (DL2);
	
	\path (R2) edge[bend left=60] node [above,inner sep=1pt]{$l_2$} (R3);
	\path (R2) edge[bend right=60] node [below,inner sep=1pt]{$l_1$} (R3);
	
	\path (UR1) edge[bend left=60] node [above left,inner sep=1pt]{$l_3$} (UR2);
	\path (UR1) edge[bend right=60] node [below right,inner sep=1pt]{$l_1$} (UR2);
	
	\path (DR1) edge[bend left=60] node [above right,inner sep=1pt]{$l_2$} (DR2);
	\path (DR1) edge[bend right=60] node [below left,inner sep=1pt]{$l_3$} (DR2);

	\path (L3) edge[draw=none] node {$\boldsymbol{n_1}$} (L2);
	\path (M1) edge[draw=none] node {$\boldsymbol{n_2}$} (M2);
	\path (M4) edge[draw=none] node {$\boldsymbol{n_3}$} (M3);
	\path (R2) edge[draw=none] node {$\boldsymbol{n_4}$} (R3);
	
	\path (UL2) edge[draw=none] node[rotate=-45] {$\boldsymbol{\infty}$} (UL1);
	\path (DL2) edge[draw=none] node[rotate=45] {$\boldsymbol{\infty}$} (DL1);
	\path (UR2) edge[draw=none] node[rotate=45] {$\boldsymbol{\infty}$} (UR1);
	\path (DR2) edge[draw=none] node[rotate=-45] {$\boldsymbol{\infty}$} (DR1);
	\end{tikzpicture}
	\caption{Line complexes of Hermite I and II Nevanlinna functions.} \label{fig:h12complex}
\end{figure}
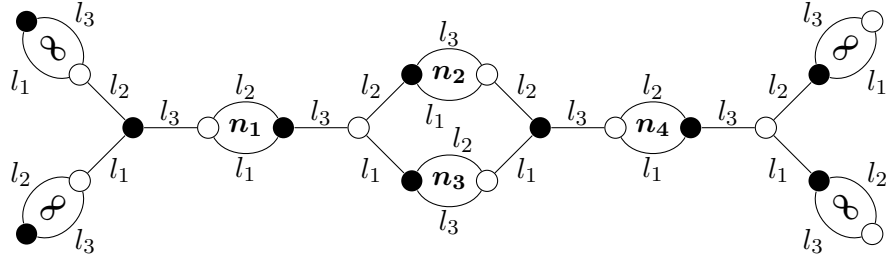

As for the line complex in Fig.~\ref{fig:h3complex} a similar analysis can be performed.
One can read that the critical point lies over $b_1$,
two transcendental singularities lie over $b_2$, and the other two over $b_3$. Moreover,
the critical point has multiplicity $n_1+n_2+n_3+n_4+1$, the equation $f(\la)=b_3$ has $n_1+n_3$ simple solutions,
and the equation $f(\la)=b_2$ has $n_2+n_4$ simple solutions.

\begin{figure}[htpb]
	\centering
	\begin{tikzpicture}
	\tikzstyle{square} = [circle, minimum width=3pt, draw, inner sep=3pt];
	\tikzstyle{squarefl} = [circle, minimum width=3pt, fill, inner sep=3pt];
	\node[square] (L1) {};
	\node[squarefl] (L2) [right of=L1] {};
	\node[square] (L3) [right of=L2] {};
	\node[squarefl] (L4) [right of=L3] {};
	
	\path (L1) edge node [above]{$l_k$} (L2);
	\path (L2) edge[bend left=60] node [above]{$l_i$} (L3);
	\path (L2) edge[bend right=60] node [below]{$l_j$} (L3);
	\path (L3) edge node [above]{$l_k$} (L4);

	\path (L2) edge[draw=none] node {$\mathbf{n}$} (L3);

	\node[square] (R1) [right of=L4] {};
	\node[squarefl] (R2) [right of=R1] {};
	\node[square] (R3) [right of=R2] {};
	\node[squarefl] (R4) [right of=R3] {};
	\path (R1) edge node [above]{$l_k$} (R2);
	\path (R2) edge[bend left=60] node [above]{$l_i$} (R3);
	\path (R2) edge[bend right=60] node [below]{$l_j$} (R3);
	\path (R3) edge node [above]{$l_k$} (R4);
	
	\path (L4) edge[draw=none] node[below=-5.2pt] {{\LARGE $=$}} (R1);
	
	\node[squarefl] (R5) [right of=R4] {};
	\node[square] (R6) [right of=R5] {};
	\node[squarefl] (R7) [right of=R6] {};

	\path (R5) edge[bend left=60] node [above]{$l_i$} (R6);
	\path (R5) edge[bend right=60] node [below]{$l_j$} (R6);
	\path (R6) edge node [above]{$l_k$} (R7);
	
	\path (R4) edge[draw=none] node {$\cdots$} (R5);
	\draw[decoration={brace,mirror,raise=26pt},decorate,thick]
	(R2) -- node[below=28pt] {$n$ times} (R7);
	\end{tikzpicture}
	\caption{Shorthand notation used in depicting line complexes where $n\in\mathbb{N}$ and $\{i,j,k\}=\{1,2,3\}$.}\label{fig:defelementsh3complex}
\end{figure}
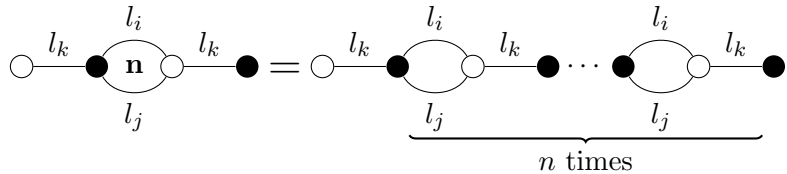

We now prove that a function belongs to the family $F_1$ if and only if $f$ is the ratio of two solutions of a~Hermite I or Hermite II harmonic oscillator.
\begin{thm}\label{thm:nevanlinnahermite12}
	The ratio of two particular solutions of a $(m,n)$ Hermite I oscillator or to a~$(m,n)$ Hermite~II oscillator
	belongs to the family $F_1$ with, in the case I, $m=n_1+n_4+1$ and $n=n_2+n_3+1$, and in the case~II, $m=n_2+n_3+1$ and $n=n_1+n_4+1$.

	Conversely, if $f$ belongs to the family $F_1$ then $f$ is the ratio of
	two solutions of a $(n_1+n_4+1,$ $n_2+n_3+1)$ Hermite I oscillator and to a $(n_2+n_3+1,n_1+n_4+1)$ Hermite II oscillator.
	In other words, $-\frac12 \mathcal{S}(f)(\la)$ is a level $(n_1+n_4+1,n_2+n_3+1)$ Hermite I harmonic oscillator and
	$\frac12 \mathcal{S}(f)(i \la)$ is a level $(n_2+n_3+1,n_1+n_4+1)$ Hermite II oscillators.
\end{thm}
\begin{proof}
The f\/irst part of the theorem follows by construction of the family $F_1$. We now prove the second part. Since $f$ belongs to $F_1$, it has a unique critical point of multiplicity $n_2+n_3$. Around the critical point $\la=0$, $-\frac12 \mathcal{S}(f)= \frac{n^2-1}{4\la^2}+\mathcal{O}(\frac{1}{\la})$ where $n=n_2+n_3$. Moreover, $f$ has four transcendental singularities and therefore $-\frac12 \mathcal{S}(f)(\la)$ is of the form $P_2(\la)+\frac{c_{-1}}{\la}+\frac{n^2-1}{4\la^2}$ for some $c_{-1} \in \bb{C}$ and some quadratic polynomial of the form $P_2(\la)=\la^2+c_1 \la+c_0$.

Consider the equation $\psi''(\la)=-\frac12 \mathcal{S}(f)(\la) \psi(\la)$. By the WKB asymptotic, see, e.g., Lem\-ma~\ref{lem:wkbturning}(iii) below, the logarithmic derivative of the subdominant solution $\psi_0$ (resp.\ $\psi_2$) is well approximated~-- for $|\la|\gg0$ in the union of the Stokes sectors $\Omega_0$, $\Omega_1$, $\Omega_{-1}$ (resp.\ $\Omega_2$, $\Omega_1$, $\Omega_{-1}$)~-- by $-V^{\frac12}(\la)$ up to an error $\mathcal{O}(\la^{-2})$, where the sign of the square root is chosen such that $V^{\frac12}=\la + \mathcal{O}(1)$. Since $w_0$ and $w_2$ coincide, $\psi_0$ and $\psi_2$ are linearly dependent. Therefore we can extend the WKB asymptotic to the whole complex plane
 \begin{gather}\label{eq:wkblog}
 \frac{\psi'_0(\la)}{\psi_0(\la)}=-V(\la)^{\frac12}+\mathcal{O}\big(\la^{-2}\big) \qquad \mbox{as} \quad |\la| \to \infty .
 \end{gather}
Consider the function $f=\frac{\psi_0}{\psi_1}$. The asymptotic values $w_0=w_2$ are clearly $0$ and so is, by the hypothesis, the critical value. Therefore, $\psi_0$ has a zero of order $\frac{n_2+n_3+1}2$ at $0$ ($\psi_0$ is two-valued if $n_2+n_3$ is even) and further $n_1+n_4$ simple zeros. We conclude that
\begin{gather*}
\lim_{R \to +\infty}\frac{1}{2 \pi i}\oint_{|\mu|=R}\frac{\psi'_0(\mu)}{\psi_0(\mu)}=\frac{n_2+n_3+1}2+n_1+n_4.
\end{gather*}
Because of the WKB estimates~(\ref{eq:wkblog}), the latter number coincides with the residue of~$\sqrt{V}$ at inf\/inity, which is equal to $\frac18 c_1^2-\frac12 c_0$. This means that $V(\la)$ is precisely of the form~(\ref{eq:anharmoniccent}) for a~level $(m,n)$ Hermite I potential with $m=n_1+n_4+1$, $n=n_2+n_3+1$ and $a=\frac{c_1}{2}$.
\end{proof}

\begin{cor}Fix $m,n\in\mathbb{N}^*$. Let $\widehat{F}_{m,n}=\lbrace f \in F_1 \, \text{s.t.}\, m=n_1+n_4+1 \,\text{and}\, n=n_2+n_3+1 \rbrace$.
Denote $\mathcal{S}(f)_1\colon f \to \bb{C}$ the coefficient of the linear term of the Schwarzian derivative of~$f$.
The mapping \begin{gather*}
 \Pi\colon \ \widehat{F}_{m,n} \to \bb{C},\qquad f \mapsto -\frac14\mathcal{S}(f)_1,
 \end{gather*}
is a bijection between the set $\widehat{F}_{m,n}$ and the set of the roots of generalised Hermite polyno\-mial~$H_{m,n}(z)$.
\end{cor}

\begin{cor}\label{cor:numberrealroots}
 For all $m,n\in\mathbb{N}$, $H_{m,n}(z)$ is a polynomial of order $m \times n$. Moreover, $H_{m,n}(z)$ has exactly $m$ real roots when $n$ is odd, and none when $n$ is even.
\end{cor}
 \begin{proof} About the order of $H_{m,n}(z)$. By the previous corollary, the number of roots of $H_{m,n}$ coincides with the number of distinct
 Stokes complexes such that $m=n_1+n_4+1$ and $n=n_2+n_3+1$; as it was already noted, this number equals $m \times n$.

 Concerning the number of real roots. By construction of the line complexes (Fig.~\ref{fig:h12complex}), if $f \in \widehat{F}_{m,n} $, has indices~$\{n_i\}$, then its conjugate, def\/ined as $\overline{f(\overline{\la})}$, belongs to $\widehat{F}_{m,n} $ with indi\-ces~$\{n'_i\}$, given by $n_1'=n_1$, $n_2'=n_3$, $n_3'=n_2$ and $n_4'=n_4$. In particular $f$ is self-conjugated, i.e., $f$ is a real analytic function, if and only if $n_2=n_3$.

 Clearly $a$ is a real root if\/f $(a,b)$ ($b$ is the coef\/f\/icient of the Laurent expansion at $z=a$ def\/ined in~\eqref{eq:laurentpole}) are reals
 if\/f the Schwarzian derivative of~$f$ is real if\/f~$f$ is real.

 Therefore the number of real roots is equal to the number of real functions in $\widehat{F}_{m,n}$. Since $n=n_2+n_3+1$, the constraint $n_2=n_3$ has one and only one solution if $n$ is odd, and no solutions if $n$ is even. Therefore, if $n$ is even, there are no normalised real functions $f$ and thus~$H_{m,n}(z)$ has no real roots.

 On the other hand, if $n$ is odd, the number of real functions equals the number of non-negative pairs of integers $n_1$, $n_4$ such $n_1+n_4=m-1$. This
 number is clearly~$m$.
 \end{proof}

We now prove the analogue of Theorem \ref{thm:nevanlinnahermite12} for Hermite~III oscillators.

\begin{thm}\label{thm:nevanlinnahermite3} The ratio of two particular solutions of a level $(m,n)$ Hermite~III oscillator belongs to the family~$F_2$ with $m=n_1+n_3+1$ and $n=n_2+n_4+1$.

Conversely, if $f$ belongs to the family $F_2$, then $f$ is the ratio of two solutions of a~level $(n_1+n_3+1,n_2+n_4+1)$ Hermite~III oscillator. In other words $-\frac12 \mathcal{S}(f)(\la)$ is a level $(n_1+n_3+1,$ $n_2+n_4+1)$ Hermite~III harmonic oscillator.
\end{thm}
\begin{proof}
The proof is almost identical to the proof of Theorem~\ref{thm:nevanlinnahermite12} and therefore omitted.
\end{proof}
	
	\begin{figure}
		\centering
		\begin{tikzpicture}
		\tikzstyle{circ} = [circle, minimum width=8pt, draw, inner sep=0pt];
		\tikzstyle{star} = [circle, minimum width=8pt, fill, inner sep=0pt];
		
		\node[circ] (A00) at (-1.53,-1.5) {};
		\node[star] (A01) at (-1.5,-0.5) {};
		\node[circ] (A02) at (-1.5,0.5) {};
		\node[star] (A03) at (-1.5,1.5) {};
		\node[star] (A10) at (-0.5,-1.5) {};
		\node[circ] (A20) at (0.5,-1.5) {};
		\node[star] (A30) at (1.5,-1.5) {};
		\node[circ] (A31) at (1.5,-0.5) {};
		\node[star] (A32) at (1.5,0.5) {};
		\node[circ] (A13) at (-0.5,1.5) {};
		\node[star] (A23) at (0.5,1.5) {};
		\node[circ] (A33) at (1.5,1.5) {};
		
		\node[star] (DL) at (-2.207,-2.207) {};
		\node[circ] (DL2) at (-2.914,-2.914) {};
		
		\node[circ] (UL) at (-2.207,2.207) {};
		\node[star] (UL2) at (-2.914,2.914) {};
		
		\node[star] (UR) at (2.207,2.207) {};
		\node[circ] (UR2) at (2.914,2.914) {};
		
		\node[circ] (DR) at (2.207,-2.207) {};
		\node[star] (DR2) at (2.914,-2.914) {};

		\path (A00) edge node [below,inner sep=1pt]{$l_1$} (A10);
		\path (A10) edge[bend left=60] node [above,inner sep=1pt]{$l_2$} (A20);
		\path (A10) edge[bend right=60] node [below,inner sep=1pt]{$l_3$} (A20);
		\path (A20) edge node [below,inner sep=1pt]{$l_1$} (A30);
		\node at (0,-1.5) {$\mathbf{n_4}$};
		
		\path (A00) edge node [left,inner sep=1pt]{$l_2$} (A01);
		\path (A01) edge[bend left=60] node [left,inner sep=1pt]{$l_3$} (A02);
		\path (A01) edge[bend right=60] node [right,inner sep=1pt]{$l_1$} (A02);
		\path (A02) edge node [left,inner sep=1pt]{$l_2$} (A03);
		\node at (-1.5,0) {$\mathbf{n_1}$};
		
		\path (A03) edge node [above,inner sep=1pt]{$l_1$} (A13);
		\path (A13) edge[bend left=60] node [above,inner sep=1pt]{$l_3$} (A23);
		\path (A13) edge[bend right=60] node [below,inner sep=1pt]{$l_2$} (A23);
		\path (A23) edge node [above,inner sep=1pt]{$l_1$} (A33);
		\node at (0,1.5) {$\mathbf{n_2}$};
		
		\path (A30) edge node [right,inner sep=1pt]{$l_2$} (A31);
		\path (A31) edge[bend left=60] node [left,inner sep=1pt]{$l_1$} (A32);
		\path (A31) edge[bend right=60] node [right,inner sep=1pt]{$l_3$} (A32);
		\path (A32) edge node [right,inner sep=1pt]{$l_2$} (A33);
		\node at (1.5,0) {$\mathbf{n_3}$};

		\path (A00) edge node [above left,inner sep=1pt]{$l_3$} (DL);
		\path (DL) edge[bend left=60] node [below right,inner sep=1pt]{$l_1$} (DL2);
		\path (DL) edge[bend right=60] node [above left,inner sep=1pt]{$l_2$} (DL2);
		\node[rotate=45] at (-2.560,-2.560) {$\boldsymbol{\infty}$};
		
		\path (A03) edge node [below left,inner sep=1pt]{$l_3$} (UL);
		\path (UL) edge[bend left=60] node [below left,inner sep=1pt]{$l_2$} (UL2);
		\path (UL) edge[bend right=60] node [above right,inner sep=1pt]{$l_1$} (UL2);
		\node[rotate=-45] at (-2.560,2.560) {$\boldsymbol{\infty}$};
		
		\path (A33) edge node [below right,inner sep=1pt]{$l_3$} (UR);
		\path (UR) edge[bend left=60] node [above left,inner sep=1pt]{$l_1$} (UR2);
		\path (UR) edge[bend right=60] node [below right,inner sep=1pt]{$l_2$} (UR2);
		\node[rotate=45] at (2.560,2.560) {$\boldsymbol{\infty}$};
		
		\path (A30) edge node [above right,inner sep=1pt]{$l_3$} (DR);
		\path (DR) edge[bend left=60] node [above right,inner sep=1pt]{$l_2$} (DR2);
		\path (DR) edge[bend right=60] node [below left,inner sep=1pt]{$l_1$} (DR2);
		\node[rotate=-45] at (2.560,-2.560) {$\boldsymbol{\infty}$};
		\end{tikzpicture}
		\caption{Line complexes of Hermite III Nevanlinna functions.} \label{fig:h3complex}
	\end{figure}
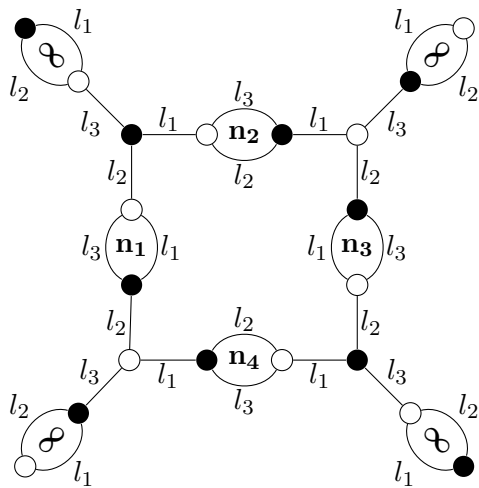
	
\subsection{Okamoto oscillators}
In this subsection we discuss the Okamoto case. We identify a family of meromorphic functions to which the oscillators related to roots of the generalised Okamoto polynomials belong. Firstly, it is helpful to apply the following change of variables,
\begin{gather}\label{eq:changevariables}
\widetilde{\psi}(x)=x^{-1}\psi\big(\tfrac{1}{3}\sqrt{3}x^3\big),
\end{gather}
which yields
\begin{gather}\label{eq:anharmonicokamoto}
\widetilde{\psi}''(x) = V_{10}(x;a,b,\theta)\widetilde{\psi}(x)\\ \nonumber
V_{10}(x) = 3x^4V\big(\tfrac{1}{3}\sqrt{3}x^3\big)+2x^{-2}= x^{10}+2\sqrt{3}ax^7+3\big(a^2+2(1-\theta_\infty)\big)x^4\\ \nonumber
\hphantom{V_{10}(x)=}{}
-3\sqrt{3}\big[b+\big(2\theta_\infty-\tfrac{1}{2}\big)a\big]x +\big(2+9\big(\theta_0^2-\tfrac{1}{4}\big)\big)x^{-2}.
\end{gather}
This equation has a regular singular point at $x=0$ with indices $\tfrac{1}{2}\pm 3\theta_0$, and an irregular singular point at $x=\infty$ of Poincar\'e rank $11$.

\begin{def*}For $m,n\in\mathbb{Z}$, we call the anharmonic oscillator \eqref{eq:anharmonicokamoto} (or its potential) with $\theta_0=\tfrac{1}{2}n-\tfrac{1}{6}$ and $\theta_\infty=\tfrac{1}{2}(2m+n+1)$, a level $(m,n)$ Okamoto oscillator, if it Stokes multipliers satisfy
	\begin{gather}\label{eq:okamotostokes}
	1+s_ks_{k+1}=0, \qquad k\in\mathbb{Z}_{12}.
	\end{gather}
\end{def*}

\begin{Def}\label{def:F3} We say that a meromorphic function $f$ belongs to the class $F_3^n$, $n \in \bb{Z}$ if it has~$13$ singularities: $12$ asymptotic values and $1$ critical point. Moreover
\begin{itemize}\itemsep=0pt
\item The Schwarzian derivative is normalised such that $-\frac12{\mathcal{S}}(f)(x)=x^{10}+O(x^9)$ as $x \to \infty$.
\item The critical point lies at $x=0$. It has multiplicity $|3n-1|-1$ for some $n \in \bb{Z}$ and $f(0)=0$ if $n\geq 1$ and $f(0)=\infty$ otherwise.
\item The asymptotic values $w_0,\dots,w_{11}$ in the Stokes sectors $\Omega_k^{10}$, $k =0,\dots, 11$ are
\begin{subequations}\label{eq:asymptoticokamoto}
	\begin{alignat}{5}
& w_0 =1,\qquad &&w_1 =\zeta^{-1},\qquad && w_2 =\zeta, \qquad && w_3=1&,\\
& w_4 =\zeta^{-1},\qquad && w_5 =\zeta,\qquad && w_6=1, \qquad && w_7 =\zeta^{-1}, & \\
& w_8 =\zeta,\qquad && w_9 =1,\qquad && w_{10} =\zeta^{-1}, \qquad && w_{11} =\zeta, &	
	\end{alignat}
\end{subequations}
	where $\zeta=e^{\frac{2\pi i}{3}}$.
	\item $f$ possesses the symmetry
	\begin{gather}\label{eq:symmetryokamoto}
	f(\zeta x)=\zeta^{-1}f(x).
	\end{gather}
\end{itemize}
\end{Def}

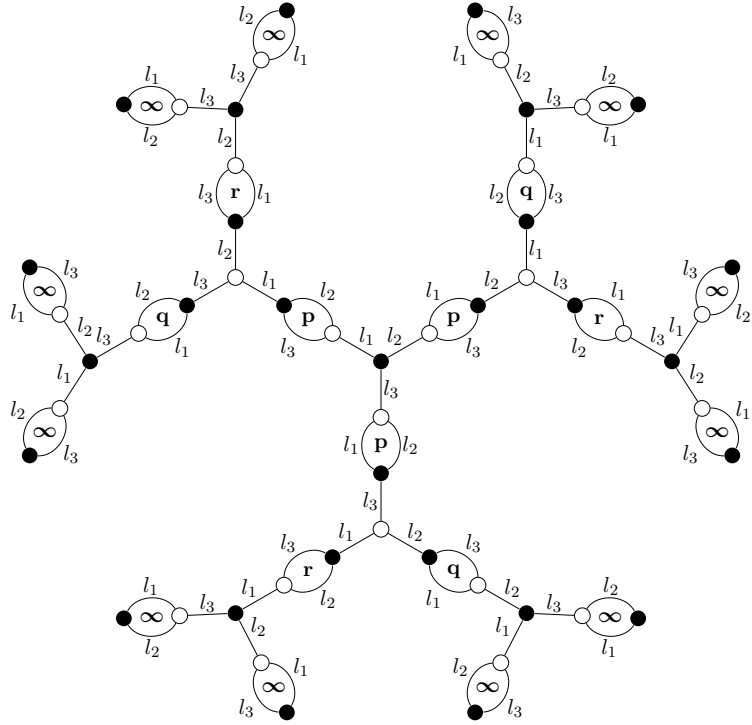
\begin{figure}[ht]
		\centering
		\resizebox{10cm}{!}{
		\begin{tikzpicture}
		\tikzstyle{circ} = [circle, minimum width=8pt, draw, inner sep=0pt];
		\tikzstyle{star} = [circle, minimum width=8pt, fill, inner sep=0pt];
		
		\node[star] (Or) at (0,0) {};
		
		\node[circ] (x1) at ({3*cos(pi/6 r)},{3*sin(pi/6 r)} ) {};
		\node[circ] (x2) at ({3*cos(5*pi/6 r)},{3*sin(5*pi/6 r)} ) {};
		\node[circ] (x3) at ({3*cos(-9*pi/6 r)},{3*sin(9*pi/6 r)} ) {};
		
		\node[star] (y1) at ({3*sqrt(3)*cos(0*2*pi/6 r)},{3*sqrt(3)*sin(0*2*pi/6 r)} ) {};
		\node[star] (y2) at ({3*sqrt(3)*cos(1*2*pi/6 r)},{3*sqrt(3)*sin(1*2*pi/6 r)} ) {};
		\node[star] (y3) at ({3*sqrt(3)*cos(2*2*pi/6 r)},{3*sqrt(3)*sin(2*2*pi/6 r)} ) {};
		\node[star] (y4) at ({3*sqrt(3)*cos(3*2*pi/6 r)},{3*sqrt(3)*sin(3*2*pi/6 r)} ) {};
		\node[star] (y5) at ({3*sqrt(3)*cos(4*2*pi/6 r)},{3*sqrt(3)*sin(4*2*pi/6 r)} ) {};
		\node[star] (y6) at ({3*sqrt(3)*cos(5*2*pi/6 r)},{3*sqrt(3)*sin(5*2*pi/6 r)} ) {};
		
		\node[star] (z1) at ({5/4*3*sqrt(3)*cos((0+1/2)*2*pi/12 r)},{5/4*3*sqrt(3)*sin((0+1/2)*2*pi/12 r)} ) {};
		\node[star] (z2) at ({5/4*3*sqrt(3)*cos((1+1/2)*2*pi/12 r)},{5/4*3*sqrt(3)*sin((1+1/2)*2*pi/12 r)} ) {};
		\node[star] (z3) at ({5/4*3*sqrt(3)*cos((2+1/2)*2*pi/12 r)},{5/4*3*sqrt(3)*sin((2+1/2)*2*pi/12 r)} ) {};
		\node[star] (z4) at ({5/4*3*sqrt(3)*cos((3+1/2)*2*pi/12 r)},{5/4*3*sqrt(3)*sin((3+1/2)*2*pi/12 r)} ) {};
		\node[star] (z5) at ({5/4*3*sqrt(3)*cos((4+1/2)*2*pi/12 r)},{5/4*3*sqrt(3)*sin((4+1/2)*2*pi/12 r)} ) {};
		\node[star] (z6) at ({5/4*3*sqrt(3)*cos((5+1/2)*2*pi/12 r)},{5/4*3*sqrt(3)*sin((5+1/2)*2*pi/12 r)} ) {};
		\node[star] (z7) at ({5/4*3*sqrt(3)*cos((6+1/2)*2*pi/12 r)},{5/4*3*sqrt(3)*sin((6+1/2)*2*pi/12 r)} ) {};
		\node[star] (z8) at ({5/4*3*sqrt(3)*cos((7+1/2)*2*pi/12 r)},{5/4*3*sqrt(3)*sin((7+1/2)*2*pi/12 r)} ) {};
		\node[star] (z9) at ({5/4*3*sqrt(3)*cos((8+1/2)*2*pi/12 r)},{5/4*3*sqrt(3)*sin((8+1/2)*2*pi/12 r)} ) {};
		\node[star] (z10) at ({5/4*3*sqrt(3)*cos((9+1/2)*2*pi/12 r)},{5/4*3*sqrt(3)*sin((9+1/2)*2*pi/12 r)} ) {};
		\node[star] (z11) at ({5/4*3*sqrt(3)*cos((10+1/2)*2*pi/12 r)},{5/4*3*sqrt(3)*sin((10+1/2)*2*pi/12 r)} ) {};
		\node[star] (z12) at ({5/4*3*sqrt(3)*cos((11+1/2)*2*pi/12 r)},{5/4*3*sqrt(3)*sin((11+1/2)*2*pi/12 r)} ) {};
		
		\node[circ] (a11) at ({1*cos(pi/6 r)},{1*sin(pi/6 r)} ) {};
		\node[circ] (a12) at ({1*cos(5*pi/6 r)},{1*sin(5*pi/6 r)} ) {};
		\node[circ] (a13) at ({1*cos(-9*pi/6 r)},{1*sin(9*pi/6 r)} ) {};
		
		\node[star] (b11) at ({2*cos(pi/6 r)},{2*sin(pi/6 r)} ) {};
		\node[star] (b12) at ({2*cos(5*pi/6 r)},{2*sin(5*pi/6 r)} ) {};
		\node[star] (b13) at ({2*cos(-9*pi/6 r)},{2*sin(9*pi/6 r)} ) {};

		\node[circ] (a21) at ($(x1)!0.6666!(y1)$) {};
		\node[circ] (a22) at ($(x1)!0.6666!(y2)$) {};
		\node[circ] (a23) at ($(x2)!0.6666!(y3)$) {};
		\node[circ] (a24) at ($(x2)!0.6666!(y4)$) {};
		\node[circ] (a25) at ($(x3)!0.6666!(y5)$) {};
		\node[circ] (a26) at ($(x3)!0.6666!(y6)$) {};
		
		\node[star] (b21) at ($(x1)!0.3333!(y1)$) {};
		\node[star] (b22) at ($(x1)!0.3333!(y2)$) {};
		\node[star] (b23) at ($(x2)!0.3333!(y3)$) {};
		\node[star] (b24) at ($(x2)!0.3333!(y4)$) {};
		\node[star] (b25) at ($(x3)!0.3333!(y5)$) {};
		\node[star] (b26) at ($(x3)!0.3333!(y6)$) {};
		
		\node[circ] (a31) at ($(y1)!0.5!(z1)$) {};
		\node[circ] (a32) at ($(y2)!0.5!(z2)$) {};
		\node[circ] (a33) at ($(y2)!0.5!(z3)$) {};
		\node[circ] (a34) at ($(y3)!0.5!(z4)$) {};
		\node[circ] (a35) at ($(y3)!0.5!(z5)$) {};
		\node[circ] (a36) at ($(y4)!0.5!(z6)$) {};
		\node[circ] (a37) at ($(y4)!0.5!(z7)$) {};
		\node[circ] (a38) at ($(y5)!0.5!(z8)$) {};
		\node[circ] (a39) at ($(y5)!0.5!(z9)$) {};
		\node[circ] (a310) at ($(y6)!0.5!(z10)$) {};
		\node[circ] (a311) at ($(y6)!0.5!(z11)$) {};
		\node[circ] (a312) at ($(y1)!0.5!(z12)$) {};

		\path (Or) edge node [above left,inner sep=1pt]{$l_2$} (a11);
		\path (a11) edge[bend left=60] node [above left,inner sep=1pt]{$l_1$} (b11);
		\path (a11) edge[bend right=60] node [below right,inner sep=1pt]{$l_3$} (b11);
		\path (b11) edge node [above left,inner sep=1pt]{$l_2$} (x1);
		\node at ($(a11)!0.5!(b11)$) {$\mathbf{p}$};
		
		\path (Or) edge node [above right,inner sep=1pt]{$l_1$} (a12);
		\path (a12) edge[bend left=60] node [below left,inner sep=1pt]{$l_3$} (b12);
		\path (a12) edge[bend right=60] node [above right,inner sep=1pt]{$l_2$} (b12);
		\path (b12) edge node [above right,inner sep=1pt]{$l_1$} (x2);
		\node at ($(a12)!0.5!(b12)$) {$\mathbf{p}$};
		
		\path (Or) edge node [right,inner sep=1pt]{$l_3$} (a13);
		\path (a13) edge[bend left=60] node [right,inner sep=1pt]{$l_2$} (b13);
		\path (a13) edge[bend right=60] node [left,inner sep=1pt]{$l_1$} (b13);
		\path (b13) edge node [left,inner sep=1pt]{$l_3$} (x3);
		\node at ($(a13)!0.5!(b13)$) {$\mathbf{p}$};

		\path (x1) edge node [above right,inner sep=1pt]{$l_3$} (b21);
		\path (b21) edge[bend left=60] node [above right,inner sep=1pt]{$l_1$} (a21);
		\path (b21) edge[bend right=60] node [below left,inner sep=1pt]{$l_2$} (a21);
		\path (a21) edge node [above right,inner sep=1pt]{$l_3$} (y1);
		\node at ($(a21)!0.5!(b21)$) {$\mathbf{r}$};
		
		\path (x1) edge node [right,inner sep=1pt]{$l_1$} (b22);
		\path (b22) edge[bend left=60] node [left,inner sep=1pt]{$l_2$} (a22);
		\path (b22) edge[bend right=60] node [right,inner sep=1pt]{$l_3$} (a22);
		\path (a22) edge node [right,inner sep=1pt]{$l_1$} (y2);
		\node at ($(a22)!0.5!(b22)$) {$\mathbf{q}$};
		
		\path (x2) edge node [left,inner sep=1pt]{$l_2$} (b23);
		\path (b23) edge[bend left=60] node [left,inner sep=1pt]{$l_3$} (a23);
		\path (b23) edge[bend right=60] node [right,inner sep=1pt]{$l_1$} (a23);
		\path (a23) edge node [left,inner sep=1pt]{$l_2$} (y3);
		\node at ($(a23)!0.5!(b23)$) {$\mathbf{r}$};
		
		\path (x2) edge node [above left,inner sep=1pt]{$l_3$} (b24);
		\path (b24) edge[bend left=60] node [below right,inner sep=1pt]{$l_1$} (a24);
		\path (b24) edge[bend right=60] node [above left,inner sep=1pt]{$l_2$} (a24);
		\path (a24) edge node [above left,inner sep=1pt]{$l_3$} (y4);
		\node at ($(a24)!0.5!(b24)$) {$\mathbf{q}$};
		
		\path (x3) edge node [above left,inner sep=1pt]{$l_1$} (b25);
		\path (b25) edge[bend left=60] node [below right,inner sep=1pt]{$l_2$} (a25);
		\path (b25) edge[bend right=60] node [above left,inner sep=1pt]{$l_3$} (a25);
		\path (a25) edge node [above left,inner sep=1pt]{$l_1$} (y5);
		\node at ($(a25)!0.5!(b25)$) {$\mathbf{r}$};
		
		\path (x3) edge node [above right,inner sep=1pt]{$l_2$} (b26);
		\path (b26) edge[bend left=60] node [above right,inner sep=1pt]{$l_3$} (a26);
		\path (b26) edge[bend right=60] node [below left,inner sep=1pt]{$l_1$} (a26);
		\path (a26) edge node [above right,inner sep=1pt]{$l_2$} (y6);
		\node at ($(a26)!0.5!(b26)$) {$\mathbf{q}$};

		\path (y1) edge node [above left,inner sep=1pt]{$l_1$} (a31);
		\path (a31) edge[bend left=60] node [above left,inner sep=1pt]{$l_3$} (z1);
		\path (a31) edge[bend right=60] node [below right,inner sep=1pt]{$l_2$} (z1);
		\node at ($(a31)!0.5!(z1)$) {$\boldsymbol{\infty}$};
		
		\path (y2) edge node [above,inner sep=1pt]{$l_3$} (a32);
		\path (a32) edge[bend left=60] node [above,inner sep=1pt]{$l_2$} (z2);
		\path (a32) edge[bend right=60] node [below,inner sep=1pt]{$l_1$} (z2);
		\node at ($(a32)!0.5!(z2)$) {$\boldsymbol{\infty}$};
		
		\path (y2) edge node [above right,inner sep=1pt]{$l_2$} (a33);
		\path (a33) edge[bend left=60] node [below left,inner sep=1pt]{$l_1$} (z3);
		\path (a33) edge[bend right=60] node [above right,inner sep=1pt]{$l_3$} (z3);
		\node at ($(a33)!0.5!(z3)$) {$\boldsymbol{\infty}$};
		
		\path (y3) edge node [above left,inner sep=1pt]{$l_3$} (a34);
		\path (a34) edge[bend left=60] node [above left,inner sep=1pt]{$l_2$} (z4);
		\path (a34) edge[bend right=60] node [below right,inner sep=1pt]{$l_1$} (z4);
		\node at ($(a34)!0.5!(z4)$) {$\boldsymbol{\infty}$};
		
		\path (y3) edge node [above,inner sep=1pt]{$l_3$} (a35);
		\path (a35) edge[bend left=60] node [below,inner sep=1pt]{$l_2$} (z5);
		\path (a35) edge[bend right=60] node [above,inner sep=1pt]{$l_1$} (z5);
		\node at ($(a35)!0.5!(z5)$) {$\boldsymbol{\infty}$};
		
		\path (y4) edge node [above right,inner sep=1pt]{$l_2$} (a36);
		\path (a36) edge[bend left=60] node [below left,inner sep=1pt]{$l_1$} (z6);
		\path (a36) edge[bend right=60] node [above right,inner sep=1pt]{$l_3$} (z6);
		\node at ($(a36)!0.5!(z6)$) {$\boldsymbol{\infty}$};
		
		\path (y4) edge node [above left,inner sep=1pt]{$l_1$} (a37);
		\path (a37) edge[bend left=60] node [below right,inner sep=1pt]{$l_3$} (z7);
		\path (a37) edge[bend right=60] node [above left,inner sep=1pt]{$l_2$} (z7);
		\node at ($(a37)!0.5!(z7)$) {$\boldsymbol{\infty}$};
		
		\path (y5) edge node [above,inner sep=1pt]{$l_3$} (a38);
		\path (a38) edge[bend left=60] node [below,inner sep=1pt]{$l_2$} (z8);
		\path (a38) edge[bend right=60] node [above,inner sep=1pt]{$l_1$} (z8);
		\node at ($(a38)!0.5!(z8)$) {$\boldsymbol{\infty}$};
		
		\path (y5) edge node [above right,inner sep=1pt]{$l_2$} (a39);
		\path (a39) edge[bend left=60] node [above right,inner sep=1pt]{$l_1$} (z9);
		\path (a39) edge[bend right=60] node [below left,inner sep=1pt]{$l_3$} (z9);
		\node at ($(a39)!0.5!(z9)$) {$\boldsymbol{\infty}$};
		
		\path (y6) edge node [above left,inner sep=1pt]{$l_1$} (a310);
		\path (a310) edge[bend left=60] node [below right,inner sep=1pt]{$l_3$} (z10);
		\path (a310) edge[bend right=60] node [above left,inner sep=1pt]{$l_2$} (z10);
		\node at ($(a310)!0.5!(z10)$) {$\boldsymbol{\infty}$};
		
		\path (y6) edge node [above,inner sep=1pt]{$l_3$} (a311);
		\path (a311) edge[bend left=60] node [above,inner sep=1pt]{$l_2$} (z11);
		\path (a311) edge[bend right=60] node [below,inner sep=1pt]{$l_1$} (z11);
		\node at ($(a311)!0.5!(z11)$) {$\boldsymbol{\infty}$};
		
		\path (y1) edge node [above right,inner sep=1pt]{$l_2$} (a312);
		\path (a312) edge[bend left=60] node [above right,inner sep=1pt]{$l_1$} (z12);
		\path (a312) edge[bend right=60] node [below left,inner sep=1pt]{$l_3$} (z12);
		\node at ($(a312)!0.5!(z12)$) {$\boldsymbol{\infty}$};
		
\end{tikzpicture}}
		\caption{Line complexes of functions in the class $F_3^n$ with $n=0$; $p$, $q$, $r$ are arbitrary natural numbers.}		\label{fig:okamotocomplexI}
	\end{figure}

\begin{thm}\label{thm:nevanlinnaokamoto}
 The ratio of two particular solutions of a level $(m,n)$ Okamoto oscillator belongs to the family $F_3^n$.

Conversely, if $f$ belongs to $F_3^n$, then $-\frac12\mathcal{S}(f)(x)$ is a level $(m,n)$ Okamoto oscillator, for some $m\in\mathbb{Z}$.
\end{thm}	
	
\begin{proof}
	It is immediately clear that $f$ has a critical point at $\lambda=0$, of multiplicity $|3n-1|-1$. Furthermore, by equation \eqref{eq:semi}, we have
	\begin{gather*}
	E=\begin{pmatrix}
	E_{12} & 0\\
	0 & E_{22}
	\end{pmatrix}\begin{pmatrix}
	s_0e^{-\frac{\pi i}{3}} & 1\\
	s_0e^{\frac{\pi i}{3}} & 1,
	\end{pmatrix}
	\end{gather*}
	and hence
	\begin{gather*}
	f=\frac{\widetilde{\psi}_+}{\widetilde{\phi}_-}=\frac{E_{22}}{E_{12}}\frac{\widetilde{\phi}_1-s_0\big(1+e^{\frac{\pi i}{3}}\big)\widetilde{\phi}_0}{\widetilde{\phi}_1+s_0\big(1+e^{\frac{-\pi i}{3}}\big)\widetilde{\phi}_0}.
	\end{gather*}
Therefore, upon rescaling $f\mapsto E_{12}/E_{22}f$, it is easy to see that the Stokes data \eqref{eq:okamotostokes} imply that~$f$ takes the asymptotic value $w_k$ def\/ined in \eqref{eq:asymptoticokamoto}, in the Stokes sector~$\Omega_k$, for $k\in\mathbb{Z}_{12}$. The symmetry~\eqref{eq:symmetryokamoto} is easily deduced from the fact that the potential satisf\/ies $V_{10}(\zeta x)=\zeta V_{10}(x)$.
	
Conversely, let $f$ belong to $F_3^n$ for some $n\in\mathbb{Z}$. Then, by Theorem \ref{thm:prenevanlinna}, it is clear that $V(x):=-\frac12 \mathcal{S}(f)(x)$ def\/ines a Laurent polynomial, of degree~$10$. Furthermore, by the symmetry~\eqref{eq:symmetryokamoto}, we have $V(\zeta x)=\zeta V(x)$, hence $V(x)$ is a Laurent polynomial with only $ 1 \pmod{3}$ degree terms. Namely $ V(x)=V_{10}(x;a,b,\theta)$ for some $a$, $b$, $\theta_0$ and $\theta_\infty$. By the characterisation near $x=0$ of~$f$, i.e., the f\/irst property in Def\/inition \ref{def:F3}, it is clear that $\theta_0=\tfrac{1}{2}n-\tfrac{1}{6}$.
 	
Now, equations \eqref{eq:asymptoticokamoto} imply that the Stokes multipliers of $V_{10}(x;a,b,\theta)$ satisfy~\eqref{eq:okamotostokes}. In particular, 'undoing' the change of variables~\eqref{eq:changevariables}, equation \eqref{eq:stcubic}, implies $-e^{-2\pi i \theta_\infty}=(-1)^n$. So there exists an $m\in\mathbb{Z}$ such that $\theta_\infty=\tfrac{1}{2}(2m+n+1)$ and indeed $V_{10}(x;a,b,\theta)$ is a $(m,n)$ Okamoto oscillator.
\end{proof}

The description of the correspondence in the above theorem is weaker than the ones in the Hermite cases, as the value of $m$ is not understood on the coverings side of the correspondence. To be more concrete, let us discuss the special case $n=0$, i.e., $f$ has no critical points, for which it is easy to work out all possible line complexes for the $F_3^0$ family (up to rotation by $\frac\pi2 i$). Indeed in such case the line complex must be given by Fig.~\ref{fig:okamotocomplexI}, for some $p,q,r\in\mathbb{N}$, where we used $b_1=1$, $b_2=\zeta$, $b_3=\zeta^{-1}$ and the Jordan curve $\gamma$ equal to the unit circle. The question that remains, what values of $p$, $q$, $r$ correspond to which value of $m$? We will not pursue this question further here.

\section{Asymptotic analysis of Hermite~I oscillators}\label{section:asymptotic}
This section of the article is dedicated to the asymptotic analysis of the zeros and poles of rational solutions of Hermite type, that is to say roots of the generalised Hermite polyno\-mials~$H_{m,n}(z)$, in \textit{the asymptotic regime $n$ bounded and $m \to \infty$}.

Our analysis is based on the solution of the inverse problem characterising zeros of Hermite~I solutions of ${\rm P}_{\rm IV}$, see Theorem~\ref{thm:hermiteanharmonic}.
For later convenience, we restate the inverse problem using the variables $\{\al=E^{-\frac12}a,E=2m+n\}$ introduced in equation~\eqref{eq:alE} above.
\begin{probinv*} Given $E,n\in\mathbb{N}^*$, determine the points $(\alpha,\beta)\in\mathbb{C}^2$ such that the anharmonic oscillator
	\begin{gather}\nonumber
	\psi''(\lambda)=V_{{\rm I}}(\lambda;\al,\beta,E,n)\psi(\lambda), \\ \label{eq:hIK}
	 V_{\rm I}(\la)= \lambda^2+2E^{\frac12}\al\lambda-E\big(1-\al^2\big)-\frac{E^{\frac12 }\beta}{\lambda}+\frac{n^2-1}{4 \la^2},
	\end{gather}
	satisfies the following two properties:
	\begin{enumerate}\itemsep=0pt
		\item[$1.$] \textbf{No-logarithm condition.} The resonant singularity at $\la=0$ is apparent.
		\item[$2.$] \textbf{Quantisation condition.} The solution subdominant at $\la=+\infty$, which we denote by $\psi_0$, is also subdominant at $\la=0$,
		namely $\psi_0 \equiv \psi_+ \sim \la^{\frac{n+1}{2}}$ as $\la \to 0$.
	\end{enumerate}
\end{probinv*}

We tackle f\/irst Theorem \ref{thm:asymptoticzeros}, whose proof is divided in four steps:
\begin{enumerate}\itemsep=0pt
 \item In Section \ref{sub:apparent} we analyse the no-logarithm condition and by doing so we express the unknown $\beta$ as an $n-$valued function of the unknown~$\al$.
\item Section \ref{sub:whittaker} is devoted to the large $E$ limit of the solution $\psi_+$ subdominant at $\la=0$.
\item In Section \ref{sub:wkb} we compute, by means of the WKB method, the large $E$ limit of the solu\-tion~$\psi_0$.
\item In Section \ref{sub:matching}, we match the asymptotic expansions of $\psi_+$ and $\psi_0$ to solve the inverse monodromy problem in the large $E$ limit and prove Theorem~\ref{thm:asymptoticzeros}.
\end{enumerate}
Afterwards we show, in Section \ref{sub:thm2}, all above steps can be appropriately modif\/ied in order to prove Theorem~\ref{thm:asymptoticzerosedge}. Finally in Section~\ref{sub:edge} we comment on the distribution of zeros in the edge.

\subsection{Asymptotic analysis of the no-logarithm condition}\label{sub:apparent}
Equation \eqref{eq:hIK} has a resonant singularity at $\la=0$ with exponents $\frac{1\pm n}{2}$. For generic values of the parameters $\{\al,\beta\}$, the Frobenius expansion of dominant solutions contain logarithmic terms. The singularity is apparent in those cases when these logarithmic terms are absent.

The no-logarithm condition imposes a polynomial constraint on the coef\/f\/icients of equation~\eqref{eq:hIK} which allow us to express $\beta$ as an $n$-valued function of the variable~$\al$. The $n$ branches of this functions, which we denote by $\beta_j(\al)$, $j \in J_n$, are asymptotic to $j i (1-\al^2)^{\frac12}$, in the large~$E$ limit. Indeed, we have the following proposition.
\begin{pro}\label{pro:apparent}
	Fix a simply connected compact domain $D$ in the $\al$-plane, not containing the points $\al=\pm1$. Then there exists an $E_0 >0$, and $n$ analytic functions $\beta_j(\al)=\beta_j(\al,E)$ on $D \times [E_0,\infty)$, $j\in J_n$, such that the following statements hold true:
	\begin{enumerate}\itemsep=0pt
		\item[$1.$] For every $\al\in D$ and $E\geq E_0$, there exist exactly $n$ distinct values $\beta$ such that the resonant singularity of equation~\eqref{eq:hIK} is apparent, given by $\beta_j(\al)$, $j \in J_n$.
		\item[$2.$] For $j \in J_n$, the branch $\beta_j(\al)$ has the following asymptotic expansion
		\begin{gather*}
		\beta_j(\al,E)=j i \big(1-\al^2\big)^{\frac12} \big(1+E^{-1}r_j(\al,E))\big) ,
		\end{gather*}
		where $r_j(\al,E)$ is a bounded function on $D \times [E_0,\infty)$.
	\end{enumerate}
\end{pro}

\begin{proof}
	As a f\/irst simplif\/ication we apply a change of variables
	$\nu=E^{\frac{1}{2}}(1-\al^2)^{\frac12}\la$ and $\tilde{\beta}=\beta(1-\al^2)^{-\frac12}$. The resulting equation reads
	\begin{gather}\label{eq:hIKres}
	\psi''= \left(E^{-2}(1-\alpha^2)^{-2}\nu^2 + 2 E^{-1}\big(1-\alpha^2\big)^{-3/2} \al \nu-1- \frac{\tilde{\beta}}{\nu} + \frac{n^2-1}{4 \nu^2}
	\right) \psi .
	\end{gather}
	Clearly the singularity of the latter equation is apparent if and only if the singularity of equation~\eqref{eq:hIK} is apparent. Let us now consider the existence of a dominant solution, without logarithmic terms, of~\eqref{eq:hIKres},
	\begin{gather*}
	\psi_-(\la)=\la^{\frac{-n+1}{2}}\sum_{k=0}^\infty \gamma_k \la^k,\qquad \gamma_0=1.
	\end{gather*}
	Following the Frobenius method, the coef\/f\/icients $\gamma_k$ satisfy the recursion
	\begin{gather}\label{eq:frobenius}
	k(n-k)\gamma_k= - \tilde{\beta} \gamma_{k-1}-\gamma_{k-2}+c_1 E^{-1} \gamma_{k-3}+ c_2 E^{-2} \gamma_{k-4} ,
	\end{gather}
	where $c_1=2 (1-\alpha^2)^{-3/2} \al$ and $c_2=(1-\alpha^2)^{-2}$.

	The recursion \eqref{eq:frobenius} is overdetermined at $k=n$. It has a~solution, and thus inf\/initely many, if and only if
	\begin{gather*}
	-\tilde{\beta} \gamma_{n-1}-\gamma_{n-2}+c_1 E^{-1} \gamma_{n-3}+ c_2 E^{-2} \gamma_{n-4}=0 .
	\end{gather*}
	Solving equation \eqref{eq:frobenius} recursively for all
	$\gamma_k$ with $k \leq n-1$, the no-logarithm condition becomes a single polynomial constraint of order $n$ in $\tilde{\beta}$, of the form
	\begin{gather*}
	Q_n(\tilde{\beta})+ \sum_{j=1}^{\lfloor \frac n2 \rfloor}E^{-j}Q_{n,j}(\tilde{\beta},\alpha)=0,
	\end{gather*}
	where $Q_n(\tilde{\beta})$ is a polynomial of order $n$, and the $Q_{n,j}$ are certain polynomial expressions in $\tilde{\beta}$ and
	$(1-\al^2)^{-\frac12}$. By def\/inition, the equation $Q_n(\tilde{\beta})=0$ is the no-logarithm condition for the
	dif\/ferential equation \eqref{eq:hIKres} with $E$ set equal to $\infty$, namely
	\begin{gather*}
	\psi''=\left(- 1 - \frac{\tilde{\beta} }{\nu}+\frac{n^2-1}{4 \nu^2}\right)\psi.
	\end{gather*}
	The latter equation is the Whittaker equation in a non-standard normal form, see \cite[Section~13.14]{nist}. 	In fact, the general solution of this ODE is a linear combination of the Whittaker functions $M_{\kappa,\mu}(z)$ and $W_{\kappa,\mu}(z)$ with $\kappa=-\tfrac{1}{2} \beta i$, $\mu=\tfrac{1}{2}n$ and $z= 2 i\nu$. From the formula \cite[equation~(13.14.18)]{nist}, it follows that the above ODE has an apparent singularity if and only if $\tilde{\beta}=j i$, for some $j \in J_n$. The thesis now easily follows.
\end{proof}

As a consequence of the latter proposition, we have solved half of the inverse monodromy problem. Indeed, we can now limit our study to the following inverse problem depending on just one unknown $\al$.
\begin{probinvred*}
	Upon fixing a suitable domain $D$ and $E_0>0$ as in Proposition~{\rm \ref{pro:apparent}}, given $j\in J_n$ and $E\geq E_0$, determine all $\alpha\in D$ such that the anharmonic oscillator
	\begin{gather}\label{eq:h1asym}
	\psi''(\lambda)=V_A\left(\lambda;\alpha,E,n\right)\psi(\lambda), \\
	V_A= \lambda^2+2\al E^{\frac12}\lambda- (1-\al^2)E-\frac{j i (1-\al^2)^{\frac12} E^{\frac12}\big(1+E^{-1} r_j(\al,E))\big)}{\lambda}
	+\frac{n^2-1}{4 \la^2} ,\nonumber
	\end{gather}
	satisfies the single constraint
	\begin{itemize}\itemsep=0pt
\item\textbf{Quantisation condition.} The solution subdominant at $\la=+\infty$, is also subdominant at $\la=0$, namely $\psi_0(\lambda) \equiv \psi_+(\lambda) \sim \la^{\frac{n+1}{2}}$ as $\la \to 0$.
	\end{itemize}
	Here the function $r_j(\al,E)$ is an asymptotically irrelevant contribution.
\end{probinvred*}

\subsection{Whittaker asymptotics near the origin}\label{sub:whittaker}
In this subsection we f\/ind an approximation for the solution $\psi_+(\la)$ of equation (\ref{eq:h1asym}) that is valid uniformly in $E$, in some $E$-dependent neighbourhood of the origin, slowly shrinking as $E \to +\infty$. According to our result, the solution $\psi_+$ takes the asymptotic form
\begin{subequations}\label{eq:trig}
	\begin{gather}
	 \Sigma(\la;j,\phi,\chi):=\sin\big({-} E^{\frac12}\big(1-\al^2\big)^{\frac12}\la + \varphi -\tfrac{1}{2}ji \log \la +i \chi \big) ,\\
	 \varphi= \frac{(n+3)\pi}{4} ,\qquad \chi=\tfrac12 \big[\log F_{n,j} -j \log\big(2 E^{\frac12}\big(1-\al^2\big)^\frac12\big)\big].
	\end{gather}
\end{subequations}
Here $F_{n,j}$ is def\/ined in \eqref{eq:Fnj}. This approximation is valid for~$\la$ belonging to a domain which scales as $E^{\gamma-\frac12}$, for some $\gamma<\frac12$, on which the modulus of the function $\Sigma$ is uniformly bounded. We therefore consider subsets of the $\la$-plane, which take the form
\begin{gather}\label{eq:Lambdag}
\Lambda_{+}=\big\lbrace \Re \la>0 , \, |\la|= E^{\gamma-\frac12} s \,\mbox{and}\, \big|\Im\big[E^{\frac12}\big(1-\al^2\big)^{\frac12} \la +
\tfrac{1}{2}ji \log{\big( E^{\frac12}\big(1-\al^2\big)^{\frac12}\la\big)} \big]\big|\leq c \big\rbrace .\!\!\!
\end{gather}
for some f\/ixed positive real numbers $s$, $c$, $\gamma$ with $\gamma<\frac12$. On such domains we have the following estimate for the solution subdominant at zero of equation \eqref{eq:h1asym}.
\begin{pro}\label{pro:estimateat0}
	Upon fixing a suitable domain $D$ and $E_0>0$ as in Proposition~{\rm \ref{pro:apparent}} and $j\in J_n$, positive real numbers $s$, $c$, $\gamma$ with $\gamma<\frac12$ and defining $\Lambda_+$ as in \eqref{eq:Lambdag}, let $\psi_+(\la)$ be the subdominant solution at zero of \eqref{eq:h1asym} normalised as follows
	\begin{gather*}
	\psi_+(\la)=\big(E^{\frac12}\big(1-\al^2\big)^{\frac12} \la\big)^{\frac{n+1}{2}} \big( 1+ \mathcal{O}(\la) \big).
	\end{gather*}
	Then there exists an $E_1\geq E_0$ such that for all $E\geq E_1$, the solution $\psi_+(\la)$ has the following properties:
	\begin{enumerate}[label={\rm (\roman*)}] \itemsep=0pt
		\item $\psi_+(\la)$ is a holomorphic function of $\al$ in $D$;
		\item There exists a constant $C$ -- independent of $\al$ and $E$ -- such that for all $\la \in \Lambda_{+}$,
		\begin{gather}\nonumber
		 |\psi_+(\la)-\Sigma(\la;j,\varphi,\chi \big)| \leq C E^{-\gamma'} , \\ \label{eq:estimate0}
		\big| E^{-\frac12} \psi_+^{\prime}(\la)+ \Sigma\big(\la;j,\varphi+\tfrac{1}{2}\pi,\chi\big)\big| \leq C E^{-\gamma'} ,
		\end{gather}
		where $\Sigma$, $\varphi$ and $\chi$ are as defined in equations \eqref{eq:trig} with $\gamma'=\min\lbrace1-2\gamma,\gamma \rbrace$.
	\end{enumerate}
\end{pro}
\begin{proof}	Part (i) of the proposition is a standard application of the regular perturbation theory for linear ODEs.	We divide the proof of part (ii) into several steps.	Firstly, we apply the change of variables $u_E(\nu)=\psi_+\big(E^{-\frac12}(1-\al^2)^{-\frac12} \nu\big)$ to equation~\eqref{eq:h1asym} to obtain
	\begin{gather*}
	u''(\nu) =W_E(\nu) u(\nu) ,\\
	W_E(\nu) =E^{-2}\big(1-\al^2\big)^{-2}\nu^2+ 2 \al E^{-1}\big(1-\al^2\big)^{-\frac32}\nu- 1-\frac{j i \big(1+E^{-1} r_j(\al,E)\big)}{\nu}+\frac{n^2-1}{4 \nu^2}.\nonumber
	\end{gather*} Next we compare the solution $u_E(\nu)=\nu^{\frac{n+1}{2}}( 1+ \mathcal{O}(\nu) )$ with $u_{\infty}(\nu)$, the subdominant solution of the
	same equation after $E$ has been set equal to $+\infty$,
	\begin{gather}\label{eq:h1at0W}
	u''(\nu)= \left(- 1-\frac{j i }{\nu}+\frac{n^2-1}{4 \nu^2}\right) u(\nu) .
	\end{gather}
As was already mentioned in the proof of Proposition (\ref{pro:apparent}), the latter equation is the Whittaker equation in disguise; the general solution of this ODE is a linear combination of the Whittaker functions $M_{\kappa,\mu}(z),W_{\kappa,\mu}(z)$	with $\kappa=-\tfrac{1}{2} j$, $\mu=\tfrac{1}{2}n$ and $z= 2 i\nu$. In particular
	\begin{gather}\label{eq:uWhitaker}
	u_{\infty}(\nu)=(2i)^{-\frac{n+1}{2}} M_{\kappa,\mu}(2 i \nu) , \qquad \kappa=-\tfrac{1}{2}ji , \qquad \mu=\tfrac{1}{2}n .
	\end{gather}
	We now compare the solutions $u_E(\nu)$ and $u_{\infty}(\nu)$.
	\begin{claim}Let $V=\lbrace \Re \nu >0 , \, |\nu|\geq 1 \, \mbox{and} \, \Im\big[ \nu-\tfrac{1}{2}ji \log \nu \big]\leq c \rbrace$, then there exist constants $C_1$, $C_2$, $C_3$~--- independent of~$\al$ and~$E$~-- such that for all $\nu \in V$, the following estimates hold
		\begin{gather*}
		| u_{E}(\nu)-u_{\infty}(\nu) | \leq C_1 E^{-1} |\nu|^2 + C_2 E^{-2} |\nu|^3 + C_3 E^{-1} \log{|\nu|}, \\
		| u_{E}^{\prime}(\nu)-u_{\infty}^{\prime}(\nu)| \leq C_1 E^{-1} |\nu|^2 + C_2 E^{-2} |\nu|^3 + C_3 E^{-1} \log{|\nu|}.
		\end{gather*}
	\end{claim}
	
	To prove the claim, let us f\/irst def\/ine
	\begin{gather*}
	R(\nu)=E^{-2}\big(1-\al^2\big)^{-2}\nu^2 + 2 \al E^{-1}\big(1-\al^2\big)^{-\frac32} \nu+E^{-1}\frac{r_j(\al,E)}{\nu}.
	\end{gather*}
	It is a standard result (see, e.g., \cite[Section~4]{erdelyi10}) that $u_E(\nu)$ satisf\/ies the following integral equation
	\begin{gather*}
	u_E(\nu)=u_{\infty}(\nu)-\int_0^{\nu}K(\nu,t)R(t)u_E(t) {\rm d}t ,
	\end{gather*}
provided the integral equation has a solution. Here $K(\nu,t)=u_{\infty}(\nu)\tilde{u}(t)-u_{\infty}(t)\tilde{u}(\nu)$, where~$\tilde{u}(\nu)$ is any solution of Whittaker equation \eqref{eq:h1at0W} such that the Wronksian $W[u_{\infty},\tilde{u}]=1$; for example $\tilde{u}(\nu)=\frac{1}{n}M_{\kappa,-\mu}(2 i \nu)$, where $\kappa$, $\mu$ are def\/ined in equation~\eqref{eq:uWhitaker} above.
	
In order to prove the claim we need to study the integral operator
	\begin{gather*}
	\mathcal{K}[f](\nu)=-\int_{0}^\nu K(\nu,t) R(t) f(t) {\rm d}t.
	\end{gather*}
To this end, note that
	\begin{gather}\label{eq:firstKestimate}
	K(\nu,t)=\frac{1}{n}(\nu t)^{-\frac{n-1}2}\big(\nu^n-t^n\big)\big(1+ \mathcal{O}(\nu)+\mathcal{O}(t) \big)\nu, \qquad \mbox{as} \quad t \to 0 ,
	\end{gather}
	since $\tilde{u}(\nu)=\frac{\nu^{\frac{1-n}{2}}}{n} ( 1+ \mathcal{O}(\nu) )$.
	
	Let us def\/ine $B=\{|\nu|\leq 1\colon |\arg{\nu}|\leq \frac{\pi}{4}\}$ and $H_n$ the Banach space of continuous functions in~$B$, analytic in its interior and with f\/inite norm $||f||_n=\sup\limits_{\nu \in B} |\nu^{-\frac{n+1}{2}}f(z)|$. It is easy to see that~$\mathcal{K} $ is a bounded operator on~$H_n$ with operator norm
	\begin{gather*}
	||\mathcal{K}||_n \leq C_0 \int_0^1 |t R(t)| {\rm d}t,
	\end{gather*}
	where $C_0=2 \sup\limits_{|t|\leq |\nu|}\big|K(\nu,t)t^{\frac{n-1}{2}}\nu^{-\frac{n+1}{2}}\big| <\infty$, because of~\eqref{eq:firstKestimate}.

Since $|\int_0^1 |t R(t)|{\rm d}t$ is of order $E^{-1}$, we have, for $E$ big enough, that $||\mathcal{K}||_n$ is smaller than one, and hence the resolvent series converges. Therefore, for $E$ big enough, $ u_E(1)=u_{\infty}(1)+c_1(\al,E) E^{-1} $ and $ u_E'(1)=u_{\infty}'(1)+c_2(\al,E) E^{-1}$, for some bounded functions $c_1$, $c_2$.
	
Let us call $\widehat{u}(\nu)$ the solution of \eqref{eq:h1at0W} which solves the Cauchy problem $u(1)=u_E(1)$, $u'(1)=u_E'(1)$. Because of our discussion so far,
	\begin{gather}\label{eq:widehatu}
	\widehat{u}(\nu)=\big(1+E^{-1}\tilde{c_1}(\al,E) \big)u_{\infty}(\nu)+E^{-1}\tilde{c_2}(\al,E) \tilde{u}(\nu),
	\end{gather}
	for some bounded functions $\tilde{c_1}$, $\tilde{c_2}$.
	
	Note that $u_E$ also solves the integral equation
	\begin{gather*}
	u_E(\nu)=\widehat{u}(\nu)-\int_1^{\nu}\widehat{K}(\nu,t)R(t)u_E(t) {\rm d}t ,
	\end{gather*}
where $\widehat{K}(\nu,t)=\widehat{u}(\nu) u(t)-\widehat{u}(t)u(\nu)$, with $u(\nu)$ is any solution of Whittaker equation~\eqref{eq:h1at0W} such that the Wronksian $W[\widehat{u},u]=1$.
	
	To f\/inish the proof of the claim we study the new integral operator
	\begin{gather*}
	\widehat{{\mathcal{K}}}[f](\nu)=-\int_{0}^\nu \widehat{ K}(\nu,t) R(t) f(t){\rm d}t.
	\end{gather*}
on the Banach space $H_{\rho}$ of functions continuous on $V_\rho=V \cap \lbrace |\nu| \leq \rho \rbrace$, analytic in its interior, equipped with the standard supremum norm, which we denote by $||\cdot||_{\infty,\rho}$.
	
From a standard WKB estimate, it follows that the general solution of the Whittaker equation (and its f\/irst derivatives) is, for large~$\nu$, asymptotic to a~linear combination of the functions $\exp(\pm i \nu \mp \frac{j}{2} \log \nu)$. Therefore $\widehat{u}(\nu)$, $ u(t)$, $\widehat{K}(\nu,t)$ and its f\/irst derivatives are all uniformly bounded on $V$. The following estimate on the operator norm then easily follows,
	\begin{gather*}
	\big\|\widehat{\mathcal{K}}\big\|_{\infty,\rho}\leq\big(d_1 E^{-1} \rho^2 + d_2 E^{-2} \rho^3 + d_3 E^{-1} \log{\rho}\big),
	\end{gather*}
	for some $ d_1>0$, $d_2>0$, $d_3>0$ independent of~$\al$. The latter estimate together with the standard Volterra property
	\begin{gather*}
	\big\|\widehat{\mathcal{K}}^N\big\|_{\infty,\rho}\leq \frac{||\widehat{\mathcal{K}}||^N_{\infty,\rho}}{N!} , \qquad N \in \bb{N}
	\end{gather*}
	leads to the following asymptotic estimate, for $E$ big enough
	\begin{gather*}
	||u_E-\widehat{u}||_{\infty,\rho}\leq 2 \big(d_1 E^{-1} \rho^2 + d_2 E^{-2} \rho^3 + d_3 E^{-1} \log{\rho}\big).
	\end{gather*}
	The same kind of estimate also holds for the dif\/ference of the derivatives because the kernel $\widehat{K}$ has bounded derivatives. The claim now follows from these estimates and equation~\eqref{eq:widehatu}.
	
We have come to the third and last step of the proof, which is a straightforward application of the known asymptotic results of the Whittaker functions. From \cite[equation~(13.14.21)]{nist}, we know there exists a~$C_4$ such that in the domain~$V$ of the claim, the following estimates hold
	\begin{gather*}
	 \left|u_{\infty}(\nu)-\sin\left(-\nu+\frac{3+n}{4}\pi-\tfrac{1}{2}j i\log \nu + \tfrac12 i\log F_{n,j} \right)\right| \leq C_4 \nu^{-1}, \\
	 \left|u_{\infty}^{\prime}(\nu)+\cos\left(-\nu+\frac{3+n}{4}\pi-\tfrac{1}{2}j i\log \nu + \tfrac12 i\log F_{n,j} \right)\right| \leq C_4 \nu^{-1} .
	\end{gather*}
	At the same time, if we let $|\nu|=E^{\gamma}s$, then, from the claim, we obtain that there exists a~cons\-tant~$C_5$ depending only on~$s$ such that the following estimates hold
	\begin{gather*}
	|u_{E}(\nu)-u_{\infty}(\nu)| \leq C_5 E^{-1+2 \gamma}, \\
	|u_{E}^{\prime}(\nu)-u_{\infty}^{\prime}(\nu)| \leq C_5 E^{-1+2\gamma} .
	\end{gather*}
	Combining the two sets of estimates, noticing that $V \cap \lbrace |\nu|=E^{\gamma}s \rbrace$ coincides with the set $\Lambda_{+}$,
	after the change of variable $\la=E^{-\frac12}(1-\al^2)^{-\frac12} \nu$, we obtain the thesis.
\end{proof}

Observe that the best approximation is achieved with the choice $\gamma=\frac13$, in which case the error scales with exponent~$-\frac13$.

\subsection{Asymptotic in the transition region}\label{sub:wkb}
In this subsection we study the asymptotic behaviour, as $E \to \infty$, of the solution $\psi_0$ of equation~(\ref{eq:h1asym}) subdominant in the Stokes sector~$S_0$, in the right half-plane $\Re \la >0$. To this aim we use the well-known WKB method. Since the literature is abundant, we refer to it
(in particular~\cite{wasow12} and~\cite{piwkb}) for all basic results.

In order to compare the solution $\psi_0$ with the formula we derived for the solution $\psi_+$, we need to compute the asymptotic expansion of $\psi_0$ for points scaling as $E^{\gamma-\frac12}$ with $\gamma<\frac12$ (even though the WKB analysis actually yields uniform estimates in the whole complex plane). Therefore we f\/ix a compact subset $K$ of the strip $ \lbrace {0<\Re \la < 1-\Re \al }\rbrace$ and we def\/ine the set
\begin{gather*}
\Lambda_0= \big\lbrace \la \in \bb{C} \mbox{ s.t. } E^{\frac12-\gamma}\la\in K \big\rbrace.
\end{gather*}
We have the following result.

\begin{pro}\label{pro:wkbasy}
	Fix a compact simply-connected domain $D$ of the $\al$-half-plane $\Re \al < 1$ and a positive number $\gamma<\frac12$.
	Then there exists an $E_0>0$ such that for $E\geq E_0$, upon appropriate normalisation of the solution $\psi_0(\la;\al)$ of
	equation \eqref{eq:h1asym}, subdominant in the Stokes sector $S_0$, the solution $\psi_0(\la;\alpha)$ depends holomorphically on $\al$ in $D$ and there exists a constant $C$, such that for all $\la \in \Lambda_0$ and $\al\in D$,
		\begin{gather}\nonumber
		\big|\psi_0(\la;\al)-\Sigma(\la;j,\varphi',\chi') \big|\leq C E^{-\tilde{\gamma}} \big(|\Sigma(\la;j,\varphi',\chi')|+
		\big|\Sigma\big(\la;j,\varphi'+\tfrac{1}{2}\pi,\chi'\big)\big|\big),\\ \nonumber
		\big|E^{-\frac12} \frac{\partial \psi_0(\la; \al)}{\partial \la}+\Sigma\big(\la;j,\varphi'+\tfrac{1}{2}\pi,\chi'\big)\big|\leq
		C E^{-\tilde{\gamma}} \big(|\Sigma(\la;j,\varphi',\chi')|+\big|\Sigma\big(\la;j,\varphi'+\tfrac{1}{2}\pi,\chi'\big)\big|\big),\\ \label{eq:asymptransition}
		\varphi'=\tfrac{1}{2}E\big[{-}\al\big(1-\al^2\big)^{\frac12}+\arccos(\al)\big]+\tfrac{1}{4}\pi, \qquad \chi'=\tfrac{1}{2}j\log\big(2\big(1-\al^2\big)E^{-\frac12}\big),
		\end{gather}
where $\tilde{\gamma}=\min\lbrace 2 \gamma, 1-2\gamma \rbrace$ and the function $\Sigma(\la;j,\varphi,\chi)$ is defined in equation \eqref{eq:trig}.
\end{pro}

Notice that no generality is lost in restricting to $\Re \al <1$, as the roots of the generalised Hermite polynomials are symmetric under ref\/lection in the imaginary axis.

In the proof of the above proposition, we assume that the reader has some familiarity with the WKB method. We begin the proof by analysing the turning points, i.e., zeros of the potential, of equation~\eqref{eq:h1asym}. A detailed dominant balance analysis, in the large~$E$ limit of the potential~$V_A$ in equation \eqref{eq:h1asym}, shows that asymptotically there are two turning points in the right half-plane. One of order $E^{\frac12}$ and the other of order $E^{-\frac12}$, which we denote respectively by $\la_+$ and $\la_-$. In fact we have
\begin{gather}\label{eq:la+}
 \la_+=E^{\frac12}(1-\al)+ E^{-\frac{1}{2}}\frac{ji}{2(1-\al)} + \mathcal{O}\big(E^{-\frac{3}{2}}\big),\\ \nonumber
 \la_-=E^{-\frac12}\frac{\sqrt{n^2-j^2-1} -ji}{(1-\al^2)^{\frac12}}+\mathcal{O}\big(E^{-\frac32}\big) .
\end{gather}
In our analysis only $\la_+$ plays a role because the phase-shift originating from the turning point~$\la_-$ is already taken into account in the Whittaker-like asymptotic, see Proposition \ref{pro:estimateat0}.

The WKB theory is based around the analysis of the phase function
\begin{gather*}
S_A(\la)=-i\int_{\la}^{\la_+}\sqrt{-V_A(\mu)}{\rm d}\mu ,
\end{gather*}
where $V_A$ is the potential in equation~\eqref{eq:h1asym}. We compute the asymptotic behaviour of the phase function in the following lemma.
\begin{lem}\label{lem:phasefunction}
	Fix a positive $t$ and denote ${\mathcal{H}}=\lbrace t \leq {\Re \la} \leq E^{\frac12} (1-\Re\al)\rbrace $. Then the phase function~$S_A$ admits the following asymptotic expansion in~$E$, on~${\mathcal{H}}$,
	\begin{gather}\label{eq:phaseA}
	S_A(\la)=\tfrac{1}{4}E\left(-2 \big(\al+E^{-\frac12}\la\big) \sqrt{1-\big(\al+E^{-\frac12}\la\big)^2}
	-2\sin^{-1}\big(\al+E^{-\frac12}\la\big)+\pi \right)\\ \nonumber
	{}+\tfrac{1}{2}ji\log\left(
	\Big[\sqrt{(1-\al^2) \big(1-\big(\al+E^{-\frac12}\la\big)^2\big)}-\al \big(\al+E^{-\frac12}\la\big)+1 \Big]/\big(E^{-\frac12}\la\big) \right)+\mathcal{O}\big(E^{-\frac{1}{2}+\e}\big),
	\end{gather}
	for any $\e>0$.

	Moreover, f\/ix positive constants $s$ and $\gamma<\frac12$ and let $\la=E^{-\tfrac{1}{2}+\gamma} s$, then $S_A(\la)$ admits the expansion
	\begin{gather*}
	S_A(\la)=\tfrac{1}{2}E\big[-\al\big(1-\al^2\big)^{\frac12}+\cos^{-1}{\al}\big]-E^{\frac12}\big(1-\al^2\big)^{\frac12}\la +\tfrac{1}{2}ji \log{\frac{2(1-\al^2)}{E^{\frac12}\la}} +\mathcal{O}\big(E^{-\tilde{\gamma}}\big),
\end{gather*}
	where $\gamma'=\min\lbrace \gamma, 1-2 \gamma \rbrace$.

The above errors $\mathcal{O}(E^{-\frac{1}{2}+\e})$ and $\mathcal{O}(E^{-\tilde{\gamma}})$ are uniform w.r.t.\ $\al \in D$, where~$D$ is the domain defined in Proposition~{\rm \ref{pro:wkbasy}}.
\end{lem}
\begin{proof}
	The proof is in Appendix \ref{appendix:wkb}.
\end{proof}

In order to understand the WKB computation it is useful to draw the Stokes graph of the potential,
that is the union of the Stokes lines $\Re\int_{\la^*}^\la \sqrt{V_A} {\rm d}\mu=0$ where $\la^*$ is any
turning point of $V_A(\la)$; we refer to \cite{piwkb} for more details.

We are interested only in the Stokes lines emanating from $\la_+$, which can be read of\/f for\-mu\-la~\eqref{eq:phaseA}, see Fig.~\ref{fig:stokescomplex}.
\begin{figure}[htpb]		\centering
	{\setlength{\fboxsep}{0pt}%
		\setlength{\fboxrule}{0.4pt}%
		\fbox{\includegraphics[width=8cm]{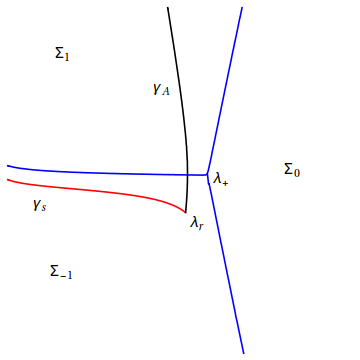}}}
	\caption{The right half-plane, with the turning point $\la_+$ and the Stokes lines emanating from it, the Stokes sectors~$\Sigma_{-1,0,1}$.
		The further point $\la_r$ and the lines $\gamma_A$, $\gamma_S$ are used in the proof of Lemma~\ref{lem:errors}.} \label{fig:stokescomplex}
\end{figure}

Of the three Stokes lines emanating from $\la_+$, one is the boundary between the Stokes sector~$\Sigma_0$ and~$\Sigma_{+1}$, another one is the boundary between the Stokes sector~$\Sigma_0$ and~$\Sigma_{-1}$, and the third one is the boundary between~$\Sigma_{+1}$ and~$\Sigma_{-1}$. Notice that the region~$\Lambda_{+}$, where the behaviour of~$\psi_+$ is known from Proposition~\ref{pro:estimateat0}, lies in the proximity of the third Stokes line, along which the solutions are oscillatory in nature\footnote{Notice that the notion of Stokes sectors $\Sigma_{k}$ that we use in WKB approximation does not coincide with the notion of Stokes sectors~$\Omega_k$ used in Theorem \ref{thm:hermiteanharmonic}. However for large $\la$ their boundaries are asymptotic.}.

Therefore in order to compare $\psi_0$ and $\psi_+$ we need to compute the asymptotic behaviour of~$\psi_0$ in the union of the Stokes domains~$\Sigma_1$ and~$\Sigma_{-1}$. To this aim we express~$\psi_0$ as a linear combination of the solutions $\psi_{\pm1}(\la)$ subdominant in the Stokes domains~$\Sigma_{\pm1}$. Before doing so we need to def\/ine the error function $\rho_{\pm1}$ that controls the WKB estimates,
\begin{gather*}
\rho_{\pm1}(\la)=\inf_{\ell}\int_{\ell}\left|\frac{4V''(\mu)V(\mu)-5V'(\mu)^2}{4 V(\mu)^{\frac52}}\right|{\rm d}\mu.
\end{gather*}
Here $\ell\colon [0,1) \to \bb{C}$ is any path such that:
\begin{enumerate}[label=(\roman*)]\itemsep=0pt
	\item $\ell(0)=\la$;
	\item $\pm \Re S_A(\ell(t))$ is monotonically increasing;
	\item $\lim\limits_{t\to 1}\Re S_A(\ell(t))=\pm \infty$;
\end{enumerate}
see \cite{piwkb} for more details. We have the following lemma.
\begin{lem}\label{lem:wkbturning}
	There exist subdominant solutions $\widetilde{\psi}_{\pm1}$ in $\Sigma_{\pm1}$ of equation~\eqref{eq:h1asym}, uniquely cha\-racterised by the asymptotic behaviour
	\begin{gather*}
	\widetilde{\psi}_{\pm1}(\la;\al)=V_A(\la)^{-\frac14}e^{i \pm ( S_A(\la)-\frac{\pi}{4}) } \big (1 + o(1) \big), \qquad \la \in
	\Sigma_{\pm1} \qquad \mbox{and} \qquad |\la|\gg0 .
	\end{gather*}
	Furthermore the following hold true:
	\begin{enumerate}[label={\rm (\roman*)}]\itemsep=0pt
\item	Both these subdominant solutions depend holomorphically on $\alpha$ in $D$.
\item The solution $\widetilde{\psi}_0=\widetilde{\psi}_1-\widetilde{\psi}_{-1}$ is subdominant in~$\Sigma_0$.
\item The solution $\widetilde{\psi}_0(\la)$ has the following asymptotic expansion in the union of the Stokes domains
		$\Sigma_{+1}$ and~$\Sigma_{-1}$,
		\begin{gather}\nonumber
		\widetilde{\psi}_0(\la) =V_A(\la)^{-\frac14}\big[c_1(\la) e^{i ( S_A(\la)-\frac{\pi}{4})}
		-c_{-1}(\la)e^{-i(S_A(\la)-\frac{\pi}{4})}\big],\\ \label{eq:basicwkb}
		\frac{\partial \widetilde{\psi}_0(\la)}{\partial \la} =-V_A(\la)^{\frac14}\big[\tilde{c}_1(\la) e^{i ( S_A(\la)-\frac{\pi}{4} )}+
		\tilde{c}_{-1}(\la)e^{-i (S_A(\la)-\frac{\pi}{4}\ )}\big],
		\end{gather}
		for some multiplicative correction terms $c_{\pm1}$, $\tilde{c}_{\pm1}$ such that $|c_{\pm1}(\la)-1|=O\big(\rho_{\pm1}(\la)\big) $ and $|\tilde{c}_{\pm1}-c_{\pm1}|\leq \big|\frac{V'(\la)}{4 V(\la)^{\frac32}} c_{\pm1}\big|$.
	\end{enumerate}
\end{lem}
\begin{proof} i) For each $\al$, the solution $\widetilde{\psi}_{\pm1}(\la;\al)$ are just multiples of the subdominant solu\-tions~$\psi_{\pm1}$ def\/ined in equation~\eqref{eq:scalarpsi10}. More precisely we have
	\begin{gather*}
	\widetilde{\psi}_{\pm1}=\psi_{\pm1} \times
	\exp\left\lbrace\lim_{\la \to \pm i \infty} \left( S_A(\la)- g\big(\la,E^{\frac12}\al\big)+\frac{E}{2} \log \la \right) \right\rbrace ,
	\end{gather*}
	where the `$g$-function' $g(\la,z)$ was def\/ined in~\eqref{eq:gfunction}. It is a standard result that the subdominant solutions (sometimes called Sibuya's solutions) $\psi_{\pm1}(\al)$ are entire functions of $\al$ with that given normalisation, see, e.g.,~\cite{sibuya75}. Since the exponential factor is a holomorphic function of $\al \in D$, the solutions $\widetilde{\psi}_{\pm1}$ are as well.
	
ii) Since, by construction, $\widetilde{\psi}_1$ and $\widetilde{\psi}_{-1}$ have the same dominant asymptotic behaviour in~$\Sigma_0$, their dif\/ference is subdominant.
	
iii) The estimate (\ref{eq:basicwkb}) follows from the standard WKB estimates (see, e.g., \cite[Appendix]{piwkb} or~\cite{fedoryuk}).
	\begin{gather*}\nonumber
	 \widetilde{\psi}_{\pm1}(\la) =V_A(\la)^{-\frac14}c_{\pm1}(\la) e^{ \pm i ( S_A-\frac{\pi}{4}) }, \\
	 \frac{\partial \widetilde{\psi}_{\pm1}(\la)}{\partial \la} =-V_A(\la)^{\frac14} \tilde{c}_{\pm1}(\la) e^{ \pm i ( S_A-\frac{\pi}{4})},
	\end{gather*}
	which f\/inishes the proof of the lemma.
\end{proof}

The error functions $\rho_{\pm1}$ of the WKB approximation are estimated in the following lemma.
\begin{lem}\label{lem:errors}
	Given the hypothesis in Proposition~{\rm \ref{pro:wkbasy}}, there exist constants $E_0$, $C$ such that, for $E\geq E_0$, the following inequality holds
	\begin{gather*}
	\rho_{\pm1}(\la)\leq C E^{- 2 \gamma},
	\end{gather*}
	for all $(\la,\al) \in \Lambda_+ \times D$.
\end{lem}
\begin{proof}
		The proof is in Appendix \ref{appendix:wkb}.
\end{proof}

We can now complete the proof of Proposition \ref{pro:wkbasy}. To this aim, we set $\psi_0(\la;\al)=E^{\frac14}\widetilde{\psi}_0(\la;\al)$, where $\widetilde{\psi}_0$ is the solution studied in Lemma~\ref{lem:wkbturning}(ii). Due to Lemma \ref{lem:errors} and Lemma~\ref{lem:wkbturning}(iii), there exists a~$C$ such that
\begin{gather*}
|1-c_{\pm1}|\leq C E^{-2\gamma}, \qquad |1-\widetilde{c}_{\pm1}|\leq C E^{-2\gamma},
\end{gather*}
for all $(\la,\al) \in \Lambda_+ \times D$. In fact, the f\/irst estimate follows directly from the estimate $|1-c_{\pm1}|=\mathcal{O}(\rho_{\pm1})$, the second arises from the previous estimate and the fact that also $\frac{V'(\la)}{4 V(\la)^{\frac32}} = \mathcal{O}(E^{-2\gamma})$ for $(\la,\al) \in \Lambda_+ \times D$. The estimate~(\ref{eq:asymptransition}) now follows from Lemma~\ref{lem:phasefunction}.

\subsection{Matching and asymptotics of poles}\label{sub:matching}
In this subsection we complete the proof of Theorem~\ref{thm:asymptoticzeros}, by comparing the asymptotic form of the solutions
$\psi_+(\la)$ and $\psi_0(\la)$ of equation~\eqref{eq:h1asym}.
These were computed in Propositions~\ref{pro:estimateat0} and~\ref{pro:wkbasy}, on the respective domains $\Lambda_+$ and $\Lambda_0$; these domains satisfy $\Lambda_+\subset \Lambda_0$.

According to these results, on the domain $\Lambda_{+}$ the solutions $\psi_0$ and $\psi_+$ share the same trigonometric form
but for a generally dif\/ferent phase, see equations \eqref{eq:estimate0} and \eqref{eq:asymptransition}.
The phase-dif\/ference is the following function of $\al$,
\begin{gather}\nonumber
\Phi(\al;E)=\frac{E}{2}\big({-}\al\big(1-\al^2\big)^{\frac12}+\cos^{-1}(\al)\big)-\frac{(2+n)\pi}{4} \\ \label{eq:Phi}
\hphantom{\Phi(\al;E)=}{}+
i\left( j \log\big(2\big(1-\al^2\big)^{\frac34}E^{\frac12}\big) -\frac{1}{2} \log F_{n,j}\right) .
\end{gather}
Here $F_{n,j}=\frac{\Gamma\big[\tfrac{1}{2}(1+n+j)\big]}{\Gamma\big[\tfrac{1}{2}(1+n-j)\big]}$, as def\/ined in~\eqref{eq:Fnj}. It plays a major role in our analysis. In fact, we can express the Wronskian $W[\psi_0,\psi_+]=\psi_0\psi_+'-\psi_0'\psi_+$ in terms of $\Phi(\al;E)$.

\begin{pro}\label{pro:wronskian}
	Here $D$ is a simply connected compact subset of the vertical strip $-1<\Re \al <1$ and
	$\psi_+(\la)$ and $\psi_0(\la)$ are the solutions of \eqref{eq:h1asym} studied in Proposition~{\rm \ref{pro:estimateat0}} and~{\rm \ref{pro:wkbasy}}.
	For $E$ big enough we have the following decomposition of the Wronskian,
	\begin{gather}\label{eq:wrdec}
	E^{-\frac12}W[\psi_0,\psi_+]=\sin\big(\Phi(\al;E)\big) + E^{-\frac13} R(\al,E) ,
	\end{gather}
	where $\Phi(\al;E)$ is defined in \eqref{eq:Phi} and
	$R(\al,E)$ is an analytic function of $\al \in D$ such that
	\begin{gather*}
	\sup_{\al \in D,\, E\geq E_0} \left| \frac{R(\al,E)}{\exp{\big( i\Phi(\al,E)\big)}}\right|< +\infty .
	\end{gather*}
\end{pro}
\begin{proof}
	$R$ is analytic because $\psi_0$, $\psi_+$ and $\sin(\Phi(\al;E))$ are analytic in $D$. The estimate is a~direct consequence of Proposition~\ref{pro:estimateat0} and~\ref{pro:wkbasy}, specialised to $\gamma=\frac13$.
\end{proof}

According to the proposition above, the asymptotic distribution of the zeros of $H_{m,n}(\alpha)$ with $n$ bounded is determined by the zeros of the transcendental equation $\sin(\Phi(\al;E))=0$. Equivalently
\begin{gather}
 \frac{1}{2}\big(\al\big(1-\al^2\big)^{\frac12}-\cos^{-1}(\al)\big)+ \frac{\pi}{4}
 - E^{-1} i\left( j \log\big(2\big(1-\al^2\big)^{\frac34}E^{\frac12}\big) -\frac{1}{2} \log F_{n,j} \right) =
\frac{\pi k}{E}\label{eq:wronskian}
\end{gather}
for some $k$, integer if $m$ is odd, half-integer if $m$ is even, whose range will be determined below.

Notice that for any compact subset of the horizontal strip $-1<{\Re} \al <1$, equation \eqref{eq:wronskian} is just a~vanishing perturbation of the equation
$f(\al)=\frac{\pi k}{E}$, where $f(\al)= \frac{1}{2}\big(\al(1-\al^2)^{\frac12}-\cos^{-1}(\al)\big)+\frac{\pi}{4} $, because
$1/f'(\al)=(1-\al^2)^{-\frac12}$ is bounded. The equation $f(\al)=\frac{\pi k}{E}$
has one and only one (and real) solution in the 'bulk' for each $|k|\leq \frac E4$ and no solutions at all if $|k|>\frac E4$.

Since we are interested in the solutions of equation \eqref{eq:wronskian} only up to an additive error of the order $E^{-\frac43}$, equation \eqref{eq:wronskian} can be greatly simplif\/ied. Firstly, the imaginary part of the solutions is of order $E^{-1}\log{E}$ and thus we can linearise the system with respect to $\al_I=\Im \al$ at $\al_I=0$. Secondly, we can discard all terms of order $E^{-\delta}$, $\delta >\frac43$ in the equation that
we obtained after the linearisation with respect to $\al_I$.

These two simplif\/ications lead to the system \eqref{eqrealandimaginaryI}, whose solutions $\al_{j,k}$ we described at the beginning of the present
section,
and to the following lemma, whose details of the proof are left to the reader.

\begin{lem}\label{lem:system} For $j\in J_n$ denote, as in Definition~{\rm \ref{def:aljk}}, by $\al_{j,k}$ the unique solution of the system~\eqref{eqrealandimaginaryI}.	 Fix $0<\sigma<\frac14$, set $d_R:=f^{-1}\big(\frac{\pi}{4}(1-4 \sigma)\big)$ and let $d_I>0$ be such that
$D_0:=\{|\Re{\al}|\leq d_R,|\Im{\al}|\leq d_I\}$ contains all points $\al_{j,k}$, with $|k| \leq \sigma E$ and $j\in J_n$, for $E$ large enough, say $E\geq E_0$.
	
Now, for an arbitrarily small but fixed $\e$, let $D:=\{|\Re{\al}|\leq d_R+\e,|\Im{\al}|\leq d_I\}$. Then there exists an $E_1\geq E_0$ and a constant $C$ such that, for all $E\geq E_1$, $|k| \leq \sigma E$ and $j\in J_n$, equation~\eqref{eq:wronskian} has a unique solution $\tilde{\al}_{j,k}$ in $D$ and
	\begin{gather*}
	|\tilde{\al}_{j,k}- \al_{j,k}|\leq C E^{-2} \log^2 E .
	\end{gather*}
\end{lem}

We have collected all the intermediate results required to prove Theorem \ref{thm:asymptoticzeros}, that we state again here.
\begin{thm*}
Fix $0<\sigma<\frac14$, then there exist a constant $E_0$ and $C_{\sigma}$ such that, for all $E\geq E_0$ and $|k|\leq \sigma E$, the polynomial~$H_{m,n}(z)$ has one and only one zero in each disc of the form $\big\{\big|z-E^{\frac12}\al_{j,k}\big|\leq C_{\sigma} E^{-\frac56}\big\}$.
\end{thm*}
\begin{proof} We follow the proof of \cite[Theorem 1]{piwkb2} quite closely. Because of Lemma \ref{lem:system}, the thesis is equivalent to: $W[\psi_0,\psi_+](\al;E)$ and $\sin(\Phi(\al;E))$ have the same number of zeros (namely one) in the disk $|\al-\al_{j,k}|\leq C E^{-\frac43}$.
	
The proof of this statement is based on Proposition \ref{pro:wronskian} and the classical Rouch\'e's theorem: Let $K$ be a compact domain of the complex plane with simple boundary $\partial K$. If $p$ and $q$ are holomorphic functions, in a larger domain~$K'$, such that $|q| <|p|$ on $\partial K$, then $p$ and $p+q$ have the same number of zeros~-- counting multiplicities~-- inside~$K$.
	
In our setting $K$ is a disc centred in $\al_{j,k}$ with radius $\kappa E^{-\rho}$ for some $\kappa$, $\delta$ to be def\/ined below, $p=\sin \big(\Phi(\al;E)\big) $ and $q=E^{-\frac12} W[\psi_0,\psi_+](\al;E)-\sin\big(\Phi(\al;E)\big)$.
	
Since $\sigma<\frac14$, the $\al_{j,k}$ belong to a bounded subset $D$ of the bulk and	$|\al_{j,k}^2-1|$ is therefore uniformly bounded from below by a~positive constant.	Hence, by means of a Taylor expansion estimate, for any $\kappa>0$, $\rho>0$ there exists some numbers, $c_{j,k}$ uniformly bounded (with the bound depending just on $\kappa,\rho$), such that $||p(\al)\big|{-}E^{1-\rho}\kappa |1-\al_{j,k}^2|^{\frac12}\big|\leq c_{j,k}E^{1-2\rho}\kappa^2$ for any $\al$ belonging to the circle
	with centre $\al_{j,k}$ and radius $\kappa E^{-\delta}$.
	
	Similarly on the union of all these circles $|\exp{\big( i\Phi(\al,E)\big)}|$ is uniformly bounded and therefore,
	by Proposition \ref{pro:wronskian} there exists a $C'$ such that, for $E$ large enough,
	$|q| \leq E^{-\frac13} C'$ on the union of all these circles.
	
	Hence, for $E$ big enough, if $\rho=\frac43$ and $\kappa >2 C' \sup\limits_{j,k}|1-\al_{j,k}^2|^{-\frac12}$, then
	$|q|<|p|$ on all the circles with centre $\al_{j,k}$ and radius $\kappa E^{-\rho}$. This concludes the proof.
\end{proof}

\subsection{Proof of Theorem \ref{thm:asymptoticzerosedge}.}\label{sub:thm2}
Recall that in stating and proving Theorem \ref{thm:asymptoticzeros} and all necessary preparatory lemmas and propositions,
we worked with a compact subset $D$ of the
$\al$-plane not containing the points $\al=\pm1$. This is because all estimates involved break down at those two points.

However, we can improve our results without modifying the procedures if we let $D_E$ vary with $E$ and approach the points $\pm1 $
slowly enough. To this aim we f\/ix $\delta>\frac13$ and we assume that the distance of $D_E$ from $\al= \pm1$ is greater
or equal to a constant times $E^{-1+\delta}$.

Indeed, we can prove Theorem \ref{thm:asymptoticzerosedge} that we state again here for convenience.
\begin{thm*}
	Fix $\frac13<\delta\leq1$ and $s >0$. There exist a constant $E_0$ and a constant $C_{s}$ such that, if $E\geq E_0$, then for all $|k|\leq \big( \frac{1}{4}-s E^{3\frac{-1+\delta}{2}} \big) E$, the polynomial $H_{m,n}(z)$ has one and only one zero in each disc of the form $\big|z-E^{\frac{1}{2}}\al_{j,k}\big|\leq C_s E^{-\delta+\frac16}$.
\end{thm*}
	\begin{proof} We assume that $\al \in D_E$, where $D_E$ is a compact simply connected domain whose distance to the points $\al=\pm1$ is bigger or equal than a constant times~$E^{1-\delta}$. In other words, there exists a $c>0$ such that $\sup\limits_{\al \in D_E}|(1-\al^2)^{-1}| \leq c E^{\delta-1}$.
		
It turns out that the statement of Propositions \ref{pro:apparent}, \ref{pro:estimateat0}, \ref{pro:wkbasy} and \ref{pro:wronskian} holds true if we appropriately modify the error bounds.
		
Below, we list the necessary modif\/ication, whose verif\/ication is straightforward.
		\begin{itemize}\itemsep=0pt
			\item[(i)] In Proposition~\ref{pro:apparent}(ii). The equation~\eqref{eq:estimate0} must be modif\/ied as follows.
			For $j \in J_n$, the branch $\beta_j(\al)$ has the following asymptotic expansion
			\begin{gather*}
			\beta_j(\al,E)=j i \big(1-\al^2\big)^{\frac12} \big(1+E^{\frac{1-3 \delta}{2}}r_j(\al,E))\big) ,
			\end{gather*}
			where $r_j(\al,E)$ is a bounded function on $D_E \times [E_0,\infty)$.
			\item[(ii)]Proposition \ref{pro:estimateat0}, having f\/ixed $\gamma=\frac{3 \delta-1}6$,
			holds true if the error terms in the estimate
			\eqref{eq:estimate0}
			are modif\/ied as follows
			\begin{gather*}\nonumber
			 |\psi_+(\la)-\Sigma(\la;j,\varphi,\chi \big)| \leq C E^{-\frac{3\delta-1}{6}} , \\
			\big| E^{-\frac12} \psi_+^{\prime}(\la)+
			\Sigma\big(\la;j,\varphi+\tfrac{1}{2}\pi,\chi\big)\big| \leq C E^{-\frac{3\delta-1}{6}} .
			\end{gather*}
			\item[(iii)]Proposition \ref{pro:wkbasy}, having f\/ixed $\gamma=\frac{3 \delta-1}6$,
			holds true if the error bound in equations \eqref{eq:estimate0}
			is modif\/ied as follows
			\begin{gather*}
			 \big|\psi_0(\la;\al)-\Sigma(\la;j,\varphi',\chi') \big|\leq C E^{-\frac{3\delta-1}{6}} \big(|\Sigma(\la;j,\varphi',\chi')|+
			\big|\Sigma\big(\la;j,\varphi'+\tfrac{1}{2}\pi,\chi'\big)\big|\big), \\
			 \left|E^{-\frac12} \frac{\partial \psi_0(\la; \al)}{\partial \la}+\Sigma\big(\la;j,\varphi'+\tfrac{1}{2}\pi,\chi'\big)\right|\\
\qquad{} \leq C E^{-\frac{3\delta-1}{6}} \big(|\Sigma(\la;j,\varphi',\chi')|+\big|\Sigma\big(\la;j,\varphi'+\tfrac{1}{2}\pi,\chi'\big)\big|\big) .
			\end{gather*}
			\item[(iv)]Proposition \ref{pro:wronskian} holds true if the error bound in equation \eqref{eq:wrdec}
			is modif\/ied as follows
			\begin{align*}
			& E^{-\frac12}W[\psi_0,\psi_+]=\sin\big(\Phi(\al;E)\big)
			+ E^{-\frac{3\delta-1}{6}} R(\al,E) .
			\end{align*}
			\item[(v)]Finally the estimate of Lemma \ref{lem:system} about the distance among solutions $\tilde{\al}_{j,k}$
			of equation~\eqref{eq:wronskian}
			and the points $\al_{j,k}$ is modif\/ied as follows: For $k$ in the range of the present Theorem,
			there exists a constant $C$
			depending on $s$ such that
	 \begin{gather*}
	 |\tilde{\al}_{j,k}- \al_{j,k}|\leq C E^{-2 \delta} \log^2 E .
	 \end{gather*}
		\end{itemize}
		We notice that as $\Re \al \to 1^- $, we have $f(\al)=\frac{2^{\frac32}}{3}(1-\al)^{\frac32}+\mathcal{O}\big((1-\al)^{\frac52}\big)$
		(and a similar behaviour for $\Re \al \to -1^+$). Therefore, we can appropriately choose a $c>0$ such that all solutions of the system~\eqref{eqrealandimaginaryI}, with $k$ in the given range, belong to the strip $ -1+cE^{-1+\delta}\leq \Re \alpha\leq 1-cE^{-1+\delta}$.
	
The proof can then be continued following the very same steps of the proof of Theorem~\ref{thm:asymptoticzeros}; we prove that $H_{m,n}\big(E^{\frac12}\al\big)$ has one and only one zero in each disc of the form $|\al-\al_{j,k}|\leq C_s E^{-\frac{5+3 \delta}6}$. Indeed we again set $p=\sin (\Phi(\al;E) ) $ and $q=E^{-\frac12} W[\psi_0,\psi_+](\al;E)-\sin (\Phi(\al;E) )$.
		
		We f\/ix some $\kappa,\rho>0$ and estimate the value of $|p(\al)|$ for $\al$ belonging to the circle with centre
		$\al_{j,k}$ and radius $\kappa E^{-\rho}$. By means of a Taylor series estimate,
		and taking in account that $|1-\al^2|\geq c E^{-1+\delta}$ and $\delta>\frac13$, we obtain the following lower bound of $p$
		on the aforementioned circles: if $E$ is big enough, there exists a
		$c'$ (depending on $c$) such that, on the union of all these circles,
		$|p(\al)|\geq c' E^{\frac{1+\delta}{2}-\rho}$ provided $\rho>1-\delta$.
		Due to the above modif\/ication of Proposition \ref{pro:wronskian}, there exists a $c''$ (depending on $c$) such that
		on the union of the circle $|q(\al)|\leq c'' E^{-\frac{3\delta+1}{6}} $.
		
		Therefore as long as $1-\delta<\rho\leq \delta+\frac13$, and $E$ is big enough, we can f\/ind
		a $\kappa$ such that $|q| < |p|$ on the aforementioned circle. The thesis then follows from Rouch\'e's theorem.
	\end{proof}
	
\subsection{Edge asymptotics}\label{sub:edge}
In this last section of the paper we brief\/ly discuss the edge asymptotic in the regime $n$ bounded and $E=2m+n \to \infty$.
We anticipate below some preliminary results about the edge-asymptotic and we postpone more details and full proofs to a further publication.

Recall that the proof of Theorem \ref{thm:asymptoticzerosedge} breaks down when $1-\al^2=\mathcal{O}\big(E^{-\frac23}\big)$, because in that regime
the error of our asymptotic method is of the same order as the distance between the roots, as depicted in Fig.~\ref{fig:edge}.
\begin{figure}[ht]
	\centering
	\begin{tabular}{ c c }
		{\setlength{\fboxsep}{0pt}%
			\setlength{\fboxrule}{0.2pt}%
			\fbox{\includegraphics[width=7cm]{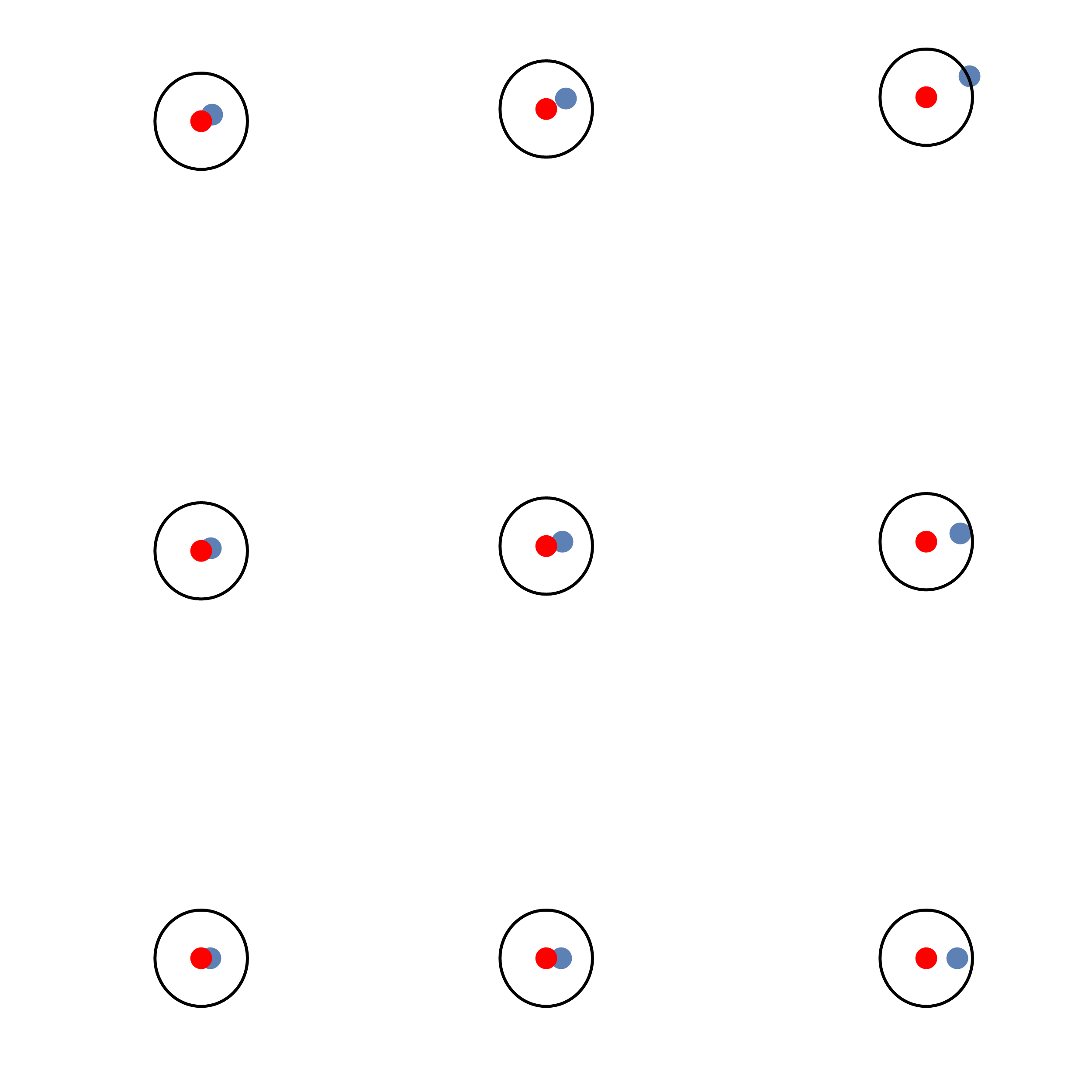}}} &	{\setlength{\fboxsep}{0pt}%
			\setlength{\fboxrule}{0.2pt}%
			\fbox{\includegraphics[width=7cm]{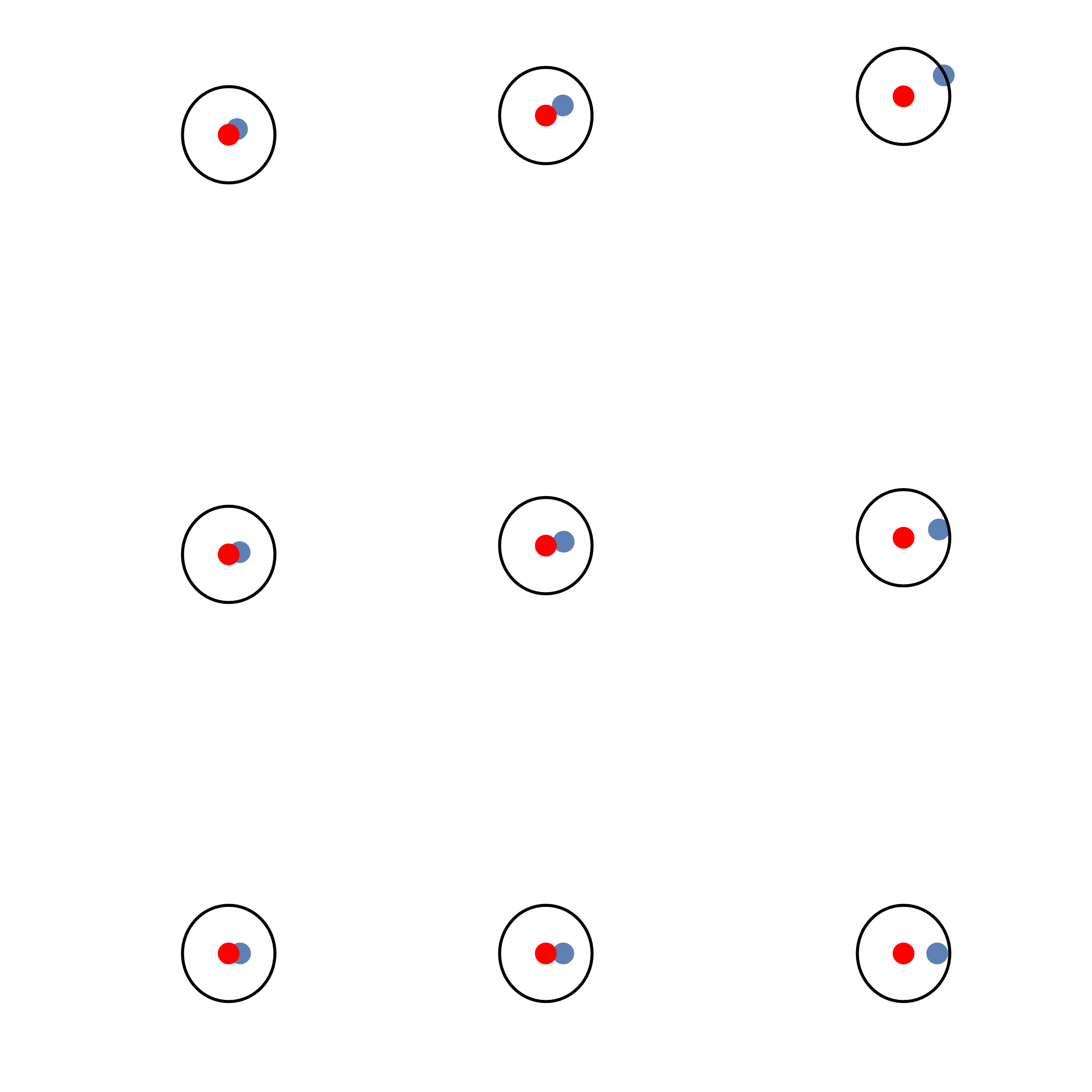}}}\\
	$m=16$ & $m=144$
	\end{tabular}
	\caption{Rescaled roots of $H_{m,5}$ in upper right corner. In blue numerically exact location, in red the asymptotic approximation $\al_{j,k}$, encircled by a circle of radius $C_s E^{-\frac{2}{3}}$, with $C_s=\frac{1}{6}$, for the range $\frac{m+1}{2}-k\in\{0,1,2\}$, $j\in\{0,2,4\}$. Even though the approximation by~$\al_{j,k}$ still look reasonably good, it does not give the dominant term in the asymptotic expansion of extreme roots. This is ref\/lected in the fact that the absolute error of Theorem \ref{thm:asymptoticzerosedge}, if extrapolated at the value $\delta=\frac13$,
	is of the same order of the distance between roots, namely~$E^{-\frac{2}{3}}$ .} \label{fig:edge}
\end{figure}

However, our procedures can be adapted appropriately to study the limit $1-\al^2=\mathcal{O}\big(E^{-\frac23}\big)$. Without losing generality we restrict our discussion to the case $\al \to 1$.

As a f\/irst step we make the Ansatz $1-\al=\kappa E^{-\frac23}$ and look for $\kappa$ of order $E^0$ that solves the inverse monodromy problem for Hermite~I rational solutions of ${\rm P}_{\rm IV}$, equation~\eqref{eq:hIK}. Upon inserting the Ansatz in equation~\eqref{eq:hIK}, redef\/ining the unknown
$\beta \to E^{\frac13}\beta$, and discarding lower order contributions, equation~\eqref{eq:hIK} reads
\begin{gather*}\nonumber
\psi''(\lambda)=V_{\kappa} (\lambda;\kappa,\beta,E,n )\psi(\lambda), \\ 
 V_{\kappa}(\la)= \lambda^2+2 E^{\frac12}\lambda-2 \kappa E^{\frac13}
-\frac{E^{\frac16 }\beta}{\lambda}+\frac{n^2-1}{4 \la^2} .
\end{gather*}
We look for $(\kappa,\beta)$ such that no-logarithm and quantisation conditions hold.

{\bf No-logarithm in the edge.} Reasoning as in Proposition \ref{pro:apparent}, we can express the unknown $\beta$ as an n-valued function of the unknown~$\kappa$. To this aim, we apply the change of coordinates $\nu=E^{-\frac16} \la$. The resulting equation reads
\begin{gather*}
\psi''= \left(E^{-\frac23}\nu^2 + 2 \nu-2\kappa - \frac{\beta}{\nu} + \frac{n^2-1}{4 \nu^2}\right) \psi .
\end{gather*}
Therefore $\beta=\beta_j(\kappa) \big(1+\mathcal{O}(E^{-\frac23})\big)$ where $\beta_j (\kappa)$, $j \in J_n $ is one of the $n$ solutions of the no-logarithm condition for the equation
\begin{gather}\label{eq:hedgenoinfinity}
\psi''= \left( 2 \nu-2\kappa - \frac{\beta}{\nu} + \frac{n^2-1}{4 \nu^2}
\right) \psi .
\end{gather}
The above equation does not belong to the any standard class of linear special functions listed in~\cite{nist} or~\cite{bateman2}. In particular, the solution of the no-logarithm condition for the latter equation does not seem to have a simple factorisation. For the sake of the present discussion we assume that for $\Re \kappa>0$, the functions $\beta_j(\kappa)$ are all distinct as in the case of the Whittaker equation.

{\bf Airy-like model near the origin.} Applying again the change of variable $\nu=E^{-\frac16} \la$, and reasoning as in Proposition~\ref{pro:estimateat0}, we deduce that the solution $\psi_+$ is asymptotic to the solution $u_+(\nu)$, subdominant at $\lambda=0$, of the equation
\begin{gather}\label{eq:hedgenoinfinityj}
u''(\nu)= \left( 2 \nu-2\kappa - \frac{\beta_j(\kappa)}{\nu} + \frac{n^2-1}{4 \nu^2}\right) u(\nu) .
\end{gather}

The latter equation is an Airy equation with a centrifugal term. The solution $u_+(\nu)$ is thus for generic values of $\kappa$
dominant at $+\infty$. However, we say that $\kappa$ is an eigenvalue of equation~\eqref{eq:hedgenoinfinityj} if
$u_+(\nu)$ is subdominant at~$+\infty$.

{\bf Matching.} The asymptotic analysis of the solution $\psi_0$ simplif\/ies in this regime. In fact, in the large $E$ limit, the turning points $\la_+$ and $\la_-$ coalesce with the Fuchsian singularity $\la=0$. Therefore the standard approximation $\psi_0\sim V_A^{-\frac14}e^{-S_A}$ is valid in a region close to $\la=0$, where the approximation of $\psi_+$ by $u_+$ is also valid. Finally, reasoning as in Proposition~\ref{pro:wronskian}, we deduce that, asymptotically,
the solutions $\psi_+$ and $\psi_0$ are proportional if and only if $\kappa$ is one of the eigenvalue of~(\ref{eq:hedgenoinfinityj}).\footnote{Remarkably a slightly simpler version of this problem was shown to 	model the growth-rate distribution of E.~Coli, see~\cite{cinese1} by one of the authors.}

{\bf Extremal asymptotic.} We have therefore arrived at a (preliminary) description for the edge asymptotic for roots with extreme real part,
without entering into the details of estimating the error of the approximation: Let $(\kappa,\beta)$ be such that the Airy-like equation~(\ref{eq:hedgenoinfinity}) satisf\/ies the no-logarithm condition and the quantisation condition $\lim\limits_{\nu \to + \infty} u_+(\nu) =0$. If $E$ is big enough, there exists a root of
the generalised Hermite polynomial $H_{m,n}(\alpha)$ with the asymptotic behaviour $1-\al=\kappa E^{-\frac23}+o(E^{-\frac23})$.

\begin{rem*}
	In the case $n=1$, the above description coincides with the well known edge-asymptotic for roots of the Hermite polynomial \cite{bateman2}.	
	Indeed, in the case $n=1$, the only solution of the no-logarithm condition is $\beta=0$: the equation \eqref{eq:hedgenoinfinity}
	reduces to the standard Airy equation with potential $2\la-2\kappa$ and $u_+$ is the solution of the Airy equation
	that vanishes at $\la=0$. Therefore the eigenvalues equal $\kappa_n=-2^{-\frac13}\e_n, n \geq 1$ where
	$\e_n$ is the $n$-th zero of the Airy function on the imaginary axis.
\end{rem*}

\section{Concluding remarks}
All singularities of Painlev\'e IV rationals can be expressed in terms of zeros of generalised Hermite~$H_{m,n}$
and Okamoto polynomials $Q_{m,n}$. We have characterised those zeros by means of the inverse monodromy problem for an anharmonic
oscillator of degree two. We have shown that these oscillators in turn characterise three special classes of
inf\/initely sheeted coverings of the sphere. In the case of generalised Hermite polynomials we have also explicitly computed the
monodromy representation of such coverings by means of line complexes; as a by-product we obtained the number of real roots.
Finally we have computed the asymptotic distribution in the bulk of zeros of
generalised Hermite polynomials in the regime $m\gg0$ and $n$ f\/ixed; we have also presented some preliminary results
about the edge asymptotics.

Many questions about rational solutions of Painlev\'e~IV remain open and we plan to address
some of them by means of the methods developed in the present paper.
These include, in the case of generalised Hermite polynomials,
a precise description of zeros in the edge~-- in particular the cross-over between the bulk and the edge behaviour described by the
large $\kappa$ limit of equation~\eqref{eq:hedgenoinfinityj}~-- and the asymptotic distribution of zeros in the more general regime
$r:=m/n\in (0,1]$ f\/ixed and both $m,n\rightarrow \infty$.
The latter is in principle a straightforward extension of our results,
even though the error analysis presents some new analytical challenges.

The real testing ground of our methods is the analysis of Okamoto polynomials.
On the theoretical side, we have yet to f\/ind a combinatorial classif\/ication
of line complexes of the functions in the class $F_3^n$ discussed in Section \ref{section:nevanlinna}. Such a classif\/ication
would lead, along the lines of Corollary \ref{cor:numberrealroots}, to compute the number of real zeros.
We also plan to tackle the asymptotic analysis of the distribution of zeros of generalised Okamoto polynomials, but at this stage it is uncertain
whether our methods allows to deal with the greater complexity of these polynomials.

\appendix
\section{Generalised Hermite and Okamoto polynomials}
\label{appendix:generalisedhermite}
The generalised Hermite polynomials $H_{m,n}(z)$, where $m,n\in\mathbb{N}$, are recursively def\/ined by
	\begin{gather*}
	2mH_{m+1,n}H_{m-1,n}=H_{m,n}H_{m,n}''- (H_{m,n}' )^2+2mH_{m,n}^2,\\
	2nH_{m,n+1}H_{m,n-1}=-H_{m,n}H_{m,n}''+ (H_{m,n}' )^2+2nH_{m,n}^2,
	\end{gather*}
with $H_{0,0}=H_{1,0}=H_{0,1}=1$ and $H_{1,1}=2z$. Specialising to $n=1$, we obtain the classical $m$-th Hermite polynomial, that is,
\begin{gather*}
H_{m,1}(z)=(-1)^me^{z^2}\frac{\partial^m}{\partial z^m}\big[e^{-z^2}\big].
\end{gather*}
Similarly the generalised Okamoto polynomials $Q_{m,n}$, where $m,n\in\mathbb{Z}$, are def\/ined via the recursion
\begin{gather*}
Q_{m+1,n}Q_{m-1,n}=\tfrac{9}{2}\big(Q_{m,n}Q_{m,n}''- (Q_{m,n}' )^2\big)+\big(2z^2+3(2m+n-1)\big)Q_{m,n}^2,\\
Q_{m,n+1}Q_{m,n-1}=\tfrac{9}{2}\big(Q_{m,n}Q_{m,n}''+ (Q_{m,n}' )^2\big)+\big(2z^2+3(1-m-2n)\big)Q_{m,n}^2,
\end{gather*}
with $Q_{0,0}=Q_{1,0}=Q_{0,1}=1$ and $Q_{1,1}=\sqrt{2}z$. The polynomials $Q_{m,0}$ and $Q_{m,1}$ with $m\in\mathbb{Z}$, were originally introduced by Okamoto \cite{okamoto}, to describe the subcases of~\eqref{eq:parameterokamoto} where $\theta_0\in\big\{\pm \frac{1}{6},\pm \frac{1}{3}\big\}$.
\begin{pro}\label{pro:elementaryprop}
	For $m,n\in\mathbb{N}$, the polynomial $H_{m,n}$ has only simple roots and none common with $H_{m+1,n}$, $H_{m,n+1}$ and $H_{m-1,n+1}$. Furthermore $H_{m,n}\big(\frac{1}{2}z\big)$ is monic of degree $mn$ with integer coefficients and we have the symmetry
	\begin{gather*}
	H_{m,n}(iz)=i^{mn}H_{n,m}(z).
	\end{gather*}
	For $m,n\in \mathbb{Z}$, the polynomial $Q_{m,n}$ has only simple roots and none common with $Q_{m+1,n}$, $Q_{m,n+1}$ and~$Q_{m-1,n+1}$. Furthermore $Q_{m,n}\big(\frac{1}{\sqrt{2}}z\big)$ is monic of degree $d_{m,n}=m^2+n^2+mn-m-n$ with integer coefficients and we have the symmetries
	\begin{gather*}
	Q_{m,n}(iz)=i^{d_{m,n}}Q_{n,m}(z),\qquad 
	Q_{m,n}(z)=Q_{-m-n+1,m}(z).
	\end{gather*}
\end{pro}
\begin{proof}
	See Noumi and Yamada \cite{noumiyamada}.
\end{proof}

\section{B\"acklund transformations}\label{section:backlund}
Okamoto \cite{okamoto} f\/irst uncovered that ${\rm P}_{\rm IV}$ enjoys the action of the af\/f\/ine Weyl group $A_1^{(1)}$ by means
of B\"acklund transformations on its solutions, see also Noumi and Yamada \cite{noumiyamada}.
In this study we only require the including translations, generated by
\begin{alignat}{3}
&\mathcal{R}_1\colon \quad&& \om^{(1)}= \frac{(\om'+4\theta_0)^2+8(\theta_\infty-\theta_0)\om^2-\om^2(\om+2z)^2}{2\om(\om^2+2z\om-\om'-4\theta_0)},&\nonumber \\
&\mathcal{R}_2\colon \quad &&\om^{(2)}= \frac{(\om'-4\theta_0)^2+8(\theta_\infty-1-\theta_0)\om^2-\om^2(\om+2z)^2}{2\om(\om^2+2z\om+\om'-4\theta_0)},&\nonumber\\
&\mathcal{R}_3\colon \quad & &\om^{(3)}= \frac{(\om'-4\theta_0)^2+8(\theta_\infty+\theta_0)\om^2-\om^2(\om+2z)^2}{2\om(\om^2+2z\om-\om'+4\theta_0)},&\nonumber\\
&\mathcal{R}_4\colon \quad &&\om^{(4)}= \frac{(\om'+4\theta_0)^2+8(\theta_\infty-1+\theta_0)\om^2-\om^2(\om+2z)^2}{2\om(\om^2+2z\om+\om'+4\theta_0)},&\label{eq:backluend1}
\end{alignat}
where the complex parameters are transformed correspondingly as
\begin{alignat*}{5}
& \theta_0^{(1)}=\theta_0-\tfrac{1}{2},\qquad &&\theta_0^{(2)}=\theta_0+\tfrac{1}{2},\qquad &&
\theta_0^{(3)}=\theta_0+\tfrac{1}{2},\qquad &&\theta_0^{(4)}=\theta_0-\tfrac{1}{2},& \\
& \theta_\infty^{(1)}=\theta_\infty+\tfrac{1}{2},\qquad &&\theta_\infty^{(2)}=\theta_\infty-\tfrac{1}{2},\qquad &&\theta_\infty^{(3)}=\theta_\infty+\tfrac{1}{2},\qquad &&\theta_\infty^{(4)}=\theta_\infty-\tfrac{1}{2}.&
\end{alignat*}
It is these B\"acklund transformations, which can be extended to Schlesinger transformations of the associated linear system~\eqref{eq:laxpair}, as shown by Fokas et al.~\cite{fokasmugan1988}. Note that the following identities hold true
\begin{gather*}
\mathcal{R}_1\mathcal{R}_2=\mathcal{I},\qquad \mathcal{R}_3\mathcal{R}_4=\mathcal{I},\qquad \mathcal{R}_1\mathcal{R}_3=\mathcal{R}_3\mathcal{R}_1,
\end{gather*}
where $\mathcal{I}$ denotes the identity.

\section{WKB Computations}\label{appendix:wkb}
\begin{proof}[Proof of Lemma \ref{lem:phasefunction}]
	After the change of variable $\nu=E^{-\frac12}\mu+\al$, the integral reads
	\begin{gather*}
	S_A(\la)= E \int^{E^{-\frac12}\la_+ +\al}_{E^{-\frac12}\la+\al}
	\sqrt{1-\nu^2+ \frac{ji (1-\al^2)^{\frac12}}{E(\nu-\al)}-\frac{n^2-1}{4 E^2(\nu-\al)^2}} {\rm d}\nu,
	\end{gather*}
where, by equation \eqref{eq:la+}, we have $E^{-\frac12}\la_+ +\al= 1+\frac{ji}{2(1-\al)}E^{-1}+\mathcal{O}(E^{-2})$. We introduce a~f\/ixed but arbitrarily small $\e>0$
	and shift the upper integration limit to get
	\begin{gather*}
	S_A(\la)=E \int^{1+ i E^{-1+\e}}_{E^{-\frac12}\la+\al}
	\sqrt{1-\nu^2+ \frac{ji (1-\al^2)^{\frac12}}{E(\nu-\al)}-\frac{n^2-1}{4 E^2(\nu-\al)^2}} {\rm d}\nu +\mathcal{O}\big(E^{-\frac{1-3\e}{2}}\big).
	\end{gather*}
	Next, expanding the integrand about $E=\infty$, gives{\samepage
	\begin{gather}\nonumber
	S_A(\la)= \int^{1+i E^{-1+\e}}_{E^{-\frac12}\la+\al} \left(E\sqrt{1-\nu^2}+ \tfrac{1}{2}ji
	\frac{ (1-\al^2)^{\frac12}}{(\nu-\al)(1-\nu^2)^{\frac12}} \right) {\rm d}\nu \\ \label{eq:phaseappendix}
\hphantom{S_A(\la)=}{} +\int^{1+i E^{-1+\e}}_{E^{-\frac12}\la+\al} \sum_{k\geq2} E^{-k+1} \frac{c_k}{(\nu-\al)^k(1-\nu^2)^{k-\frac12}} {\rm d}\nu +\mathcal{O}\big(E^{-\frac{1-3\e}{2}}\big),
	\end{gather}
	where the series $\sum c_k w^k$ has positive radius of convergence $r_0$.}

In the f\/irst line, discarding again an error $\mathcal{O}\big(E^{-\frac{1-3\e}{2}}\big)$, the upper integration limit can be replaced by $1$, and the remaining integral can be written explicitly in terms of elementary functions. Indeed, the integral of the f\/irst term of the f\/irst line equals
	\begin{gather*}
	\tfrac{1}{4}E \big({-}2 \big(\al+\la E^{-\frac12}\big) \sqrt{1-\big(\al+\la E^{-\frac12}\big)^2}
	-2\arcsin\big(\al+\la E^{-\frac12}\big)+\pi \big),
	\end{gather*}
	which in the limit $\la\asymp E^{-\frac12+\gamma}$ yields
	\begin{gather*}
	\tfrac{1}{2}E\big({-}\al\big(1-\al^2\big)^{\frac12}+\arccos{\al}\big)-E^{\frac12}\big(1-\al^2\big)^{\frac12}\la+\mathcal{O}\big(E^{-1+2\gamma}\big) .
	\end{gather*}
	The second term of the f\/irst line of~\eqref{eq:phaseappendix} is equal to
	\begin{gather*}
	\tfrac{1}{2}ji\Big({-}\log\la +\tfrac12 \log E+
	\log\Big[\sqrt{\big(1-\al^2\big) \big(1-\big(\al+E^{-\frac12}\la\big)^2\big)}-\al \big(\al+E^{-\frac12}\la\big)+1 \Big] \Big),
	\end{gather*}
	and in the limit $\la\asymp E^{-\frac12+\gamma}$ yields
	\begin{gather*}
	\tfrac{1}{2}ji\big({-}\log\la +\log E^{\frac12}+
	\log2(1-\al^2)\big)+\mathcal{O}\big(E^{-1+\gamma}\big) .
	\end{gather*}
	
To conclude the proof, we show that for $\la\asymp E^{-\frac12+\gamma}$ the second line of \eqref{eq:phaseappendix} is of order $E^{-\gamma}$, if $\e$ is suf\/f\/iciently small. We therefore f\/ix $s \in \bb{C}$, $\Re s >0$ and let $\la=s E^{-\frac12 +\gamma}$ ($\gamma \geq0$).
	
Since the contribution of the endpoints is dif\/ferent, we split the integration contour in two, namely $\int_{E^{-\frac12}\la+\al}^{\frac{1+\alpha}{2}}$ and $\int_{\frac{1+\alpha}{2}}^{1+i E^{-1+\e}}$.

Along the f\/irst sub-contour, after the change of variable $\rho=E^{1-\gamma}(\nu-\al)$, the integrand becomes
	\begin{gather*}
\int_{s}^{E^{1-\gamma} (1-\frac\al2)} \sum_{k \geq 2} E^{-\gamma(k-1)} c_k \rho^{-k}\big(1-\big(E^{\gamma-1}\rho+\al\big)^2\big)^{-k+\frac12} {\rm d}\rho .
	\end{gather*}
Estimating $(1-(E^{\gamma-1}\rho+\al)^2)^{-k+\frac12}$ with its supremum, the radius of convergence of the series is easily estimated to be at least $E^{\gamma}(1-\al^2)s r_0$. Therefore, for $E$ large enough, the integrand is seen to converge uniformly and the dominant contribution is $\mathcal{O}(E^{-\gamma})$.
	
	Similar considerations, along the second sub-contour, lead to a dominant contribution of order $\mathcal{O}(E^{-\frac{1+\e}2})$.
\end{proof}

\begin{proof}[Proof of Lemma \ref{lem:errors}]
We prove the lemma for $\rho_1$. An analogue proof is valid for the case~$\rho_{-1}$.
	
For any potential $V$, we def\/ine the function $R(V,\la)=\frac{4V''(\la)V(\la)-5V'(\la)^2}{V(\la)^\frac52}$.
	
By def\/inition $\rho_{1}\leq \int_{\ell} | R(V_A,\mu) {\rm d}\mu|$ where the oriented integration contour $\ell$ is such that
	along it the real part of the phase function, $\Re{S_A}$, is not decreasing and tends to $+\infty$
	
	To this aim, we f\/ix some $0<r< 1-\Re \al$ and def\/ine the following integration contour
	\begin{gather*}
	\int_{\la,\gamma_S}^{\la_r} |R(V_A,\mu)| {\rm d}\mu +\int_{\la_r,\gamma_A}^{\infty}|R(V_A,\mu)| {\rm d}\mu,
	\end{gather*}
	where $\la_r$ is the intercept with real part $E^{\frac12}r$ of
	the Stokes line $\gamma_S$ def\/ined by the equation $\Re S(\mu)=\Re S(\la)$. The second integral in taken along the anti-Stokes
	line $\gamma_A$, def\/ined by the equation
	$\Im S(\mu)=\Im S(\la_r)$ and the inequality $\Re S(\mu)\geq \Re S(\la_r)$; see Fig.~\ref{fig:stokescomplex}.
	
	We f\/irst analyse the integral along the anti-Stokes line $\gamma_A$. By our choice of the endpoint $\la_r$,
	along $\gamma_A$ we have that $|\mu|\geq C E^{\frac12}$ for some $C>0$.
	We make the change of variable $\nu=E^{-\frac12}\mu$ and expand the integrand around $E=\infty$ to get
	\begin{gather*}
	\int_{\la_r}^{\infty} |R(V_A,\mu)| {\rm d}\mu = E^{-1}\int_{E^{-\frac12}\la_r}^{\infty} |R(V_0,\nu)| {\rm d}\nu\big(1+\mathcal{O}\big(E^{-1}\big)\big),
	\end{gather*}
	where $V_0(\nu)=\nu^2+2\al \nu-(1-\al^2)$. Since $E^{-\frac12}\la_r$ has a well-def\/ined limit dif\/ferent from zero and
	$R(V_0,\nu)$ is integrable at $\infty$, we conclude that the integral along $\gamma_A$ is $\mathcal{O}(E^{-1})$.
	
	We tackle now the contribution from the integration along the Stokes line.
	We split the integration contour in two: We f\/ix $r'>0$ and write
	\begin{gather*}
	\int_{\la,\gamma_S}^{\la'_r} |R(V_I,\mu)| {\rm d}\mu + \int_{\la_1,\gamma_S}^{\la'_r} |R(V_I,\mu)| {\rm d}\mu,
	\end{gather*}
	where $\la'_r$ is the intercept with real part $E^{\frac16}r'$ of the Stokes line $\gamma_S$.
	
A detailed balance analysis shows that for any f\/ixed positive $s$, $s'$, 	in the domain $s E^{-\frac12+\gamma}\leq|\la|\leq E^{\frac16} s'$, $|R(V_A,E)|\leq C E^{-1}|\la|^{-3}$ for $E$ large enough (depending on $s$, $s'$). This means that that the f\/irst integral is $\mathcal{O}(E^{-2\gamma})$.
	
Similarly, for any f\/ixed positive $t$, $t'$, in the domain $ E^{\frac16} t \leq|\Re \la |\leq E^{\frac12} t'$, $|R(V_A,E)|\leq C' |\la| E^{-\frac32}$ for $E$ large enough (depending on $r$, $r'$). This means that the second integral is~$\mathcal{O}(E^{-1})$.
\end{proof}

\subsection*{Acknowledgements}
D.M.~is an FCT Researcher supported by the FCT Investigator Grant IF/00069/2015.
D.M.~is also partially supported by the FCT Research Project PTDC/MAT-STA/0975/2014.
The present work began in December 2015 while D.M.~was a Visiting Scholar at the University
of Sydney funded by the ARC Discovery Project DP130100967.
D.M.~wishes to thank the Department of Mathematics and Statistics of the University of
Sydney and the Centro di Ricerca Matematica Ennio De Giorgi in Pisa for the kind
hospitality.

P.R.~is a research associate at the University of Sydney, supported by Nalini Joshi's ARC Laureate Fellowship Project FL120100094.
P.R.~would like to extend his gratitude to the department of mathematics of the University of Lisbon
and the Centro di Ricerca Matematica Ennio De Giorgi in Pisa, where a major part of this collaboration took place.
P.R.~was also supported by an IPRS scholarship at the University of
Sydney.

We are deeply indebted to Nalini Joshi for her continuous scientif\/ic and material support. We also thank Alexandre Eremenko, Davide Guzzetti, Peter Miller and Walter Van Assche for discussions about the present topic of investigation at various stages of this work.

We also acknowledge the anonymous referees for helping us improving the paper.

\pdfbookmark[1]{References}{ref}
\LastPageEnding

\end{document}